\pdfoutput=1

\documentclass[11pt]{article}
\usepackage{fullpage}
\usepackage{natbib}
\usepackage{blindtext}
\usepackage{hyperref}
\hypersetup{
	colorlinks=true,
	linkcolor=blue!70!black,
	citecolor=blue!70!black,
	urlcolor=blue!70!black
}
\usepackage{caption}

\usepackage{tikz}
\usetikzlibrary{positioning,calc}
\usepackage[T1]{fontenc}
\usepackage{parskip}

\usepackage{amsmath}
\usepackage{amssymb}
\usepackage{amsthm}
\usepackage{booktabs}
\usepackage{accents}
\usepackage{algorithm}
\usepackage{bbm}
\usepackage{bbold}
\usepackage{bm}
\usepackage{graphicx}
\graphicspath{{./figs/}}

\usepackage{cleveref}
\usepackage{thm-restate}
\usepackage{algorithm}
\usepackage{algpseudocode}

\newtheorem{theorem}{Theorem}
\newtheorem{lemma}{Lemma}

\newtheorem{definition}{Definition}

\newtheorem{proposition}{Proposition}

\newtheorem{assumption}{Assumption}
\newtheorem{remark}{Remark}
\newcommand{\defeq}{:=}
\newcommand{\norm}[1]{\left\lVert#1\right\rVert}
\newcommand{\vertiii}[1]{{\left\vert\kern-0.25ex\left\vert\kern-0.25ex\left\vert #1 
    \right\vert\kern-0.25ex\right\vert\kern-0.25ex\right\vert}}

\newcommand{\normop}[1]{\left\lVert#1\right\rVert_{\textup{op}}}
\newcommand{\normf}[1]{\left\lVert#1\right\rVert_{\textup{F}}}

\newcommand{\eps}{\epsilon}

\newcommand{\iid}{\textup{i.i.d.\ }} 
\newcommand{\R}{\mathbb{R}}

\newcommand{\bas}[1]{\begin{align*}#1\end{align*}}
\newcommand{\ba}[1]{\begin{align}#1\end{align}}
\newcommand{\bbb}[1]{\left[#1\right]}
\newcommand{\diag}[1]{{\textup{diag}}\left(#1\right)}

\newcommand{\bk}{\color{black}}
\newcommand{\rd}{\color{red}}

\newcommand{\E}{\mathbb{E}}
\newcommand{\Var}{\textup{Var}}
\newcommand{\cov}{\textup{Cov}}

\newcommand{\rem}{R}

\newcommand{\id}{\mathbf{I}}
\newcommand{\bb}[1]{\left(#1\right)}

\newcommand{\V}{\mathbb{V}}
\newcommand{\tr}[1]
{\text{trace}\left(#1\right)}

\usepackage{xcolor}
\definecolor{burntorange}{rgb}{0.8, 0.33, 0.0}

\newcommand{\Prob}{\mathbb{P}}

\newcommand{\Abs}[1]{\left|#1\right|}

\newcommand{\poly}{\textup{poly}}

\newcommand{\ind}{\mathcal{I}}

\newcommand{\Oja}{\mathsf{Oja}}
\newcommand{\Ojamain}{\widetilde{\mathsf{Oja}}}

\newcommand{\vp}{V_{\perp}}
\newcommand{\lambp}{\Lambda_{\perp}}

\newcommand{\ojavarest}{\mathsf{OjaVarEst}}
\newcommand{\var}{\mathsf{Var}}
\newcommand{\voja}{v_\mathsf{oja}}
\newcommand{\roja}{r_\mathsf{oja}}
\newcommand{\troja}{\tilde{r}_\mathsf{oja}}
\newcommand{\vmain}{\tilde{v}}

\newcommand{\median}{\mathsf{Median}}
\newcommand{\Res}{\mathsf{Res}}
\newcommand{\sign}{\mathsf{sign}}
\newcommand{\Tr}{\mathsf{Tr}}

\newcommand{\Ezero}[1]{\Psi_{#1, 0}}
\newcommand{\Eone}[1]{\Psi_{#1, 1}}
\newcommand{\Etwo}[1]{\Psi_{#1, 2}}
\newcommand{\Ethree}[1]{\Psi_{#1, 3}}
\newcommand{\Efour}[1]{\Psi_{#1, 4}}

\newcommand{\Mone}{\mathcal{M}}
\newcommand{\Mtwo}{\mathcal{M}_{2}}
\newcommand{\Mfour}{\mathcal{M}_{4}}
\newcommand{\Nu}{\mathcal{V}}

\newcommand{\eigengap}{\lambda_{1}-\lambda_{2}}

\newcommand{\errorsmall}{e_{\textsf{small}}}
\newcommand{\bootstrap}{\mathsf{BootstrapOja}}
\title{Beyond Sin-Squared Error: Linear-Time Entrywise Uncertainty Quantification for Streaming PCA}

\author{Syamantak Kumar\thanks{University of Texas at Austin, \texttt{syamantak@utexas.edu}} 
\and Shourya Pandey \thanks{University of Texas at Austin, \texttt{shouryap@utexas.edu}}
\and Purnamrita Sarkar\thanks{University of Texas at Austin, \texttt{purna.sarkar@utexas.edu}}}

\begin{document}

\maketitle


\begin{abstract}
    We propose a novel statistical inference framework for streaming principal component analysis (PCA) using Oja’s algorithm, enabling the construction of confidence intervals for individual entries of the estimated eigenvector. Most existing works on streaming PCA focus on providing sharp sin-squared error guarantees. Recently, there has been some interest in uncertainty quantification for the sin-squared error. However, uncertainty quantification or sharp error guarantees for \textit{entries of the estimated eigenvector} in the streaming setting remains largely unexplored. We derive a sharp Bernstein-type concentration bound for elements of the estimated vector matching the optimal error rate up to logarithmic factors. We also establish a Central Limit Theorem for a suitably centered and scaled subset of the entries.  To efficiently estimate the coordinate-wise variance, we introduce a provably consistent subsampling algorithm that leverages the median-of-means approach, empirically achieving similar accuracy to multiplier bootstrap methods while being significantly more computationally efficient. Numerical experiments demonstrate its effectiveness in providing reliable uncertainty estimates with a fraction of the computational cost of existing methods.
\end{abstract}


\section{Introduction}
Principal Component Analysis (PCA)~\citep{pearson1901liii, ziegel2003principal} is a cornerstone for statistical data analysis and visualization. Given a dataset $\{X_i\}_{i=1}^{n}$, where each $X_i \in \mathbb{R}^d$ is independently drawn from a distribution $\mathcal{P}$ with mean zero and covariance matrix $\Sigma$, PCA computes the eigenvector $v_1$ of $\Sigma$ that corresponds to the largest eigenvalue $\lambda_1$, and is the direction that explains the most variance in the data. It has been established~\citep{wedin1972perturbation, jain2016streaming, vershynin2010introduction} that the leading eigenvector $\hat{v}$ of the empirical covariance matrix $\hat{\Sigma} = \frac{1}{n} \sum_{i=1}^n X_iX_i^{\top}$ is a nearly optimal estimator of $v_1$ under suitable assumptions on  the data distribution. 

While theoretically appealing, computing the empirical covariance matrix $\hat{\Sigma}$ explicitly requires $O(d^2)$ time and space, which is expensive in high-dimensional settings when both the sample size and the dimension are large. Oja’s algorithm~\citep{oja1985stochastic}--- a streaming algorithm inspired by Hebbian learning~\citep{hebb2005organization}--- has emerged as an efficient and scalable algorithm for PCA. It maintains a running estimate of $v_1$ similar to a projected stochastic gradient descent (SGD) update
\begin{gather}\label{eq:ojaupdate}
u_i \gets u_{i-1} +\eta_n X_i(X_i^T u_{i-1}), \;\;
u_i \gets \frac{u_i}{\norm{u_i}_{2}}
\end{gather}
for $i \in [n]$, where $u_0$ is a random unit vector and $\eta_{n} > 0$ is the learning rate. The algorithm is single-pass, runs in time $\mathcal{O}(nd)$, and takes only $\mathcal{O}(d)$ space. We call the output $u_n$ of the above algorithm an \textit{Oja vector }$\voja$.

Oja's algorithm has fueled significant research in theoretical statistics, applied mathematics, and computer science~\citep{jain2016streaming, allenzhu2017efficient, chen2018dimensionality, yang2018history, henriksen2019adaoja, mouzakis2022spectral, lunde2021bootstrapping, monnez2022stochastic, DBLP:journals/corr/abs-2102-03646, kumarsarkar2024markovoja, kumarsarkar2024sparse}.  
Despite the plethora of work on sharp rates for the sin-squared error $\sin^2 \bb{\voja, v_1} := 1-(v_1^T\voja)^2$, entrywise uncertainty estimation for streaming PCA has received only limited attention. Since the update rule in Oja's algorithm is similar to a broad class of important non-convex problems, uncertainty estimation for Oja's algorithm has potential implications for matrix sensing~\citep{jain2013matcompl}, matrix completion~\citep{jain2013matcompl,keshavan2010completion}, subspace estimation~\citep{pmlr-v151-balzano22a}, and subspace tracking~\citep{balzano2010sstracking}. A notable exception is~\cite{lunde2021bootstrapping}, who show that $\sin^2 \bb{\voja, v_1} := 1-(v_1^T\voja)^2$ behaves asymptotically like a high-dimensional weighted chi-squared random variable. A main ingredient in their analysis is the Hoeffding decomposition of the matrix product $B_n$. Their method takes $O(bnd)$ time and $O(bd)$ space, where $b$ is the number of bootstrap replicas.
While \cite{lunde2021bootstrapping} do uncertainty estimation of the $\sin^2$ error, we are interested in coordinate-wise uncertainty estimation.

In contrast, in offline eigenvector analysis, there has been a surge of interest for \textit{two-to-infinity} ($\ell_{2\rightarrow \infty}$) error bounds for empirical eigenvectors and singular vectors of random matrices~\citep{eldridge2018unperturbed,Mao02102021,abbe2020entrywise,cape2017singular,abbe2022lptheory,cape2019signal}. However, none of these apply directly to the matrix product structure that arises from the Oja update in Eq~\eqref{eq:ojaupdate}. Recent advances on the concentration of matrix products~\citep{huang2022matrix,kathuria2020concentration} only provide operator norm or the $\ell_q$ moment of the Schatten norm of the deviation of a matrix product and do not provide non-trivial guarantees on the coordinates.

\textbf{Our contributions:} 

In this paper, we obtain \textit{finite sample and high probability deviation bounds} for elements of $\voja$.

1. We show that the deviation of the elements of $\voja$ is governed by a suitably defined limiting covariance matrix $\V$. Furthermore, for a subset $K$ of $[d]$ of interest, the distribution of the coordinate $\voja(k)$, when suitably centered and rescaled, is asymptotically normal with variance $\V_{kk}$.

2. We provide a sharp Bernstein-type concentration bound to show that \textit{uniformly over entries of $\voja$}, $\forall \; k \in [d],$
\ba{
   |e_k^{\top} (\underbrace{\voja -(v_1^T \voja) v_1}_{:= \roja})| = \tilde{O}\bb{\sqrt{\frac{\V_{kk}}{n}}}. \label{eq:per_coord_bound}
}
where $e_k$ denotes the $k^{\text{th}}$ standard basis vector. This is a surprising and sharp result because it can be used (see Lemma~\ref{lemma:entrywise_to_sin_squared}) to recover the optimal $\sin^2$ error up to logarithmic factors with high probability.

3. We provide an algorithm that couples a subsampling-based $O(nd)$ time and $O(d \log (d/\delta))$ space algorithm with Median of Means~\citep{nemirovskij1983problem} to estimate the marginal variances of the elements of $\roja:=\voja -(v_1^T \voja) v_1$. Theorem~\ref{thm:high_prob_error_bound} provides high-probability error bounds of our variance estimator \textit{uniformly} over $\forall k \in [d]$.

4. We present numerical experiments on synthetic and real-world data to show the empirical performance of our algorithm and also compare it to the multiplier bootstrap algorithm in~\cite{lunde2021bootstrapping} to show that our estimator achieves similar accuracy in significantly less time.

The paper is organized as follows: Section~\ref{ssec:related_work} discusses related work on streaming PCA, entrywise error bounds on eigenvectors, and statistical inference for Stochastic Gradient Descent. Section~\ref{sec:prelim} provides our problem setup, assumptions, and necessary preliminaries. Section~\ref{sec:main_results} provides our main results regarding entrywise concentration, CLT and our variance estimation algorithm, Algorithm~\ref{alg:variance_estimation}. We provide proof sketches in Section~\ref{sec:proof_techiniques} and experiments in Section~\ref{sec:experiments}.

\subsection{Related Work}\label{ssec:related_work}
\textbf{Streaming PCA.} A crucial measure of performance for Oja’s algorithm is the $\sin^2$ error, which quantifies the discrepancy between the estimated direction and the principal eigenvector of $\Sigma$ (the true population eigenvector, $v_1$) and the Oja vector, $\voja$. Notably, several studies~\citep{jain2016streaming, allenzhu2017efficient, DBLP:journals/corr/abs-2102-03646} have shown that Oja’s algorithm attains the same error as its offline counterpart, which computes the leading eigenvector of the empirical covariance matrix directly. More concretely, it has been shown that for an appropriately defined variance parameter $\Nu$ (equation~\eqref{eq:Nu_assumptions}),
\bas{
\sin^2(v_1, \voja) \defeq 1-(v_1^T\voja)^2=O\bb{\frac{\Nu}{n(\lambda_1-\lambda_2)^2}}.
}

\textbf{$\ell_\infty$ error bounds.} There is an extensive body of research on eigenvector perturbations of matrices. Most traditional bounds~\citep{davis1970rotation,wedin1972perturbation,stewart1990matrix}  measure error using the $\ell_2$ norm or other unitarily invariant norms. However, for machine learning and statistics applications, element-wise error bounds provide a better idea about the error in the estimated projection of \textit{a feature} in a given direction. This area has recently gained traction for random matrices.~\cite{eldridge2018unperturbed,abbe2020entrywise,cape2017singular,abbe2022lptheory} provide $\ell_{2\rightarrow \infty}$ bounds for eigenvectors and singular vectors of random matrices with low-rank structure.~\cite{cape2017singular} show an $\ell_{2\rightarrow \infty}$ norm for the error of the singular vectors of a covariance matrix formed by $n$ \iid Gaussian vectors; as long as $\lambda_1-\lambda_2>0$ and $v_1$ satisfies certain incoherence conditions, there exists a $w \in \{-1,1\}$ such that with probability $1-d^{-2}$, the top eigenvector $\hat{v}_1$ of the sample covariance matrix satisfies, up to logarithmic factors, 
\bas{
    \|v_1-w \hat{v}_1\|_\infty
    &\lesssim \bk
    \sqrt{\frac{\Tr\bb{\Sigma}/\lambda_{1}}{n}}\bb{\frac{\max_i\sqrt{\Sigma_{ii}}}{\sqrt{\lambda_1}}+\frac{\lambda_{2}}{\lambda_{1}}}  +  \frac{\Tr\bb{\Sigma}/\lambda_1}{n} \bb{\frac{1}{\sqrt{d}} + \sqrt{\frac{\lambda_{2}}{\lambda_{1}}}}. 
}
The guarantees of~\cite{cape2017singular} are offline and provide a common upper bound on all coordinates. Our algorithm has error guarantees that scale with the variances of the coordinates. 

\textbf{Uncertainty estimation for SGD.} 
For convex loss functions, the foundational work of~\cite{polyak1992sgd,ruppert1988efficient,SGD_bather1989stochastic} in Stochastic Gradient Descent (SGD) demonstrates that averaged SGD iterates are asymptotically Gaussian. A significant body of research has focused on the convex setting. These include notable works on covariance matrix estimation~\citep{SGD_conf/aaai/LiLKC18,su2018uncertainty,SGD_JMLR:v19:17-370,chen2020SGD,SGD_lee2022fast, zhu2023online}.
In comparison, work on uncertainty estimation for nonconvex loss functions is relatively few~\citep{yu2020nonconvexAnalysis,zhong2023online}. ~\citet{yu2020nonconvexAnalysis} establishes a Central Limit Theorem (CLT) under relaxations of strong convexity assumptions.~\citet{zhong2023online} weakens the conditions but relies on online multiplier bootstrap methods to estimate the asymptotic covariance matrix. Existing methods for estimating and storing the full covariance matrix suffer from numerical instability or slow convergence rates (see~\cite{pmlr-v206-chee23a}). For convex functions and their relaxations, ~\citet{zhu2024high,carter2025statistical} present computationally efficient uncertainty estimation approaches that are related but different from ours.
   
In large-scale, high-dimensional problems, maintaining numerous bootstrap replicas is computationally expensive.~\cite{pmlr-v206-chee23a} introduce a scalable method for confidence intervals around SGD iterates, which are informative yet conservative under regularity conditions such as strong convexity at the optima. In their setting, for an appropriate initial learning rate, the covariance matrix can be approximated by a constant multiple of identity (see also~\cite{ljung1992plusminusref}). In our setting, such an approximation requires knowledge of all eigenvalues and eigenvectors of $\Sigma$. The work most relevant to ours is by~\cite{lunde2021bootstrapping}. They provide asymptotic distributions for the $\sin^2$ error of the Oja vector and an online multiplier bootstrap algorithm to estimate the underlying distribution.

\textbf{Resampling Methods and Bootstrapping.}
Nonparametric bootstrap~\citep{efron1979bootstrap,hall1992bootstrap,efron1993introduction} is a resampling method where $b$ resamples of a given size $n$ dataset are drawn with replacement and treated as $b$ independent samples drawn from the underlying distribution. Of these varieties of bootstraps, the one widely used in SGD inference is the online multiplier bootstrap, where multiple bootstrap resamples are updated in a streaming manner by sampling multiplier random variables to emulate the inherent uncertainty in the data~\citep{ramprasad2023online, zhong2023online, lunde2021bootstrapping}.

A major concern about the bootstrap is its computational bottleneck. Maintaining many bootstrap replicates is computationally prohibitive if the number of data points $n$ and the dimension $d$ are large. Some computationally cheaper alternatives to bootstrap are subsampling~\citep{politis1999subsampling,Politis_10.1093/biomet/asad021, bertail1999subsampling, levina2017subsampling, chaudhuri2024differentially, chua2024scalable} and $m$-out-of-$n$ bootstrap~\citep{Bickel-m-out-of-n, bickel2008choice, sakov1998using, andrews2010asymptotic} both of which rely on drawing $o(n)$ with-replacement samples. These methods are used in~\cite{blbjrssb} to create $n$ with-replacement samples from smaller subsamples, but require multiple bootstrap replicates and are not directly applicable to the streaming setting.

\section{Problem Setup and Preliminaries}
\label{sec:prelim}


\textbf{Notation.} Let $[n]=\{1,\dots,n\}$ for all positive integers $n$. For a vector $v$, $\|v\| = \|v\|_2$ denotes its $\ell_2$ norm. For a matrix $A$, $\|A\| = \normop{A}$ is the operator norm, $\normf{A}$ is the Frobenius norm, and $\norm{A}_p$ is the Schatten $p$-norm of $A$, which is the $\ell_{p}$ norm of the vector of singular values of $A$. We define the \textit{two-to-infinity} norm $\norm{A}_{2\leftarrow \infty} := \sup_{\norm{x}_{2}=1}\norm{Ax}_{\infty}$. For a random matrix $M$ and $p,q \ge 1$, we define the norm $\vertiii{M}_{p,q} \defeq \E[\norm{M}_p^q]^{1/q}$. Let $I\in\R^{d\times d}$ be the identity matrix with $i^{\textsf{th}}$ column $e_i$. Define the inner product of matrices as $\langle A,B\rangle=\Tr(A^T B)$. We use $\widetilde{O}$ and $\widetilde{\Omega}$ for bounds up to logarithmic factors and use $a\lesssim b$ to mean $a\le Cb$ for some universal constant $C$. $\diag{a_1,\dots,a_d}$ denotes the diagonal matrix with entries $a_1,\dots,a_d$. For a vector $v\in\R^d$ and $S\subseteq [d]$ with $|S|=k$, $v[S]\in\R^k$ is the ``sub-vector'' of $v$ with its coordinates indexed by $S$.

\textbf{Data}. Let $\left\{X_i\right\}_{i\in[n]}$ be independent and identically distributed ($\iid$) mean-zero vectors sampled from the distribution $\mathcal{P}$ over $\mathbb{R}^{d}$ with covariance matrix $\Sigma := \E\bbb{X_{i}X_{i}^{T}}$. Let $A_i := X_iX_i^{\top}$. Let $v_{1}, v_{2},  \ldots, v_{d}$ 
denote the eigenvectors of $\Sigma$ with corresponding eigenvalues $\lambda_{1} > \lambda_{2} \geq \ldots \geq \lambda_{d}$. Let $\vp := \bbb{v_{2}, v_{3}, \ldots, v_{d}} \in \mathbb{R}^{d \times \bb{d-1}}$.


We operate under the following assumptions unless otherwise specified.

\begin{assumption}\label{assumption:bounded_moments}
For any $X_{i} \sim \mathcal{P}, A_i = X_iX_i^{\top}$, we assume the following moment bounds, where $\sqrt{\Nu} \le \mathcal{M}_2 \le \mathcal{M}_4$:
\begin{gather}
    \normop{\E\bbb{\bb{A_i - \Sigma}^{2}}} \leq \Nu \label{eq:Nu_assumptions}\\
    \E\bbb{\normop{A_i - \Sigma}^{2}}^{\frac{1}{2}} \leq \Mtwo \qquad \E\bbb{\normop{A_i - \Sigma}^{4}}^{\frac{1}{4}} \leq \Mfour. \label{eq:momentbound_assumptions}
\end{gather}
\end{assumption}

\begin{assumption}\label{assumption:sample_size}
There exists a universal constant $\kappa > 5$ such that $d = o\bb{n^{\kappa}}$ and  $\frac{n}{\log\bb{n}} \geq 2\max\left\{\kappa, \frac{\kappa^{2}\Mtwo^{4}\log\bb{d}}{\bb{\eigengap}^{4}} \right\}$. 
\end{assumption}

Assumption~\ref{assumption:bounded_moments} provides a suitable moment bound on the iterates $A_i$, and Assumption~\ref{assumption:sample_size} shows that we can handle the dimension $d$ growing polynomially with the sample size $n$, while requiring a mild base number of samples for convergence. We note that the constraint $\kappa > 5$ is arbitrary and our algorithm works as long as $d = \mathsf{poly}(n)$. These assumptions are commonly used in the streaming PCA literature (see for e.g. \cite{jain2016streaming}).

\textbf{Oja's Algorithm with constant learning rate.} With a constant learning rate, $\eta_{n}$, and initial vector, $u_{0}$, Oja's algorithm \citep{oja1982simplified} (denoted as $\Oja\bb{\left\{X_{t}\right\}_{t \in [n]}, \eta_n, u_{0}}$) performs the updates in Eq~\eqref{eq:ojaupdate}. Define $\forall t \in [n]$,
\ba{B_{t} := \prod_{i=0}^{t-1}\bb{I + \eta_{n} X_{t-i}X_{t-i}^{T}};\qquad 
B_{0} = I. 
\label{definition:Bn}
}
such that $u_{t} = B_{t}u_0/\norm{B_t u_0}_{2}$.
\section{Main Results}
\label{sec:main_results}

Recall the definition of Oja's algorithm with a constant learning rate, as defined in Section~\ref{sec:prelim}. For $\iid$ data $\mathcal{D}_{n} := \left\{X_{i}; X_{i} \in \R^{d}\right\}_{i \in [n]}$,  the learning rate $\eta_{n}$ defined in Lemma~\ref{lemma:learning_rate_choice}, and a random initial vector $u_0 \defeq g/\norm{g}$ where $g \sim \mathcal{N}(0, \id_d)$, define the \textit{Oja vector}
\ba{
 \voja(\mathcal{D}_{n}) := \Oja(\mathcal{D}_{n}, \eta_{n}, u_0). \label{eq:voja_def}
}
This is a random vector, with randomness over the data $\mathcal{D}_{n}$ as well as the initial vector $u_0$. While there are a myriad of works on the sin-squared error $1-(v_1^T\voja)^2$, there is, to our knowledge, no existing analysis on the concentration of the elements of the recovered vector around their population counterparts. One exception is~\citep{kumarsarkar2024sparse}, who showed that for sparse PCA, the elements of the Oja vector in the support of the true eigenvector are large, whereas those outside are small. However, these guarantees do not show concentration in our setting. We start our analysis with the Hoeffding decomposition of the matrix product (also see~\cite{lunde2021bootstrapping, vandervaart-asymptotic}).
The Hoeffding decomposition is a powerful tool that allows one to write the \textit{residual} of the Oja vector as
\ba{\label{eq:hoeffding}
\roja := \voja - \bb{v_1^{\top}\voja}v_1 = \Psi_{n,1} + \Res_n
}
where $\Psi_{n,1}$ is $\eta_n$ times a sum of independent but non-identically distributed random vectors and the residual $\Res_n$ is negligible compared to $\Psi_{n,1}$ (see Lemma~\ref{lemma:oja_error_decomposition} for details).

First, we show that the covariance matrix $\E[\Psi_{n,1}\Psi_{n,1}^T]$ of the dominant term in the residual converges to $\V$ when suitably scaled. Later, in Proposition~\ref{prop:main:clt} we will show that the distribution of the entries of $\roja$ is asymptotically normal with covariance matrix $\E[\Psi_{n,1}\Psi_{n,1}^T]/(\eta_n\bb{\eigengap})$.

\begin{lemma}[Asymptotic variance]\label{lemma:second_moment_matrix}
    Let 
    \bas{\widetilde{M} &:= \E\bbb{\vp^\top \bb{A_{1}-\Sigma}v_{1}v_{1}^\top\bb{A_{1}-\Sigma}\vp}, \\ d_k &:= 1-\bb{\frac{\lambda_1-\lambda_{k+1}}{1+\eta_n\lambda_1}} \eta_n. } 
    Then, the matrix $R^{(n)} \in \R^{(d-1) \times (d-1)}$ with entries 
    \bas{
         R^{(n)}_{k,l} &:= \frac{\widetilde{M}_{kl}}{(1+\eta_{n}\lambda_{1})^2} \bb{\frac{1 - \bb{d_kd_l}^{n}}{1 - d_kd_l}}, 
    }
    satisfies $\E\bbb{\Psi_{n,1}\Psi_{n,1}^{\top}} = \eta_n^2\vp R^{(n)}\vp^{\top}$. 
    
    Define the matrices $R_0 \in \R^{(d-1) \times (d-1)}$ and $ \mathbb{V} \in \R^{d \times d}$ as
    \ba{\label{eq:asympvar}
     (R_{0})_{k,l} \defeq \frac{\widetilde{M}_{k\ell}}{2\lambda_1-\lambda_{k+1}-\lambda_{\ell+1}}; \;\; \mathbb{V} \defeq \frac{1}{\eigengap}\vp R_0\vp^T. 
    }
    then,
    \ba{
    \norm{\frac{1}{\eta_n\bb{\eigengap}}\E[\Psi_{n,1}\Psi_{n,1}^T]-\mathbb{V}}_F\lesssim \frac{\eta_n \lambda_{1}\Mtwo^{2}}{\bb{\lambda_1-\lambda_2}^{2}}. \label{eq:variance_diff_bound}
    }
\end{lemma}

This shows that suitably scaled, $\E[\Psi_{n,1}\Psi_{n,1}^T]$ converges to the matrix $\mathbb{V}$. Note that the scaling factor $\eta_n \bb{\eigengap} = \frac{\alpha \log n}{n}$ is independent of model parameters for the choice of $\eta_{n}$ defined in Lemma~\ref{lemma:learning_rate_choice}.

The next result establishes a Central Limit Theorem (CLT) for the subset of elements in the residual vector $r_{\text{oja}}$ with sufficiently large limiting variance. 


\begin{proposition}[CLT for a suitable subset of entries]\label{prop:main:clt}
Let $\{X_i\}_{i=1}^n$ be independent mean-zero random vectors with covariance matrix $\Sigma$ such that $\mathbb{E}\bigl[\exp(v^\top X_1)\bigr]\le\exp\bigl(\tfrac{\sigma^2\,v^\top\Sigma\,v}{2}\bigr)$ for all $v\in\mathbb{R}^d$ and $\sigma > 0$ is some constant. 

For all $i \in [n]$, let
\[
H_i := \frac{\sign\bb{v_{1}^{\top}u_0}}{(1+\eta_n \lambda_1)}\vp\,\lambp^{\,n-i}\vp^\top\bigl(A_i-\Sigma\bigr)v_1,
\]
Let $b > 0$ be a constant, and let $J \subseteq [d]$ be the set of coordinates with $\V_{jj} \ge b$. Let $p \defeq |J|$. 

Let $Y_i\in\mathbb{R}^p$ be independent mean-zero Gaussian vectors with covariance matrix $$\mathbb{E}[Y_iY_i^\top]=\frac{n\eta_n}{\eigengap}\,\mathbb{E}[H_i[J]H_i[J]^\top],$$ and let $S_{Y} := \sum_{i=1}^{n}Y_i$. Suppose the learning rate $\eta_n$, set according to Lemma~\ref{lemma:learning_rate_choice}, satisfies
$\frac{\Mtwo^{2} \lambda_1 \eta_n}{\bb{\eigengap}^2}\lesssim b$. Then,
\bas{
& \sup_{A\in\mathcal{A}^{\text{re}}}\bigg|\Prob\Bigl(\frac{\roja[J]}{\sqrt{\bb{\eigengap}\eta_n}}\in A\Bigr)-\Prob\Bigl(\frac{S_Y}{\sqrt{n}}\in A\Bigr)\bigg| =\tilde{O}\bb{ \bb{\frac{\Mfour}{\eigengap }}^{1/3}n^{-1/6} + \bb{\frac{\Mtwo}{\eigengap}}^{1/2}n^{-1/8}},
}
where $\mathcal{A}^{\text{re}}$ is the collection of all hyperrectangles in $\mathbb{R}^p$, i.e, sets of the form $A=\{u\in\mathbb{R}^p : a_j\le u_j\le b_j\text{ for }j=1,\dots,p\}$ and each $a_j$ and $b_j$ belongs to $\mathbb{R}\cup\{-\infty,\infty\}$. Here,
$\tilde{O}$ hides logarithmic factors in $n$, $d$, and polynomial factors in $b$ and in model parameters $ \lambda_1,\eigengap, \Mtwo, \Mfour$. 
\end{proposition}

\begin{remark}
    Note that the first $n^{-1/6}$ term in the convergence rate arises from the high-dimensional CLT result by~\cite{ChernoCLT2015} applied to $\Psi_{n,1}$. 
    The main bottleneck is the $n^{-1/8}$ term, resulting from the higher-order terms of the Hoeffding decomposition ($\Res_n$ in equation~\ref{eq:hoeffding}). We note that the second term may be tightened by using better concentration bounds. 
    We point the reader to Proposition~\ref{prop:clt_appendix} in the Appendix for a complete statement and proof. 
\end{remark}

Proposition~\ref{prop:main:clt} establishes a Gaussian approximation of suitably scaled $\roja[J]$, where $J$ is a set of elements with large enough asymptotic variance. Our proof uses results from~\cite{chernozhukov2017detailed} on the Hájek projection~\eqref{eq:hoeffding} and bounds the effect of the remainder term by using Nazarov's Lemma~\citep{nazarov2003maximal} (Theorem~\ref{thm:Nazarov}). We use this to derive concentration bounds for all coordinates. The lower bound on the variance is crucial and comes from Nazarov's inequality. It is also a condition of the results in~\cite{chernozhukov2017detailed}. A simple observation here is that when $b_k$ is zero, i.e. $v_1(k)=1$, then $\V_{kk}=0$. Here, CLT may not hold since the Hájek projection is zero, and the perturbation arises from some of the smaller error terms in the error decomposition. 

\begin{theorem}\label{thm:main:entrywise_concentration_bound} Let the learning rate $\eta_n$ be set according to Lemma~\ref{lemma:learning_rate_choice}. Further, for $X_i \sim \mathcal{P}, A_i = X_iX_i^{\top}$, let $\normop{A_i - \Sigma} \leq \Mone$ almost surely. Then, for $b_k \defeq \norm{e_{k}^{\top}\vp}_{2}$, with probability at least $3/4$, uniformly for all $k \in [d]$, 
\bas{
    \frac{\Abs{e_k^{\top}\roja}}{\sqrt{\eta_n\bb{\eigengap}}} &\lesssim \sqrt{\mathbb{V}_{kk}\log\bb{d}} + C b_k\sqrt{\frac{\log n}{n}},
}
where $\mathbb{V}$ is defined in Eq~\ref{eq:asympvar}, and $C$ is a constant that depends on $\lambda_1, \eigengap, \Mtwo,$ and $\Mfour$.
\end{theorem}

\begin{remark}
The limiting marginal variances $\V_{kk}$ also appear in the finite-sample bound for the elements of the residual vector. Estimating these variances enables us to quantify the uncertainty associated with each component of $\hat{v}_1$, even when the sample size is finite.

\end{remark}
In Appendix~\ref{appendix:entrywise_error_bounds}, we provide a complete result with arbitrary failure probability  $\delta$ in Lemma~\ref{lemma:entrywise_concentration_bound}. The above guarantee can be boosted to a high probability one using geometric aggregation (see e.g. Alg. 3 in~\cite{kumarsarkar2024sparse}).

\subsection{Uncertainty estimation}
\label{ssec:uncertainty_estimation}
Proposition~\ref{prop:main:clt} shows that the asymptotic variance of elements of the residual $\roja(i)$ is governed by the variance of the entries $\E[(e_i^T\Psi_{n,1})^2]$ of $\Psi_{n,1}$. We cannot directly get to $\Psi_{n,1}$ since we only observe $\voja$. If we could estimate $\roja$, it would give us an idea of the error. However, we do not know $v_1$, and so cannot directly access $\roja$. We alleviate this difficulty by using the following high-accuracy estimate $\vmain$ of $v_1$ constructed using $N$ samples:
\ba{\label{eq:vtilde}
\vmain \gets \Ojamain(\mathcal{D}_{N}, u_0),} 
where $N$ satisfies the bounds of Theorem~\ref{thm:high_prob_error_bound}. The vector satisfies the bound
\bas{
    \sin^2 \bb{\vmain, v_1} \le \frac{C\log\bb{\frac{1}{\delta}}\log\bb{N/\log\bb{\frac{1}{\delta}}}\Mtwo^{2}}{N(\eigengap)^{2}}
}
with probability at least $1-\delta$. Such an estimator can be constructed by splitting the $N$ samples into $\Theta\bb{\log (1/\delta)}$ batches of roughly equal size, running Oja's algorithm on each batch, and aggregating them by geometric aggregation. 

Algorithm~\ref{alg:variance_estimation} takes as input the data $\{X_{i}\in \mathbb{R}^{d}\}_{i \in [n]}$, a failure probability $\delta$, and the proxy unit vector $\vmain$.
The $n$ samples are split into $m_1$ batches with $n/m_1$ samples each. Then, the ${\ell}^{\text{th}}$ batch of $n/m_1$ samples is further split into $m_2$ batches of size $B \defeq n/(m_1m_2)$ each. Oja vectors $\left\{\hat{v}_{\ell, j}\right\}_{j \in [m_2]}$ are computed on each of these $m_2$ batches, and the variance of the $k^{\mathsf{th}}$ coordinate is estimated as
\ba{
    \hat{\sigma}^{2}_{k, \ell} := \sum_{j \in [m_2]} \dfrac{\bb{e_k^{\top} \bb{\hat{v}_{\ell, j} - (\vmain^\top \hat{v}_{\ell, j})\vmain}}^2}{m_2}. \label{eq:def_sigma_hat_ell_main}
}


\begin{algorithm}[H]
\caption{$\ojavarest(\{X_{i}\in \mathbb{R}^{d}\}_{i \in [n]},  \delta, \vmain, \eigengap)$}
\label{alg:variance_estimation}
\begin{algorithmic}[1]
\State \textbf{Input:} Data $\mathcal{D}_n := \{X_{i}\in \mathbb{R}^{d}\}_{i \in [n]}$, failure probability $\delta\in (0,1)$, unit vector $\vmain$, eigengap $\eigengap$

\State \textbf{Output:} Estimates $\left \{ \hat{\gamma}_k \right \}_{k \in [d]}$ of $\{\V_{kk}\}_{k \in [d]}$
\State $m_{1} \gets 8\log(d/\delta), \;  m_{2} \gets \log n, \; B \gets n/(m_1m_2) $.
\For{$\ell \in [m_1]$}
    \For{$j \in [m_2]$}
        \State $\mathcal{D}_{\ell,j} \gets \left\{X_{B( m_2(\ell-1) + (j-1)) + t}\right\}_{t \in [B]}$
        \State $g \gets \mathcal{N}(0, I), \quad u \gets g/\norm{g}_2$
        \State $\hat{v}_{\ell,j} \gets \Oja\bb{\mathcal{D}_{\ell,j}, \eta_{B}, u_0}$
    \EndFor
    \For{$k \in [d]$}
        \State $\hat{\sigma}^{2}_{\ell, k} \gets \frac{\sum_{j \in [m_2]} \bb{e_k^\top\bb{\hat{v}_{\ell,j} - \bb{\vmain^\top \hat{v}_{\ell,j}} \vmain}}^2}{m_2}$
    \EndFor
\EndFor
\For{$k \in [d]$}
    \State $\hat{\gamma}_k \gets  \median\bb{\left\{\hat{\sigma}^{2}_{\ell, k}\right\}_{\ell \in [m_1]}}/\eta_B (\lambda_1 - \lambda_2)$
\EndFor
\State \Return $\left \{ \hat{\gamma}_k \right \}_{k \in [d]}$
\end{algorithmic}
\end{algorithm}


We will show that with a constant success probability, $\hat{\sigma}^{2}_{k, \ell}$ is close to the true variance of the corresponding coordinate. This is essentially the variance of a smaller dataset with scale $\eta_B$. To obtain a bound over all coordinates with an arbitrary failure probability, we take a median of the $m_1$ variances. 

For the final estimate of the diagonal elements $\V_{kk}$ of $\V$, the median is scaled by $\eta_B \bb{\eigengap}$. In Theorem~\ref{thm:high_prob_error_bound}, we show that $\hat{\gamma}_k$ concentrates around $\V_{kk}$ (see~\eqref{eq:main_error_bound_all}). For elements with large $\V_{kk}$, appropriate sample size $N$ and batch size $B$, Theorem~\ref{thm:high_prob_error_bound} also provides multiplicative error guarantees for the variance estimate (see~\eqref{eq:main_error_bound_some}).

\begin{remark}
We are using an estimate of $\E[(e_k^T\Psi_{n,1})^2]$ to provide the confidence interval around $\hat{v}_1(k)$. Algorithm~\ref{alg:variance_estimation} requires an estimate $\tilde{v}$ of $v_1$ for computing the estimates $\hat{\sigma}^{2}_{\ell, k}$ in Line 11, which is assumed to satisfy equation~\ref{eq:vtilde}. For large $N$, this error of approximating $v_1$ by $\vmain$ is small. In our experiments, we choose $N = n$ and obtain $\vmain$ by running the algorithm on the entire data.
\end{remark}


\begin{theorem}\label{thm:high_prob_error_bound}
Let $K$ be the set of indices in [d] that satisfy 
\begin{align}
    N &= \tilde{\Omega}\bb{B/c_k^2} ~~\text{ and} \label{eq:N_lower_bound_main}\\
    B &= \tilde{\Omega}\bb{\bb{\frac{b_k}{c_k}}^{2}\bb{\frac{\Mtwo}{\lambda_1 - \lambda_2}}^{2}+ \bb{\frac{b_k}{c_k}}^4 \bb{\frac{\Mfour}{\Mtwo}}^4 + \frac{\lambda_1}{c_k^2 \bb{\eigengap}}}, \label{eq:B_lower_bound_main}
\end{align}
where $b_k := \norm{e_k^{\top}\vp}$, $c_k := \sqrt{\tfrac{\E\bbb{\bb{e_k^{\top}\Eone{B}}^{2}}}{\eta_B}\tfrac{\eigengap}{\Mtwo^{2}}}$, and $B, N$ are respectively the batch size and the number of samples used for the proxy estimate $\tilde{v}$ in Algorithm~\ref{alg:variance_estimation}. 

Then, with probability at least $1-\delta$, the output $\left \{ \hat{\gamma}_k \right \}_{k \in [d]}$ of Algorithm~\ref{alg:variance_estimation} satisfies
\begin{gather}
\Abs{\hat{\gamma}_{k} - 
    \V_{kk}} \lesssim \frac{\V_{kk}}{\sqrt{m}} + \tilde{O}\bb{\frac{B}{N} + \frac{1}{B^{1/2}}}  ~~\forall k \in [d], \text{ and} \label{eq:main_error_bound_all} \\
\Abs{\hat{\gamma}_{k} - \V_{kk}} \lesssim \frac{\V_{kk}}{\sqrt{m}} ~~\forall k \in [K]. \label{eq:main_error_bound_some}
\end{gather}
\end{theorem}
\begin{remark}
    The output of Algorithm~\ref{alg:variance_estimation} rescales the median of the variances by  $\eta_B \bb{\eigengap} = \frac{\alpha \log B}{B}$. This is consistent with the entrywise concentration bounds in Theorem~\ref{thm:main:entrywise_concentration_bound} (which shows that the error in the $j^{th}$ entry is $\sqrt{\eta_n\bb{\eigengap} \V_{kk}}$, up to logarithmic terms) for a sufficiently large sample size and with Proposition~\ref{prop:main:clt} and Lemma~\ref{lemma:second_moment_matrix} (which show that the limiting variance of suitable entries of $\roja$ is $\eta_n\bb{\eigengap} \V_{kk}$). 
\end{remark}

\begin{remark}Theorem~\ref{thm:main:entrywise_concentration_bound} provides bounds about entries of the leading eigenvector. We believe our techniques can be generalized to provide uncertainty estimates for entries of top-$k$ eigenvectors using deflation-based approaches (see e.g \cite{pmlr-v247-jambulapati24a}).
\end{remark}

Equation~\eqref{eq:main_error_bound_all} holds for all $k \in [d]$. In the Appendix (see Remark~\ref{remark:prop2_higher_order}) we show that for the choice of $B$ and $N$ in Theorem~\ref{thm:high_prob_error_bound}, the higher order terms are indeed $o\bb{\frac{1}{\sqrt{m}}}$. Moreover, for any coordinate $k$ for which equations~\eqref{eq:N_lower_bound_main} and~\eqref{eq:B_lower_bound_main} hold, the lower order terms of equation~\eqref{eq:main_error_bound_all} are $O(\V_{kk}/\sqrt{m})$. This implies an $O(1/\sqrt{\log n})$-multiplicative guarantee on the error of $\hat{\gamma}_k$ like equation~\eqref{eq:main_error_bound_some}.





\section{Proof Techniques}\label{sec:proof_techiniques}


Let $\voja \sim \Oja\bb{\mathcal{D}_{\ell,j}, \eta_{n}, u_0}$ for uniform unit vector $u_0$ and $\vmain$ defined as in equation~\eqref{eq:vtilde}. 
Proposition A.1 in~\cite{lunde2021bootstrapping} shows that $B_n$, defined in \eqref{definition:Bn}, can be written as
\ba{\label{eq:hoeffding-main}
    B_{n} &= \sum_{k=0}^{n} T_{n,k},
}
where
\ba{
T_{n,k} &\defeq \sum_{S \subseteq [n], |S| = k}\ \ \prod_{i=1}^{n}M_{S,n+1-i}, \label{eq:tnk} \text{ and} \\
M_{S,i} &\defeq \begin{cases}
        \eta_n\bb{X_{i}X_{i}^{\top}-\Sigma} \; \text{ if } i \in S, \\
        I + \eta_n\Sigma \; \text{ if } i \notin S.
    \end{cases}
}

The term $T_{n,1}$ is called the Hájek projection of the random variable $B_n$ on the random variables $X_1, \dots, X_n$. $T_{n,1}$ is the best approximation to $B_n$ among the estimators that can be written as the sum of independent random vectors and satisfy certain integrability conditions. Moreover,
\begin{itemize}
    \item $T_{n,k}$ and $T_{n,j}$ are uncorrelated for all $k \neq j$, and
    \item the summands in $T_{n,k}$ are also pairwise uncorrelated.
\end{itemize}

We exploit this structure of the Hoeffding decomposition to decompose the residual vector $\troja$. 

\begin{restatable}{lemma}{ojaerrordecomposition}[Error Decomposition of $\voja$]\label{lemma:oja_error_decomposition} Let $\voja, \tilde{v}$ be defined as in \eqref{eq:voja_def} and \eqref{eq:vtilde} respectively. Then, 
\ba{\label{eq:ojadecomp}
    \voja - (\vmain^{\top} \voja) \vmain = \Ezero{n} + \Eone{n} + \Etwo{n} + \Ethree{n} + \Efour{n},
}
\text{ where}
\begin{align}\label{eq:ojadecomperror}
    &\Ezero{n} := (v_{1}^{\top}\voja)v_{1}-(\vmain^{\top} \voja) \vmain, \notag \\ &\Eone{n} := \frac{\vp\vp^{\top}T_{n,1}v_{1}\sign(v_{1}^{\top}u_{0})}{(1+\eta_n\lambda_{1})^{n}}, \notag \\ &\Etwo{n} := \frac{\vp\vp^{\top}(\sum_{k\geq 2}T_{n,k})v_{1}\sign(v_{1}^{\top}u_{0})}{(1+\eta_n\lambda_{1})^{n}}, \notag \\
    &\Ethree{n} := \vp\vp^{\top}B_{n}u_{0}\bb{\frac{1}{\norm{B_{n}u_0}_{2}} - \frac{1}{\Abs{v_{1}^{\top}u_{0}}(1+\eta\lambda_{1})^{n}}}, \notag\\
    &\Efour{n} :=  \frac{\vp\vp^{\top}B_{n}\vp\vp^{\top}u_{0}}{\Abs{v_{1}^{\top}u_{0}}(1+\eta\lambda_{1})^{n}}. 
\end{align}
\end{restatable}
We bound the variance of each of these terms separately. The dominating term $\Eone{n}$ corresponding to the Hájek projection $T_{n,1}$ has the largest variance. Recall from Lemma~\ref{lemma:second_moment_matrix} that
\bas{
\Abs{\E\bbb{\bb{e_{k}^{\top}\Eone{n}}^{2}}- \eta_n \lambda_1 \V_{kk}} \le \tilde{O}\bb{\frac{1}{n^2}}.
}
A finer analysis is needed for this term than the other residual terms in~\eqref{eq:ojadecomperror}. To do this, we bound the variance of $\bb{e_k^{\top} \Eone{n}}^2$. Lemma~\ref{lemma:en1_variance_bound} shows that $\sqrt{\Var((e_k^\top \Eone{n})^2)}$ is a constant factor within $\E[(e_k^\top \Eone{n})^2] = \tilde{O}(1/n)$ up to an additive error term $\tilde{O}(1/n^{3/2})$ which depends polynomially on model parameters. 
\begin{lemma}[Variance of the Hájek projection]\label{lemma:en1_variance_bound} Let $\Eone{n}$ be defined as in Lemma~\ref{lemma:oja_error_decomposition}. Then, 
\bas{
\sqrt{\Var\bb{(e_k^\top \Eone{n})^2}} \leq \sqrt{2} \E\bbb{\bb{e_k^\top \Eone{n}}^2} + \tilde{O}\bb{\frac{1}{n^{3/2}}}.
}
\end{lemma}
The three terms $\Etwo{n}, \Ethree{n},$ and $\Efour{n}$ are lower order terms.
\begin{lemma}[Bound on lower order terms]\label{lemma:en234_variance_bound} Let $\Etwo{n}$, $\Ethree{n}$, and $\Efour{n}$ be defined as in Lemma~\ref{lemma:oja_error_decomposition}. Then, 
\bas{
\E\bbb{\bb{e_k^\top \Etwo{n}}^2+\bb{e_k^\top \Ethree{n}}^2+\bb{e_k^\top \Efour{n}}^2} = \tilde{O}\bb{\frac{1}{n^{2}}}.
}
\end{lemma}
The bound on the error term $e_k^{\top} \Etwo{n}$  stems from a more general analysis of the terms $T_{n,k}$ in the Hoeffding decomposition of $B_n$. Lemma~\ref{lemma:higher-order-norm-main} is shown by exploiting the Martingale structure of $T_{n,k}$ and using norm inequalities~\citep{huang2022matrix} to compare the operator norm with the $\vertiii{.}_{p,q}$ norm.

\begin{lemma}
\label{lemma:higher-order-norm-main} Let $T_{n,k}$ be as defined in equation~\eqref{eq:tnk}. Let for any $2 \le q \le 4 \log d$,  $\mathcal{M}_q$ be defined such that $\E\bbb{\norm{A_i -\Sigma}^q}^{1/q} \le \mathcal{M}_q$ and $\eta_n \mathcal{M}_q \sqrt{n \log d} \lesssim 1$. Then, for any $j \in [n]$, $\delta \in (0,1)$, with probability at least $1-\delta$
\bas{
\norm{\sum_{k \ge j} T_{n,k}} \le 
\frac{3 (1+\eta_n \lambda_1)^n \bb{\eta_n \mathcal{M}_q \sqrt{4n \log d}}^j}{\delta^{\frac{1}{4\log d}}}
}
\end{lemma}

\begin{proof}[Proof sketch]
    Let $\mathcal{S}_{n,k}$ be the set of subsets of $[n]$ of size $k$. 
    \bas{
    T_{n,k} 
    &= (I+\eta_n \Sigma)T_{n-1,k} + \eta_n(A_n - \Sigma)T_{n-1,k-1}.
    }
    Proposition 4.3. of~\cite{huang2022matrix} implies
    \bas{
    \vertiii{T_{n,k}}_{p,q}^2 & \le \vertiii{(I+\eta_n \Sigma) T_{n-1,k}}_{p,q}^2 \\
    &\;\; + (p-1)\vertiii{\eta_n(A_n-\Sigma) T_{n-1,k-1}}_{p,q}^2.
    }
    as long as $\E\bbb{\eta_n(A_n - \Sigma)T_{n-1,k-1} | (I+\eta_n \Sigma)T_{n-1,k}} = 0$, which is true due to $A_1, A_2, \dots, A_n$ being mutually independent. Solving the recurrence shows the bound.
\end{proof}

The term $\Ezero{n}$ arises in the decomposition~\eqref{eq:ojadecomperror} because we use $\vmain$ as a proxy to $v_1$ in Algorithm~\ref{alg:variance_estimation}.
\begin{lemma}[Variance of Approximating $v_1$]\label{lemma:enzero_variance_bound} Let $\Ezero{n}$ be defined as in Lemma~\ref{lemma:oja_error_decomposition}. Then, $
\E\bbb{\bb{e_k^\top \Ezero{n}}^2} = \tilde{O}\bb{\frac{1}{N}},
$ where $\tilde{v}$ (Eq~\ref{eq:vtilde}) uses $N$
 samples.\end{lemma}
Theorem~\ref{thm:high_prob_error_bound} follows by combining all these bounds. See Appendix~\ref{sub_appendix:uncertainty} for a complete argument.

\section{Experiments}\label{sec:experiments}

In this section, we provide experiments on synthetic and real-world data to validate our theory. For all experiments, we estimate variance of the entries of $\roja$ (see Eq~\ref{eq:hoeffding}) by scaling the output of Algorithm~\ref{alg:variance_estimation} by $\eta_{B}\bb{\lambda_1-\lambda_2}$.

\subsection{Synthetic data experiments}
\label{sec:synthetic_experiments}

We provide numerical experiments to compare Algorithm~\ref{alg:variance_estimation} ($\ojavarest$) with the multiplier bootstrap based algorithm proposed in \cite{lunde2021bootstrapping}. As discussed in Section~\ref{ssec:uncertainty_estimation}, given a dataset $\mathcal{D}_n := \left\{X_i\right\}_{i \in [n]}$, we choose $\tilde{v}$ for $\ojavarest$ as $\tilde{v} := \Oja\bb{\mathcal{D}_n, \eta_n, z/\norm{z}_2}$ for $z \sim \mathcal{N}\bb{0, I}$ and set $m_{1} = 3$, $m_{2}$ = $\log\bb{n}$, $N = n$. Given a variance estimate, $\hat{\sigma}^{2}_{\ojavarest}$, we construct a $\bb{1-\alpha}$-confidence interval as $\tilde{v} \pm z_{\frac{\alpha}{2}}\hat{\sigma}_{\ojavarest}$. 

For the bootstrap algorithm, using Algorithm 1 in the aforementioned paper, we use $b$ bootstrap samples to generate estimates $v^{*(1)}, \cdots, v^{*(b)}$ and measure the empirical variance by computing the average squared residual with $\vmain$. Again, given a variance estimate, $\hat{\sigma}^{2}_{\bootstrap}$, we construct a $\bb{1-\alpha}$-confidence interval as $\vmain \pm z_{\frac{\alpha}{2}}\hat{\sigma}_{\bootstrap}$. 

We also use the data generation process proposed in \cite{lunde2021bootstrapping} for our experiments. Specifically, we begin by generating independent samples $Z_{ij} \sim \operatorname{Uniform}(-\sqrt{3},\sqrt{3})$ for indices $i \in [n]$ and $j \in [d]$. Next, we define a positive semidefinite matrix $K$ with entries $K_{ij} = \exp(-c\,|i-j|)$ using the constant $c = 0.01$. With this matrix, we construct a covariance matrix $\Sigma$ via $\Sigma_{ij} = K(i,j)\,\sigma_i\,\sigma_j$, where the scaling factors are specified by $\sigma_i = 5\,i^{-\beta}$ for $\beta \in \left\{0.2, 1\right\}$. We finally transform the samples as $X_i = \Sigma^{1/2} Z_i$.

\begin{figure}[!hbt]
    \centering
    \includegraphics[width=0.6\linewidth]{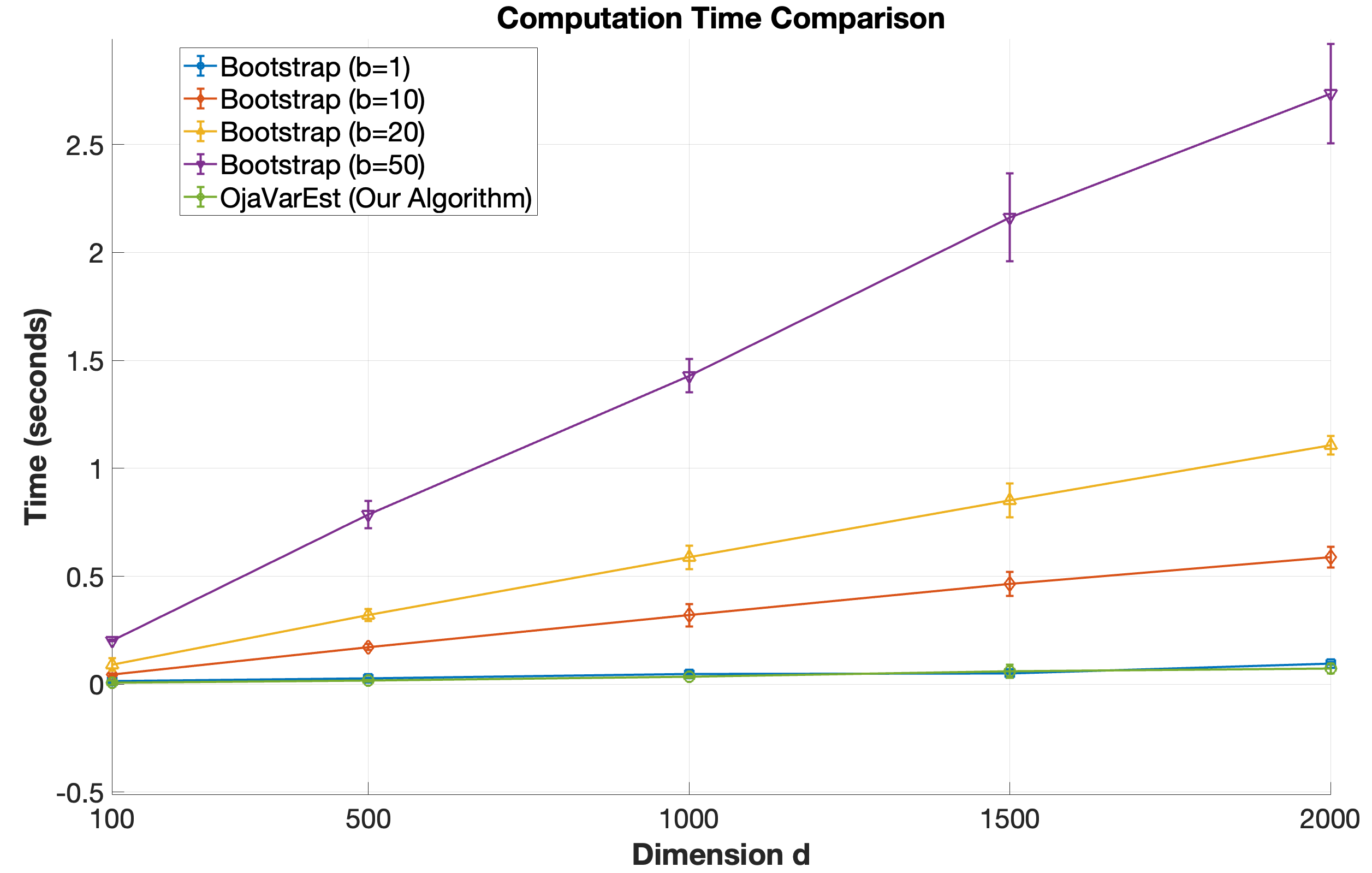} 
    \caption{Time taken by the bootstrap methods and the OjaVarEst algorithm. Experiments verify that our proposed algorithm is as fast as bootstrap with $b=1$.}
    \label{fig:computation_time}
\end{figure}

The first experiment (see Figure~\ref{fig:computation_time}) compares the computational performance of $\ojavarest$ with bootstrap to measure variance, varying the number of bootstrap samples, $b$, and recording performance for different values of $d$ with a fixed $n = 5000$ and $\beta = 1$. We note that the performance of our algorithm is computationally at par with bootstrap when using only 1 bootstrap sample, and is substantially better if the number of bootstrap samples increase. This is to be expected since for our algorithm, only two passes over the entire dataset suffice, whereas for bootstrap, $b$ bootstrap vectors are required to be maintained, which slows computation by a factor of $b$. Furthermore, it also requires $b$ times as much space to maintain $b$ different iterates, which may be costly in context of training large models.

\begin{table*}[!htb]
    \centering
    \renewcommand{\arraystretch}{1.1} 
    \resizebox{\textwidth}{!}{ 
    \begin{tabular}{l|cccc|cccc}
        \toprule
        & \multicolumn{4}{c|}{\textbf{Dist. 1} $(\beta = 1)$, Coordinate 1} 
        & \multicolumn{4}{c}{\textbf{Dist. 1} $(\beta = 1)$, Coordinate 2} \\
        \cmidrule(lr){2-5} \cmidrule(lr){6-9}
        ($n, d$) & $\ojavarest$ & BS ($b=1$) & BS ($b=10$) & BS ($b=20$) 
        & $\ojavarest$ & BS ($b=1$) & BS ($b=10$) & BS ($b=20$) \\
        \midrule
        2e3, 2e3 & 96.50\% & 65.00\% & 93.00\% & 95.00\% & 94.00\% & 69.50\% & 91.00\% & 91.50\% \\
        5e3, 2e3 & 95.50\% & 73.00\% & 91.50\% & 94.00\% & 95.50\% & 73.00\% & 89.00\% & 92.00\% \\
        1e4, 2e3 & 96.00\% & 69.00\% & 93.50\% & 94.50\% & 96.00\% & 71.50\% & 93.50\% & 96.00\% \\
        \midrule \midrule
        & \multicolumn{4}{c|}{\textbf{Dist. 2} $(\beta = 0.02)$, Coordinate 1} 
        & \multicolumn{4}{c}{\textbf{Dist. 2} $(\beta = 2)$, Coordinate 2} \\
        \cmidrule(lr){2-5} \cmidrule(lr){6-9}
        ($n, d$) & $\ojavarest$ & BS ($b=1$) & BS ($b=10$) & BS ($b=20$) 
        & $\ojavarest$ & BS ($b=1$) & BS ($b=10$) & BS ($b=20$) \\
        \midrule
        2e3, 2e3 & 94.50\% & 74.00\% & 87.00\% & 93.50\% & 94.00\% & 75.00\% & 86.50\% & 92.00\% \\
        5e3, 2e3 & 96.00\% & 71.00\% & 87.50\% & 92.00\% & 96.50\% & 72.50\% & 87.00\% & 93.00\% \\
        1e4, 2e3 & 94.00\% & 65.00\% & 95.00\% & 94.00\% & 94.50\% & 66.50\% & 94.50\% & 93.50\% \\
        \bottomrule
    \end{tabular}
    }
\caption{\label{tab:coverage_stats}Coverage statistics for our algorithm, $\ojavarest$, and the Bootstrap(BS) estimator, with varying bootstrap samples $(b = 1, 10, 20)$, data distributions ($\beta = 1, 0.02$) and sample sizes $(n = 2000, 5000, 10000)$ with a fixed dimension $d = 2000$.}
\end{table*}

The next experiment (Table~\ref{tab:coverage_stats}) compares the quality of the variance estimates of our algorithm, $\hat{\sigma}^{2}_{\ojavarest}$ with that of bootstrap $\hat{\sigma}^{2}_{\bootstrap}$ for different number of bootstrap samples, $b$, and distributions, $\beta$. We record the average coverage rate, which is the proportion of times the confidence interval provided by the algorithm contains the coordinate of the true eigenvector, for a target coverage probability of $95\%$ for the first two coordinates of the eigenvector. $\ojavarest$ performs similarly to Bootstrap with $b = 20$. However, as shown in Figure~\ref{fig:computation_time}, the bootstrap method is $20$ times slower. The time taken by bootstrap with $b=1$ is similar to $\ojavarest$ but has a significantly worse average coverage rate. 

Our final experiment compares the Algorithm~\ref{alg:variance_estimation} with $m_1 = 3$ to using just the mean ($m_1 = 1$). Even with the choice $m_1 = 3$, the uncertainty in variance estimation is reduced.

\begin{figure}[!hbt]
    \centering
    \begin{minipage}{0.6\columnwidth}
        \centering
        \includegraphics[width=\columnwidth]{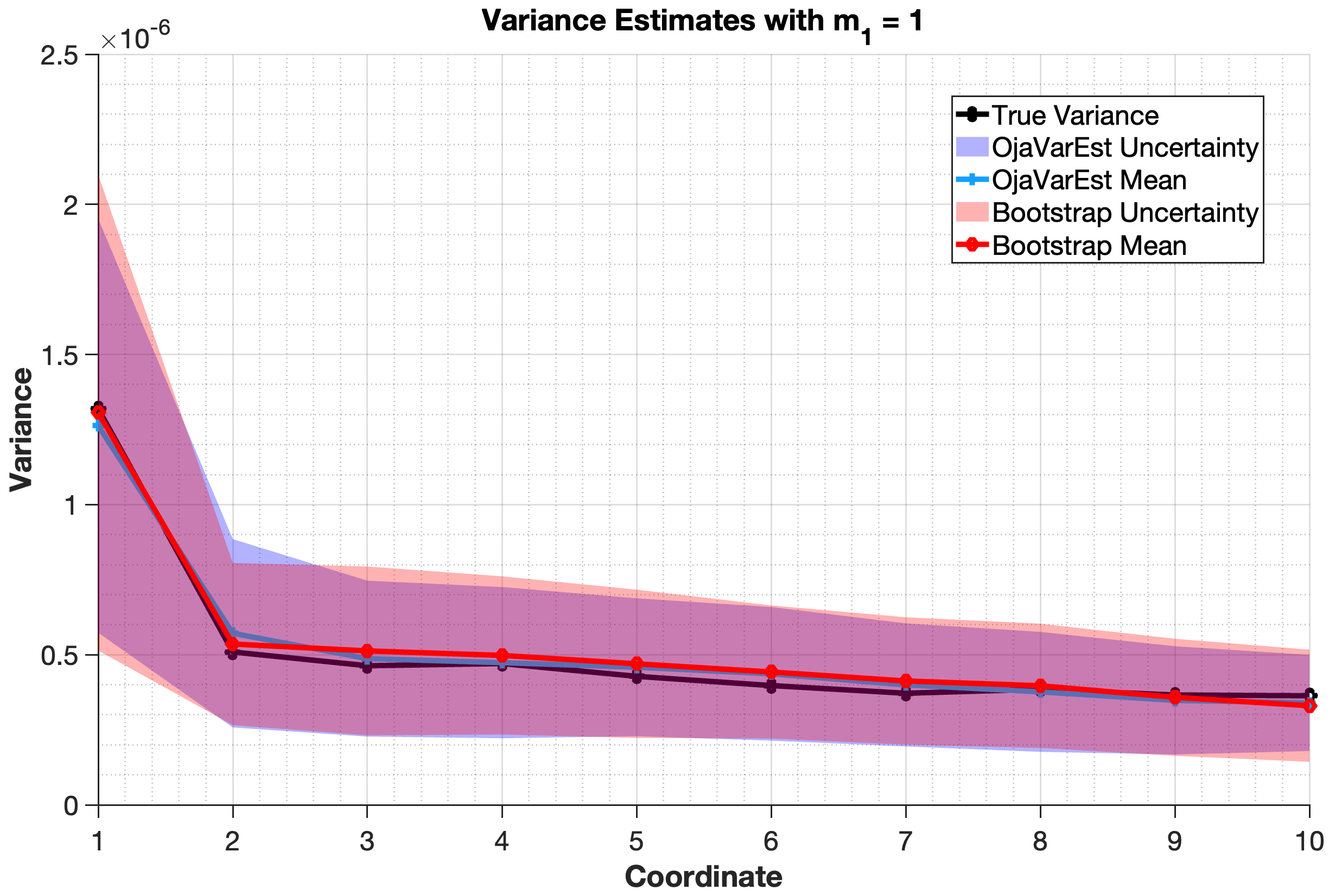}
        \captionsetup{labelformat=empty}
        \caption*{(a) Mean (with $m_1 = 1$)}
    \end{minipage}%
    \hfill
    \begin{minipage}{0.6\columnwidth}
        \centering
        \includegraphics[width=\columnwidth]{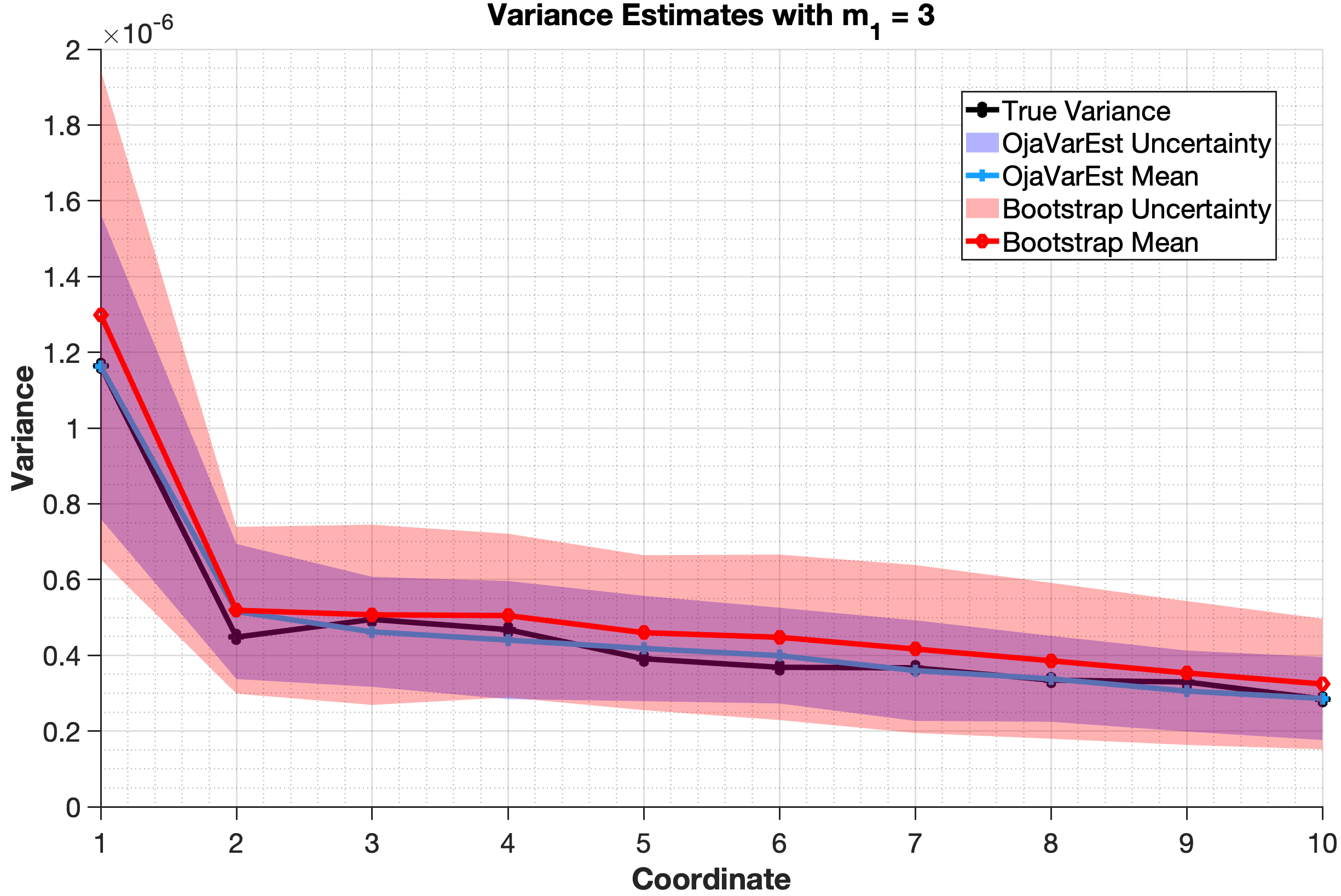}
        \captionsetup{labelformat=empty}
        \caption*{(b) Median (with $m_1 = 3$)}
    \end{minipage}
    \caption{Comparison of Median and Mean in Algorithm~\ref{alg:variance_estimation} for $n = 5000$, $d = 2000$, $\beta = 1$, $b = 10$.}
    \label{fig:mean_median_comparison}
\end{figure}

\subsection{Real-world data experiments}
\label{sec:real_world_experiments}

We provide experiments on two real-world datasets in this section. For each dataset, we show the 95\% confidence intervals and plot the top 20 coordinates of the true offline eigenvector (red dot), used as a proxy for the ground truth.

\textbf{Time series+missing data}: The Human Activity Recognition (HAR) Dataset \citep{anguita2013public} contains smartphone sensor readings from 30 subjects performing daily activities (walking, sitting, standing, etc.). Each data instance is a 2.56-second window of inertial sensor signals represented as a feature vector. Here, $n=7352$ and $d=561$. For each datum, we also replace 10\% of features randomly by zero to simulate missing data. Even in this setting, which we do not analyze theoretically, most of the top 20 coordinates of the offline eigenvector are inside the 95\% CI returned by our algorithm (see Figure~\ref{fig:har_dataset}).

\begin{figure}[!hbt]
    \centering
    \begin{minipage}{0.6\columnwidth}
        \centering
        \includegraphics[width=\columnwidth]{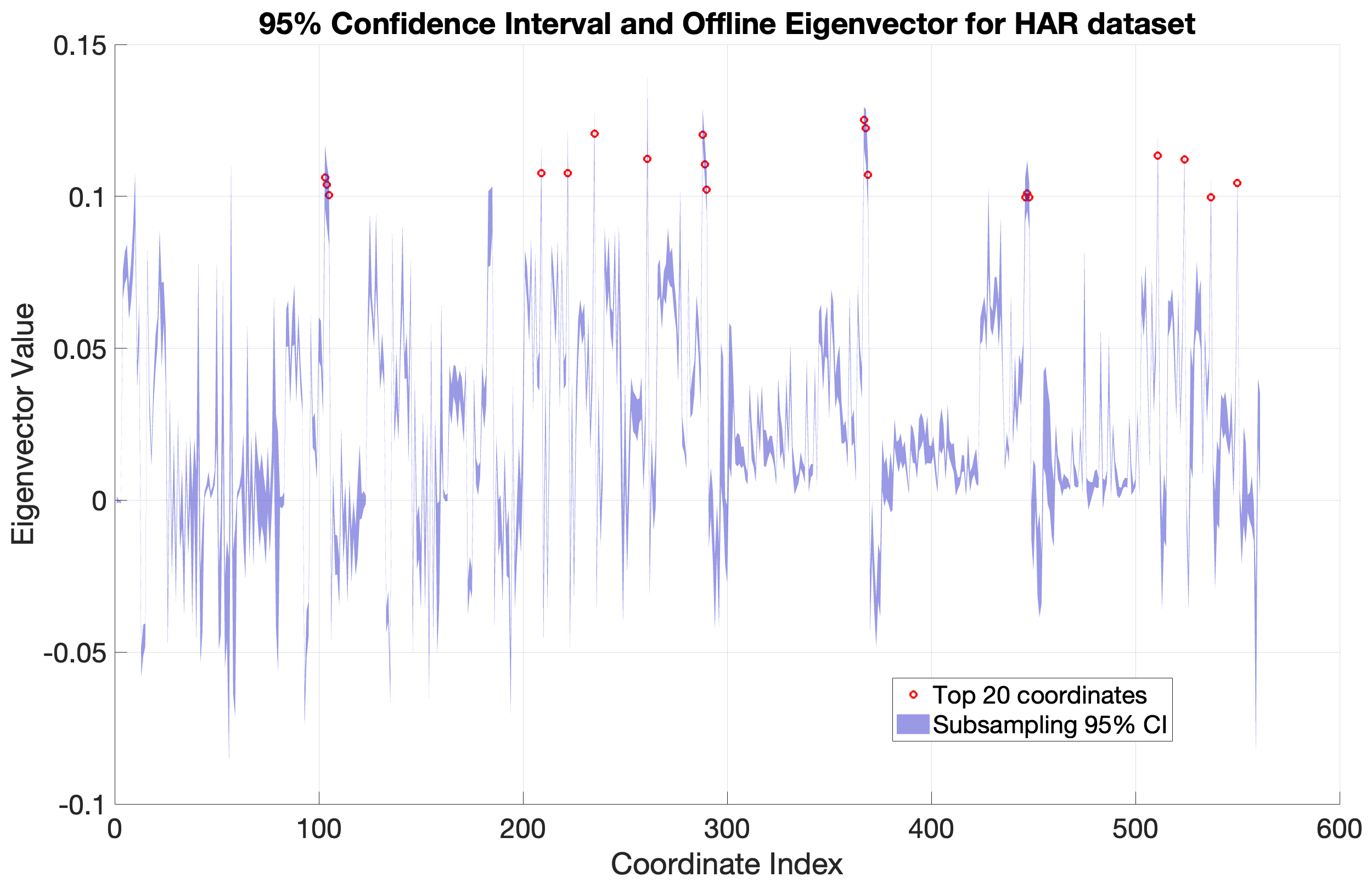}
        \captionsetup{labelformat=empty}
        \caption*{(a)}
    \end{minipage}%
    \hfill
    \begin{minipage}{0.6\columnwidth}
        \centering
        \includegraphics[width=\columnwidth]{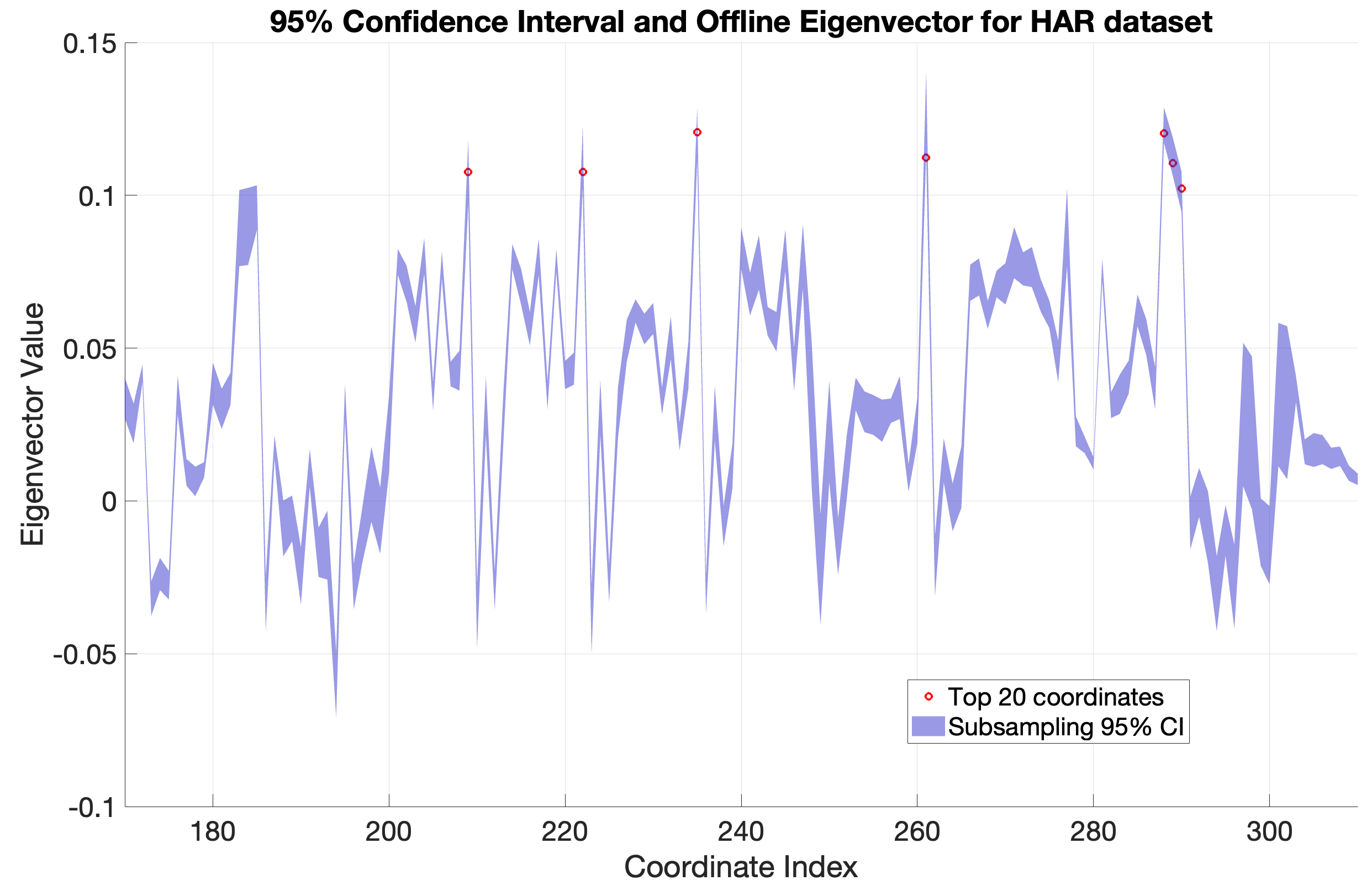}
        \captionsetup{labelformat=empty}
        \caption*{(b)}
    \end{minipage}
    \caption{Uncertainty Estimation for HAR dataset ($n = 7352, d = 561$). The sin2 error of Oja’s algorithm is equal to 0.057 for this dataset. (a) plot of the eigenvector with 95\% confidence interval for all coordinates and (b) the same plot zoomed in on indices 170-310 for exposition.}
    \label{fig:har_dataset}
\end{figure}

\begin{table*}[!hbt]
  \centering
  \begin{tabular}{c c c c c c c c c c c}
    \toprule
    Class & 0 & 1 & 2 & 3 & 4 & 5 & 6 & 7 & 8 & 9 \\
    \midrule
    $\sin^2$ error
      & 0.12 & 0.07 & 0.18 & 0.32 & 0.53 & 0.18 & 0.08 & 0.09 & 0.20 & 0.17 \\
    \bottomrule
  \end{tabular}
  \caption{  
    $\sin^2$ of the angle between the offline eigenvector and the subsampling
    eigenvector output by our algorithm, computed separately after filtering the
    MNIST data for each class.
  }
  \label{tab:mnist-sin2-error}
\end{table*}

\textbf{Image data}: We use the MNIST dataset~\citep{lecun1998gradient} of images of handwritten digits (0 through 9). Here, $n=60,000, d=784$, with each image normalized to a $28 \times 28$ pixel resolution. We see (Figure~\ref{fig:mnist_dataset}) that for the classes where Oja’s algorithm converges (small $\sin^2$ error in Table~\ref{tab:mnist-sin2-error}), most of the top 20 coordinates are inside their confidence intervals (CIs). Notable exceptions are classes 3 and 4, where several of the top 20 coordinates are not contained inside the corresponding CIs. This is expected because our theory is applicable when Oja’s algorithm converges.

\begin{figure}[!hbt]
    \centering
    \begin{minipage}{0.32\columnwidth}
        \centering
        \includegraphics[width=\columnwidth]{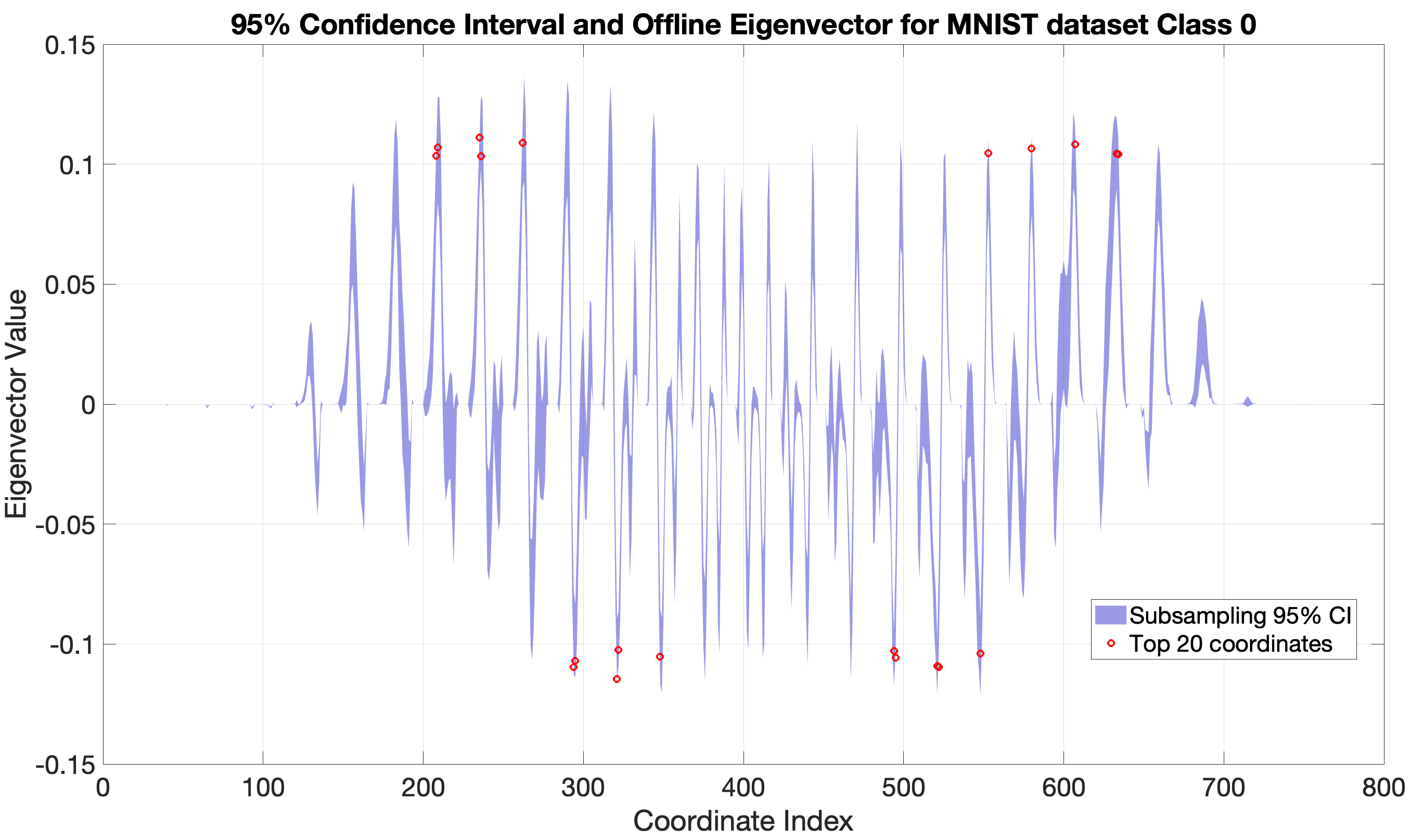}
        \captionsetup{labelformat=empty}
        \caption*{(0)}
    \end{minipage}%
    \hfill
    \begin{minipage}{0.32\columnwidth}
        \centering
        \includegraphics[width=\columnwidth]{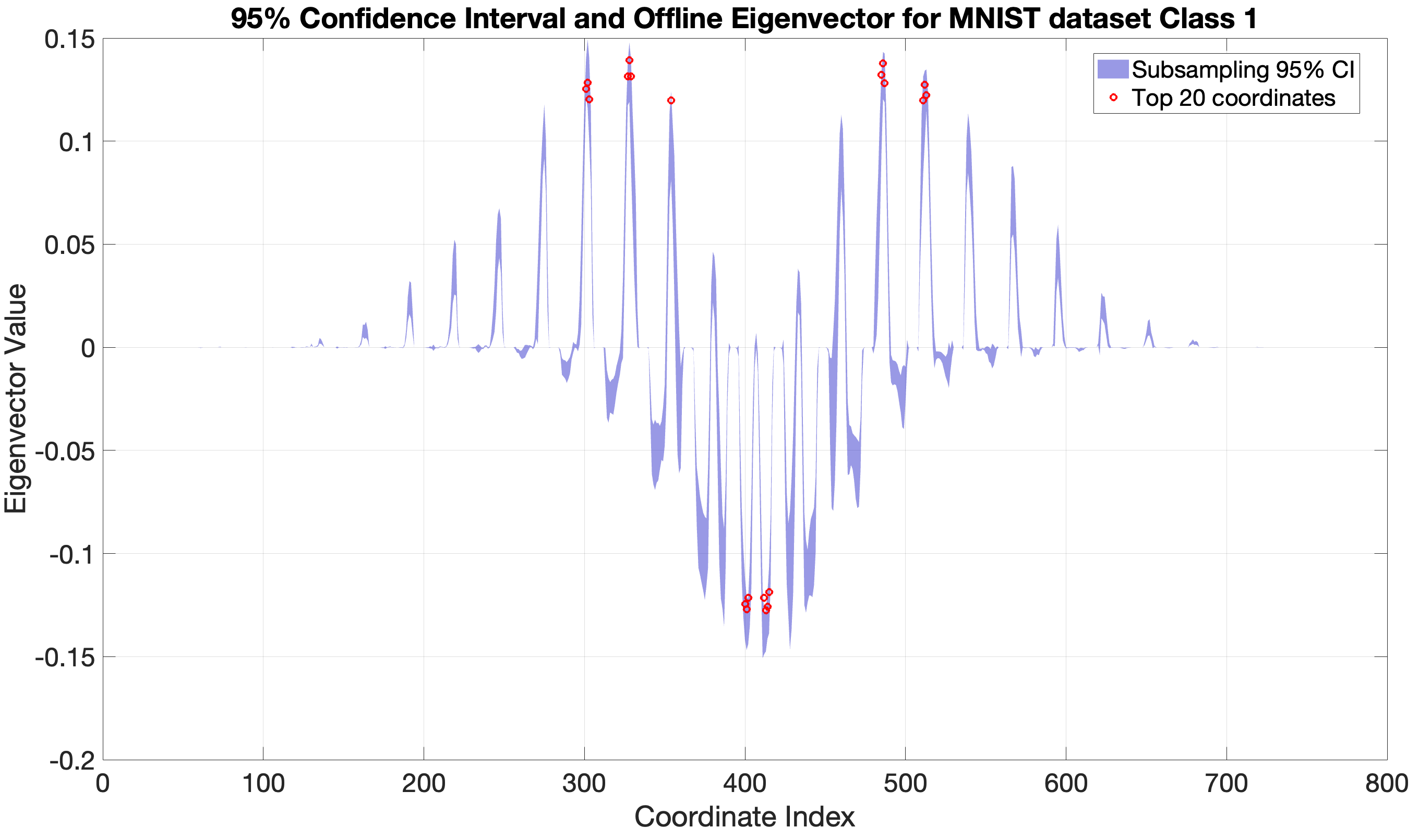}
        \captionsetup{labelformat=empty}
        \caption*{(1)}
    \end{minipage}
    \begin{minipage}{0.32\columnwidth}
        \centering
        \includegraphics[width=\columnwidth]{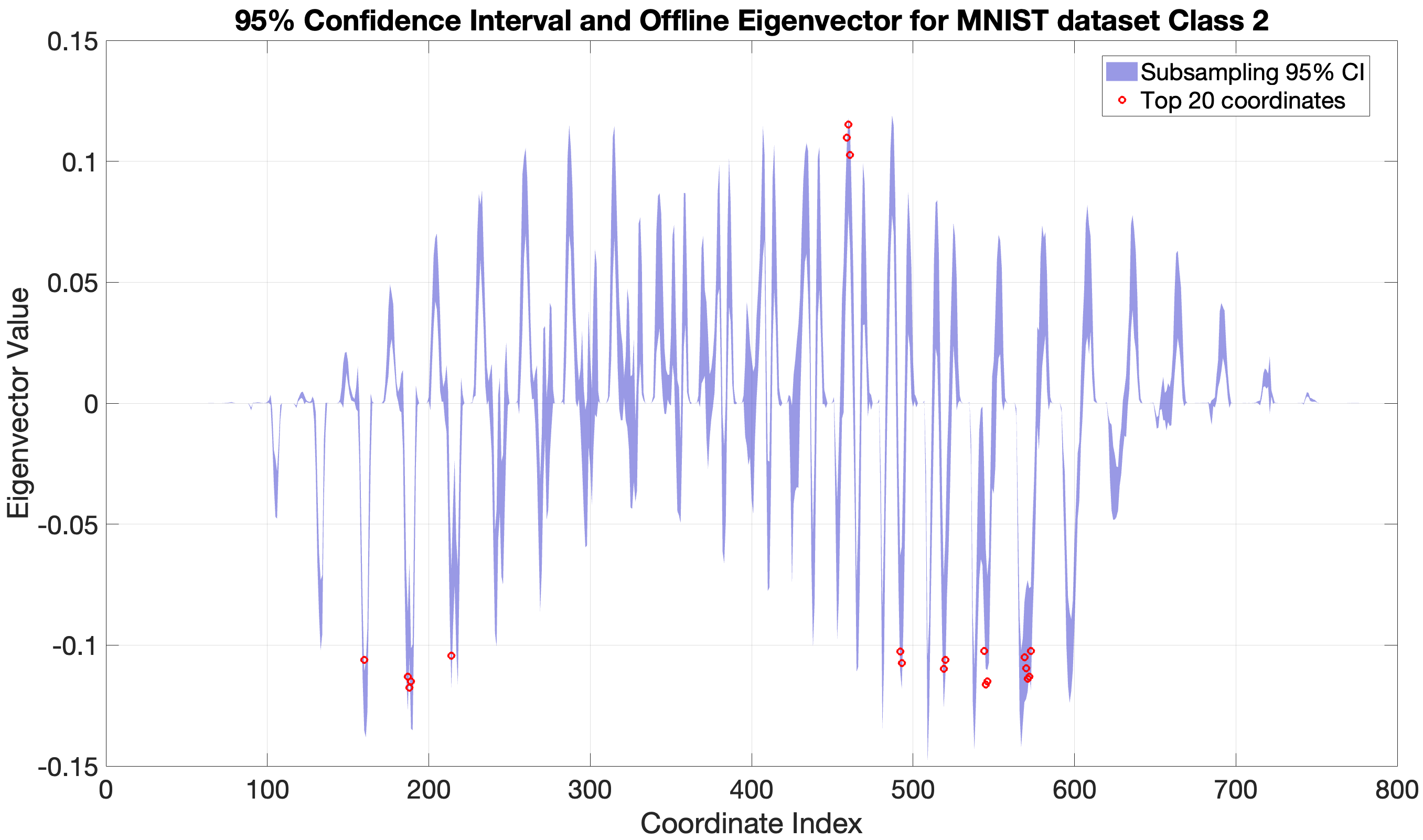}
        \captionsetup{labelformat=empty}
        \caption*{(2)}
    \end{minipage}%
    \hfill
    \begin{minipage}{0.32\columnwidth}
        \centering
        \includegraphics[width=\columnwidth]{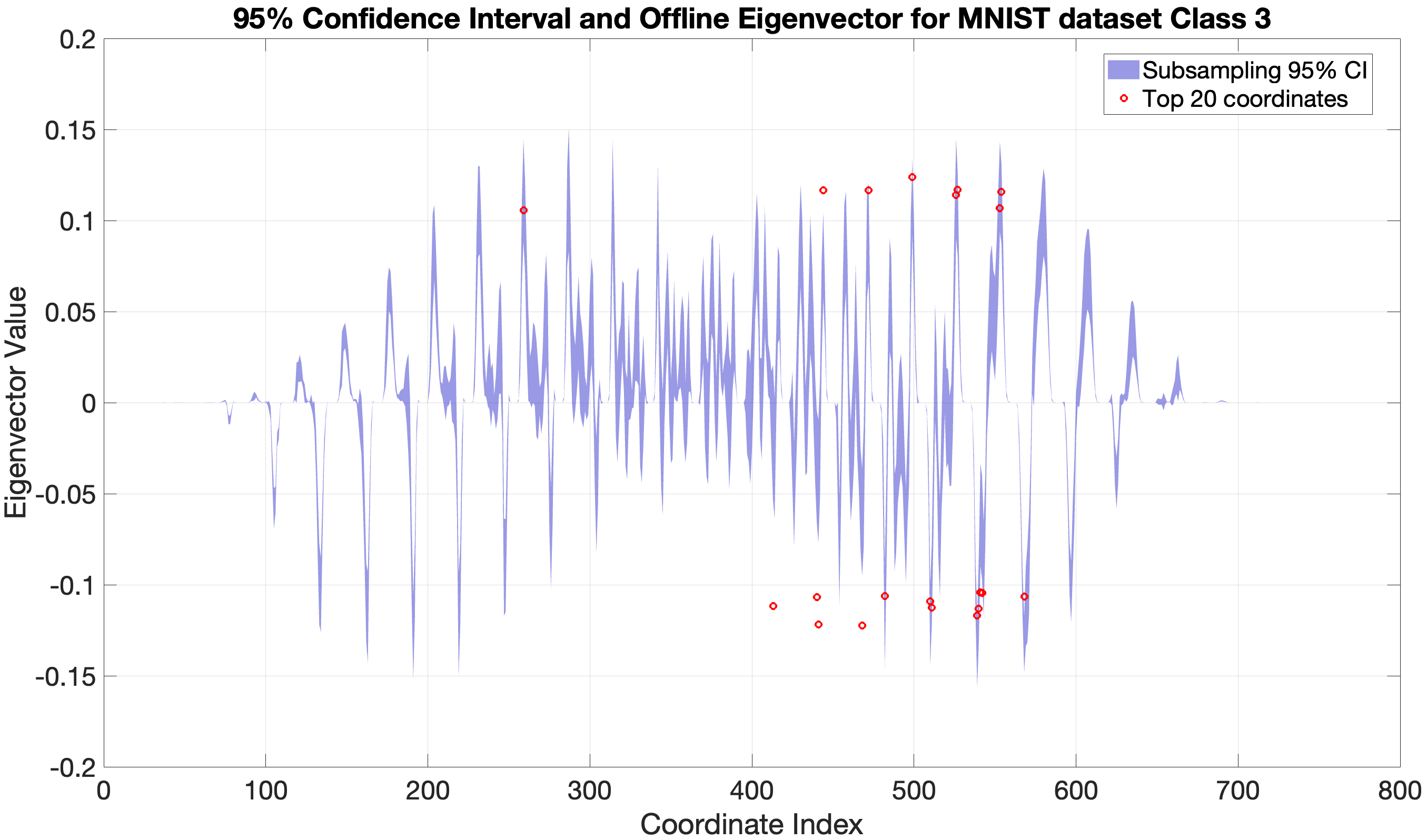}
        \captionsetup{labelformat=empty}
        \caption*{(3)}
    \end{minipage}%
    \hfill
    \begin{minipage}{0.32\columnwidth}
        \centering
        \includegraphics[width=\columnwidth]{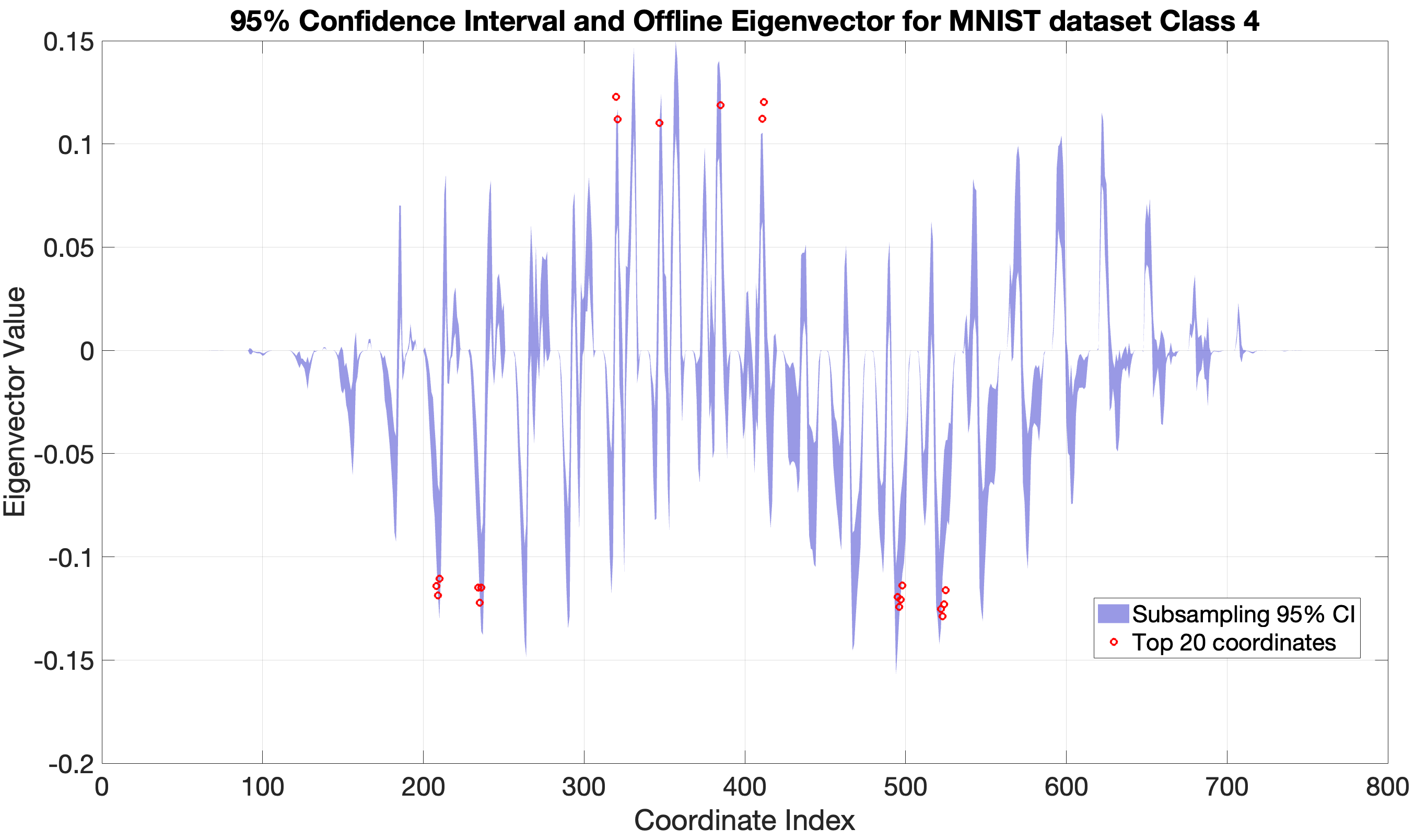}
        \captionsetup{labelformat=empty}
        \caption*{(4)}
    \end{minipage}%
    \hfill
    \begin{minipage}{0.32\columnwidth}
        \centering
        \includegraphics[width=\columnwidth]{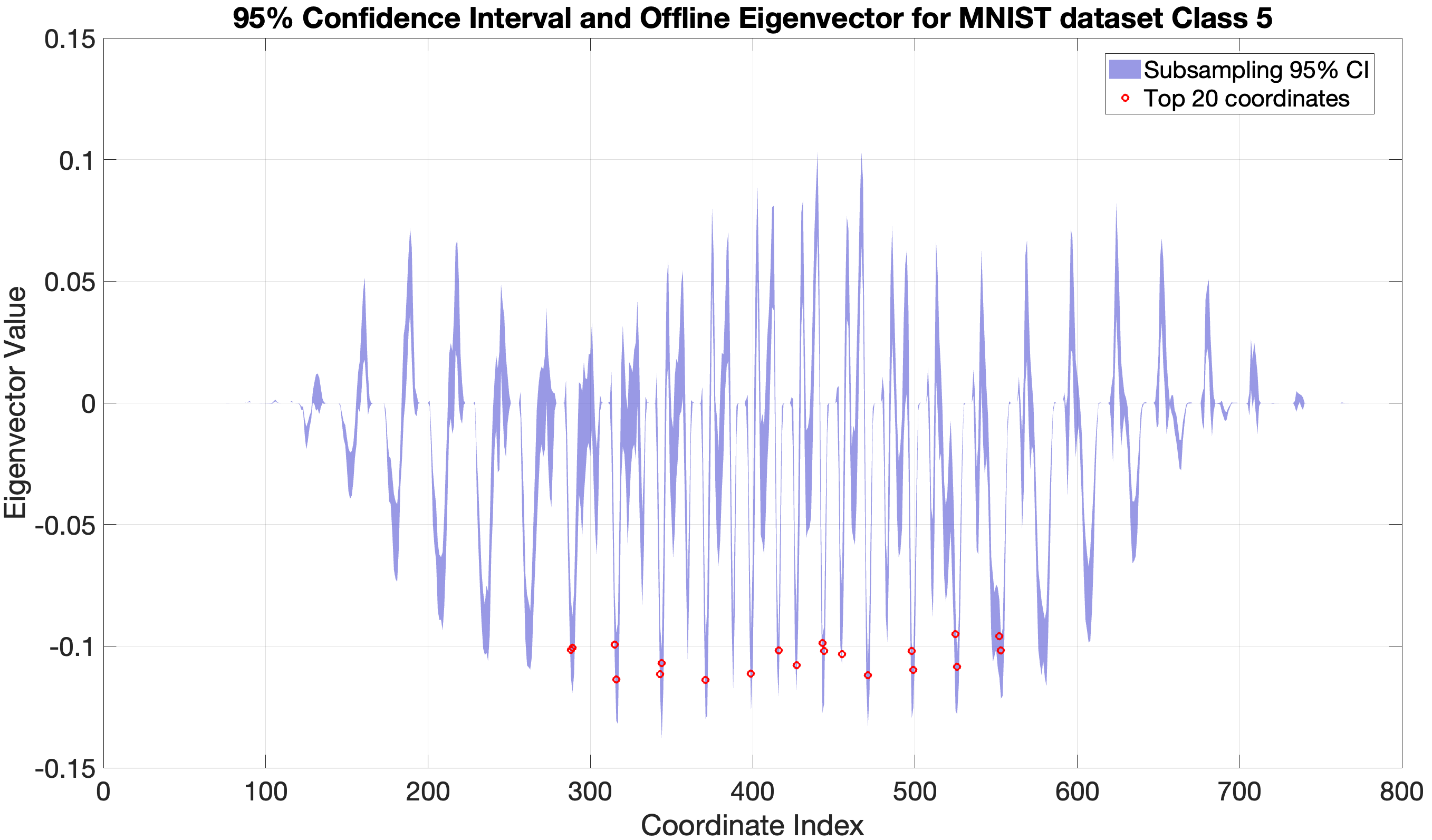}
        \captionsetup{labelformat=empty}
        \caption*{(5)}
    \end{minipage}%
    \hfill
    \begin{minipage}{0.32\columnwidth}
        \centering
        \includegraphics[width=\columnwidth]{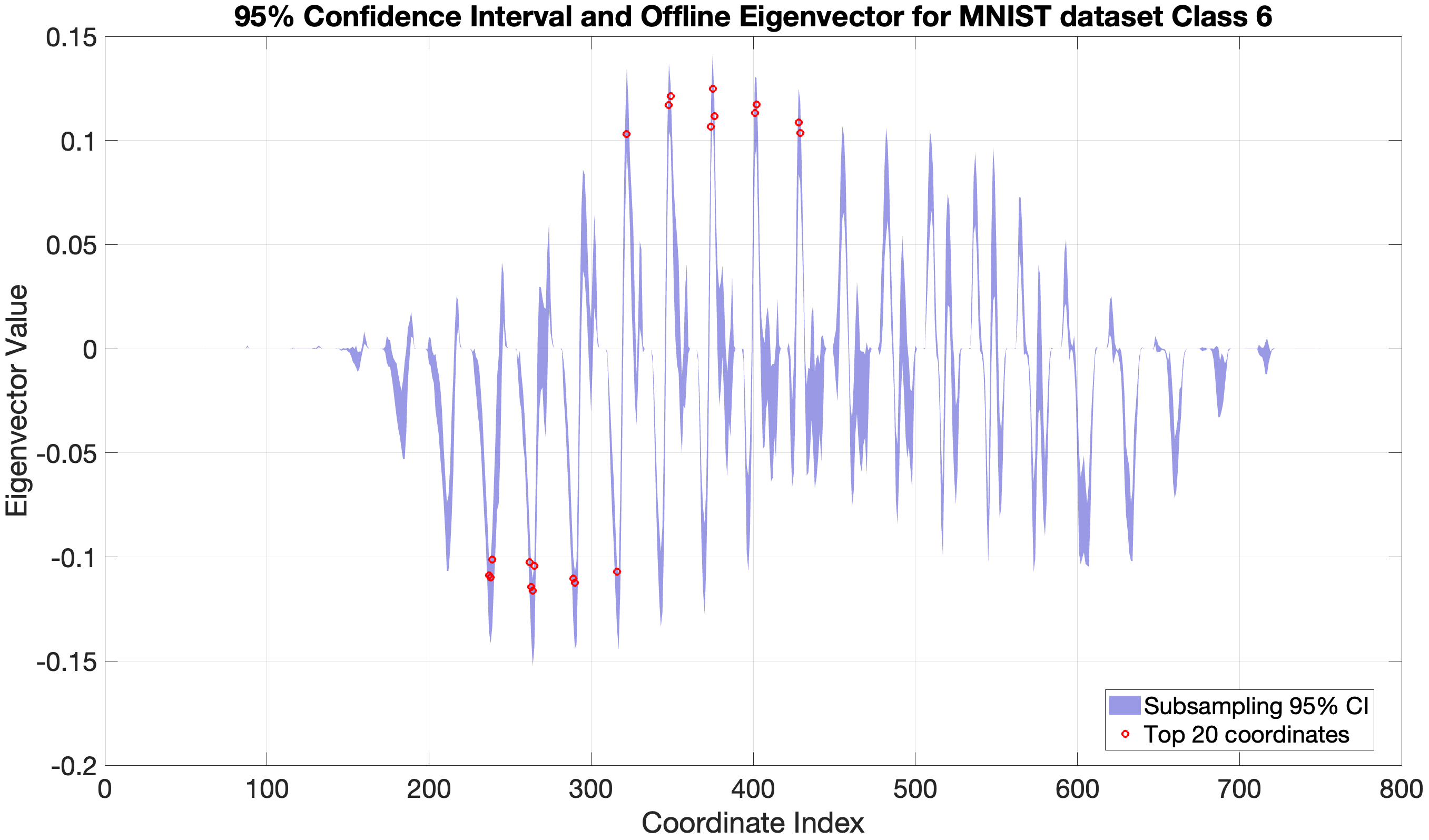}
        \captionsetup{labelformat=empty}
        \caption*{(6)}
    \end{minipage}%
    \hfill
    \begin{minipage}{0.32\columnwidth}
        \centering
        \includegraphics[width=\columnwidth]{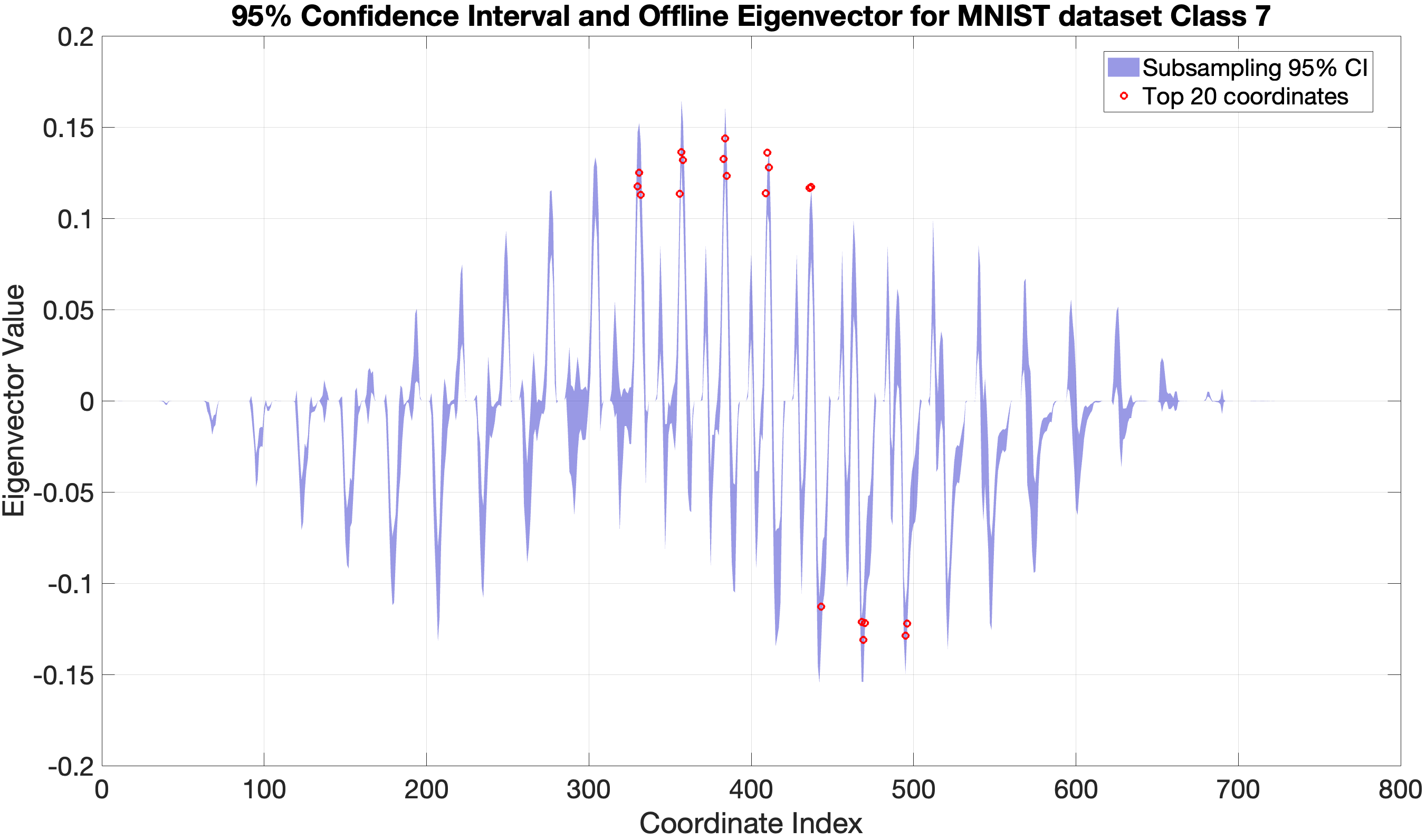}
        \captionsetup{labelformat=empty}
        \caption*{(7)}
    \end{minipage}%
    \hfill
    \begin{minipage}{0.32\columnwidth}
        \centering
        \includegraphics[width=\columnwidth]{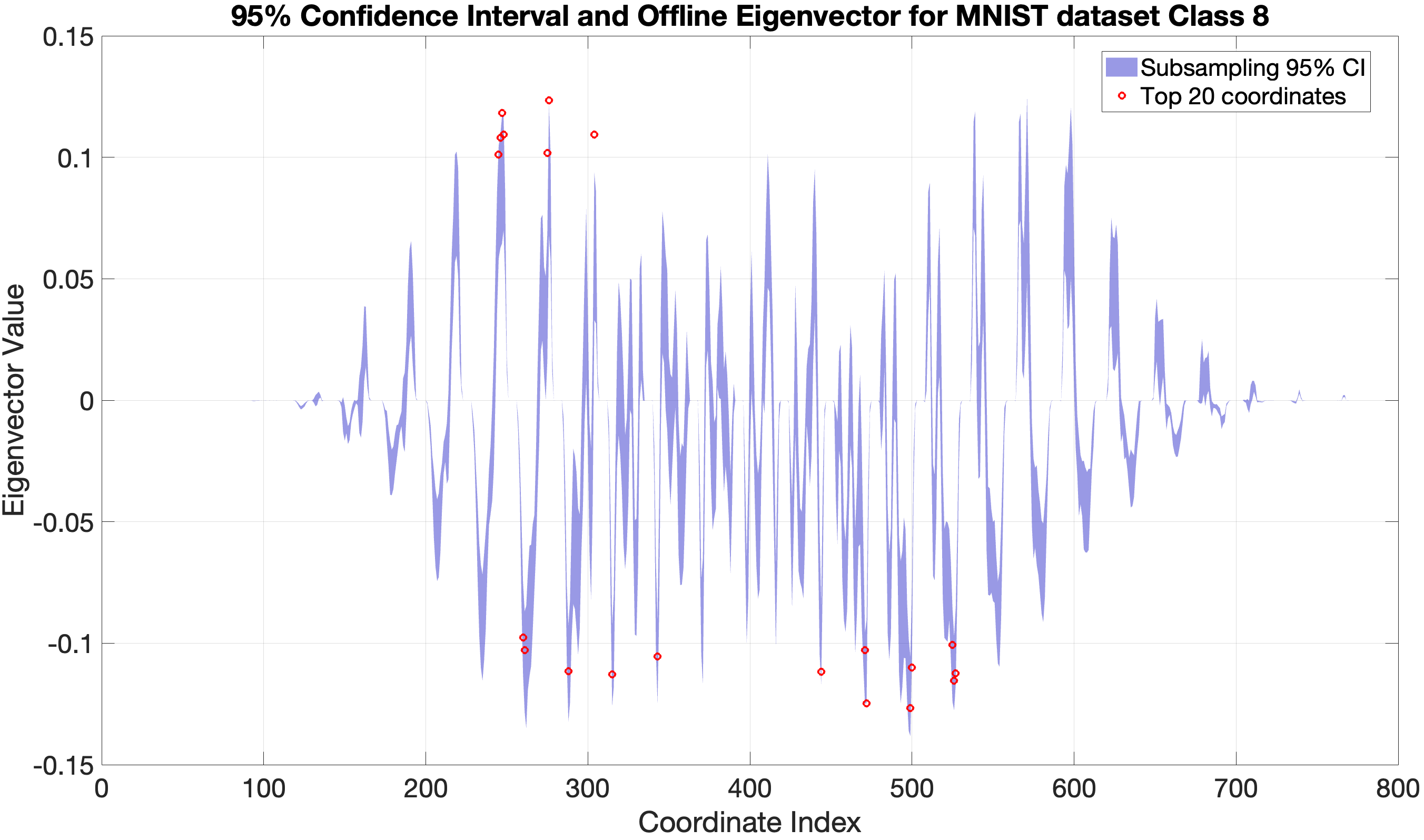}
        \captionsetup{labelformat=empty}
        \caption*{(8)}
    \end{minipage}%
    \hfill
    \begin{minipage}{0.32\columnwidth}
        \centering
        \includegraphics[width=\columnwidth]{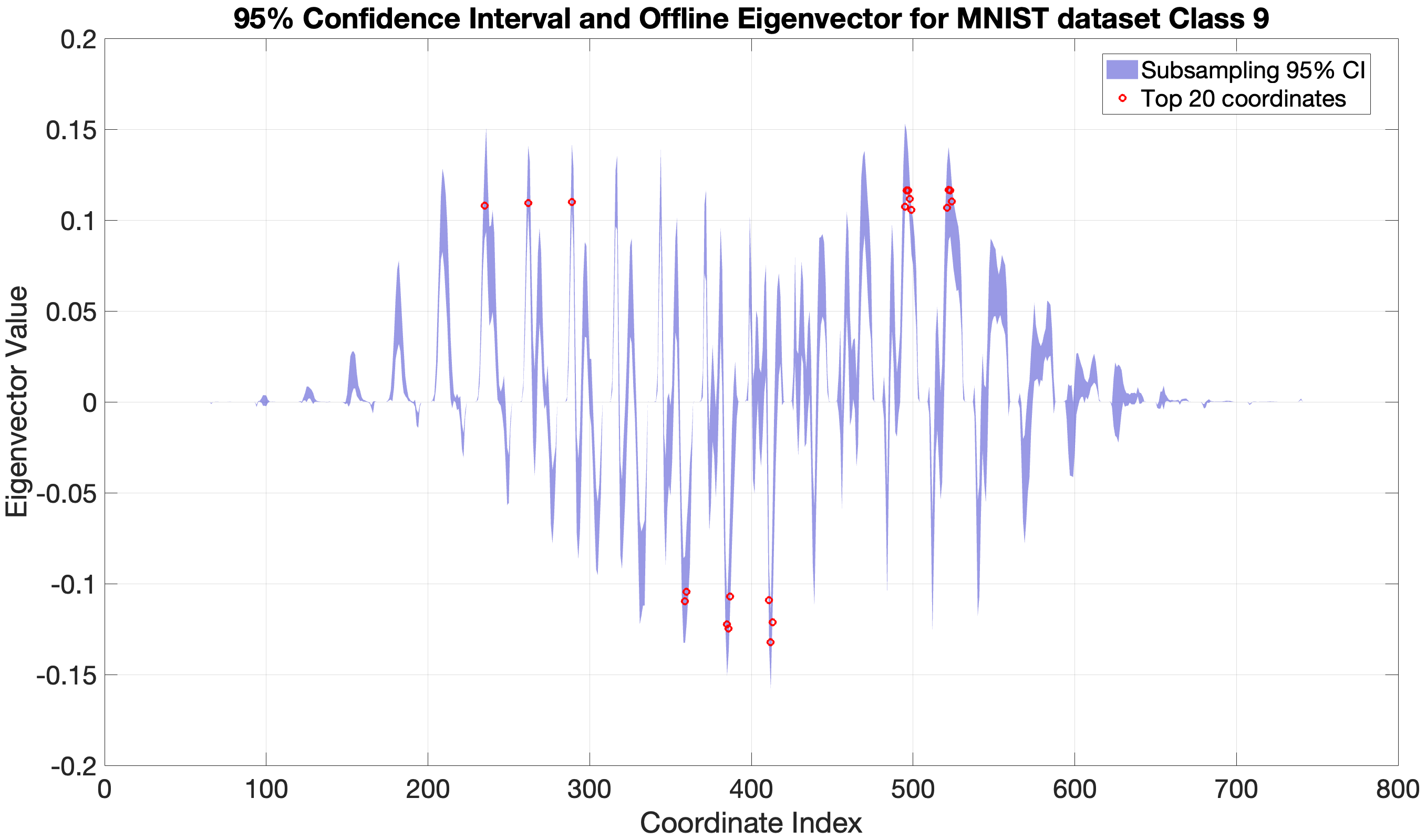}
        \captionsetup{labelformat=empty}
        \caption*{(9)}
    \end{minipage}%
    \hfill
    \caption{Uncertainty Estimation for MNIST dataset. The $\sin^{2}$ error of Oja’s algorithm for each class is provided in Table~\ref{tab:mnist-sin2-error}.}
    \label{fig:mnist_dataset}
\end{figure}

\section{Conclusion}
In this work, we develop a novel statistical inference framework for streaming PCA using Oja’s algorithm. We derive finite-sample and high-probability deviation bounds for the coordinates of the estimated eigenvector, establish a Bernstein-type concentration bound on the residual of the Oja vector, establish a Central Limit Theorem for suitable subsets of entries, and devise an efficient subsampling-based variance estimation algorithm. By leveraging the structure of the Oja updates, we provide entrywise confidence intervals, bypassing expensive resampling techniques such as bootstrapping. Our theoretical results are supported by extensive numerical experiments, indicating that our proposed estimator achieves accuracy similar to the multiplier bootstrap method while requiring significantly less time. 

We believe that our subsampling algorithm can be adapted to any SGD problem where the covariance matrix of the estimator $\hat{\theta}_n$ scales as $c_n$ times some scale-free matrix $\V$, where $c_n$ is known. This structure aligns with subsampling and m-out-of-n bootstrap methods, where the variance estimated from a subsample of size $m$ is scaled by $m/n$ to approximate the variance of the full sample estimator. Our findings also highlight the potential for improved uncertainty quantification techniques in streaming non-convex optimization problems beyond PCA, since Oja-type updates can be found in many important non-convex optimization algorithms such as matrix sensing, matrix completion, and subspace estimation. Further directions include deflation-based methods to apply our method to variance estimation for top $k$ eigenvectors.

\section*{Acknowledgments}
We gratefully acknowledge NSF grants 2217069, 2019844, and DMS 2109155. We are thankful to Soumendu Sundar Mukherjee and Arun Kuchibhotla for helpful discussions.

\bibliographystyle{plainnat}
\bibliography{refs}

\newpage
\onecolumn
\begin{appendix}



The Appendix is organized as follows:
\begin{enumerate}
    \item Section~\ref{appendix:utility_results} provides some useful results used in subsequent analyses
    \item Section~\ref{appendix:estimator_bias_variance} has the Bias and Concentration calculation of our estimator designed in Algorithm~\ref{alg:variance_estimation}
    \item Section~\ref{appendix:entrywise_error_bounds} provides high probability Entrywise Error Bounds on the entries of $\voja$
    \item Section~\ref{appendix:entrywise_clt} provides a Central Limit Theorem for the entries of the Oja vector, $\voja$, which ties the results developed in Section~\ref{appendix:estimator_bias_variance} to provide confidence intervals
\end{enumerate}

\section{Utility Results}\label{appendix:utility_results}

\begin{lemma}\label{lemma:square_expansion_cs}
    For any integer $n \ge 2$, real $\eps \in (0,1)$, and reals $\left\{a_i\right\}_{i \in [n]}$, 
    \bas{
        (1-\epsilon)a_1^2 - \frac{n-1}{\epsilon} \sum_{i=2}^n a_i^2
        \le \Bigl(\sum_{i=1}^n a_i\Bigr)^2
        \le (1+\epsilon)a_1^2 + \frac{2(n-1)}{\epsilon} \sum_{i=2}^n a_i^2.
    }
\end{lemma}
\begin{proof}
    We begin by writing
    \ba{
        \Bigl(a_1+ \sum_{i=2}^n a_i\Bigr)^2
        = a_1^2 + 2a_1\Bigl(\sum_{i=2}^n a_i\Bigr) + \Bigl(\sum_{i=2}^n a_i\Bigr)^2. \label{eq:square_expansion}
    }
    By Cauchy-Schwarz inequality,
    \ba{
         0 \le \Bigl(\sum_{i=2}^n a_i\Bigr)^2 \le (n-1)\sum_{i=2}^n a_i^2. \label{eq:cs_applicaition}
    }
    The cross-term can be bounded using the inequality
    \[
        -\epsilon x^2-\frac{1}{\epsilon}\,y^2 \le 2xy \le \epsilon x^2+\frac{1}{\epsilon}\,y^2
    \]
    with \(x=a_1\) and \(y=\sum_{i=2}^n a_i\) to get
    \[
        2a_1\Bigl(\sum_{i=2}^n a_i\Bigr)
        \ge -\epsilon a_1^2 - \frac{1}{\epsilon}\Bigl(\sum_{i=2}^n a_i\Bigr)^2 \geq -\epsilon a_1^2 - \frac{n-1}{\epsilon}\sum_{i=2}^n a_i^2,
    \]
    and
    \[
        2a_1\Bigl(\sum_{i=2}^n a_i\Bigr)
        \le \epsilon a_1^2 + \frac{1}{\epsilon}\Bigl(\sum_{i=2}^n a_i\Bigr)^2 \leq \epsilon a_1^{2} + \frac{n-1}{\epsilon}\sum_{i=2}^n a_i^2.
    \]
    The proof follows by using the above inequalities in \eqref{eq:square_expansion} followed by another application of \eqref{eq:cs_applicaition}. 
\end{proof}

\begin{lemma}\label{lemma:entrywise_to_sin_squared} Let $\mathbb{V}$ be the asymptotic variance matrix defined in Lemma~\ref{lemma:second_moment_matrix}, and let $\voja$ be the Oja vector as defined in \eqref{eq:voja_def}. If the coordinate-wise bound
\bas{
    \Abs{e_i^{\top}\bb{\voja - \bb{v_1^{\top}\voja}v_1}} \lesssim C_{d,n}\sqrt{\frac{\mathbb{V}_{kk}}{n}}
}
holds for every $i \in [d]$, where $C_{d,n}^2$ hides logarithmic factors in $d, n$, then
\bas{
    \sin^{2}\bb{\voja, v_1} = \sum_{i \in [d]}\bb{e_i^{\top}\bb{\voja - \bb{v_1^{\top}\voja}v_1}}^{2} \lesssim C_{d,n}^2\frac{\Nu}{\bb{\eigengap}^{2}n},
}
where $\Nu$ is the matrix variance statistic defined in Assumption~\ref{assumption:bounded_moments}. 
\end{lemma}
\begin{proof}
    By the definitions of $\mathbb{V}$ and $R_0$ as in Lemma~\ref{lemma:second_moment_matrix},
    \bas{
    \sum_{i \in [d]}\bb{e_i^{\top}\bb{\voja - \bb{v_1^{\top}\voja}v_1}}^{2} &\lesssim C_{d,n}^2\frac{\Tr\bb{\mathbb{V}}}{n} \leq \bb{\frac{C_{d,n}^2}{\eigengap}}\frac{\Tr\bb{R_0}}{n} = \bb{\frac{C_{d,n}^2}{\eigengap}}\frac{1}{n}\sum_{2 \le k \le d}\frac{\widetilde{M}_{kk}}{2\bb{\lambda_{1}-\lambda_{k}}} \\
    &\leq \frac{C_{d,n}^2}{\bb{\eigengap}^{2}}\frac{\Tr\bb{\E\bbb{\vp\bb{A-\Sigma}v_1v_1^{\top}\bb{A-\Sigma}\vp^{\top}}}}{n} \\
    &=  \frac{C_{d,n}^2}{\bb{\eigengap}^{2}}\frac{\E\bbb{\Tr\bb{\vp\bb{A-\Sigma}v_1v_1^{\top}\bb{A-\Sigma}\vp^{\top}}}}{n} \\
    &=  \frac{C_{d,n}^2}{\bb{\eigengap}^{2}}\frac{v_{1}^{\top}\E\bbb{\bb{A-\Sigma}\vp\vp^{\top}\bb{A-\Sigma}}v_{1}}{n} \\
    &\leq  \frac{C_{d,n}^2}{\bb{\eigengap}^{2}}\frac{v_{1}^{\top}\E\bbb{\bb{A-\Sigma}^{2}}v_{1}}{n} \leq C_{d,n}^2\frac{\Nu}{\bb{\eigengap}^{2}n}.
    }
\end{proof}

\begin{lemma}[Choice of learning rate]\label{lemma:learning_rate_choice} Let $\eta_{n} := \frac{\alpha\log\bb{n}}{n\bb{\lambda_1-\lambda_2}}$ for $\alpha > 1$. Then, under Assumptions~\ref{assumption:bounded_moments} and~\ref{assumption:sample_size} 
\begin{enumerate}
    \item $nd\exp\bb{-\eta_n n \bb{\eigengap}} = o\bb{1}$.
    \item $\max\left\{\eta_{n}, \frac{\log\bb{d}}{\eigengap}\right\}\frac{\Mtwo^{4}}{\eigengap}\eta_n^{2} = o\bb{1}$.
    \item $n\eta_n^2(2\lambda_1^2+\Mtwo^2) \leq 1$
\end{enumerate}
\end{lemma}
\begin{proof}
The above conditions on $\eta_n$ imply Corollary 1 in \cite{lunde2021bootstrapping}. Let's start with the first condition. We have
\bas{
    nd\exp\bb{-\eta_n n \bb{\eigengap}} \leq nd\exp\bb{-\alpha\log\bb{n}} = \frac{d}{n^{\alpha-1}} = o(1), \text{ using the bound on } d
}
For the second condition, we first note that for $n \geq \alpha\log\bb{n}$ provided by Assumption~\ref{assumption:sample_size}, 
\bas{
    \eta_n \leq \frac{\log\bb{d}}{(\eigengap)}
}
Now for the second condition, we require,
\bas{
\frac{\alpha^{2}\Mtwo^{4}\log^{2}\bb{n}\log\bb{d}}{n^2\bb{\eigengap}^{4}} = o(1)
}
which is again ensured by the condition on $n$ in Assumption~\ref{assumption:sample_size}.
\end{proof}

\begin{lemma}\label{lemma:MoM}
Let $t$ be a positive integer, $\delta \in (0,1)$, and let $I$ be an interval in $\mathcal{R}$. Suppose $a_1, a_2, \dots, a_t$ are independent random variables such that $P\bb{a_i \in I} \ge 3/4$. Then, for $t \ge 8 \log\bb{1/\delta}$,
\bas{
P\bb{\median\bb{\left\{a_{i} \right\}_{i \in [t]}} \in I} \ge 1-\delta.
}
\end{lemma}

\begin{proof}
Since $I$ is an interval, the median does lies in $I$ if at least half the $a_i$ are in $I$. Let $b_i$ be the indicator that $a_i \notin I$, and let $B = \sum_{i \in [t]} b_i$. Then, $b_1, b_2, \dots, b_t$ are independent Bernoulli random variables each with mean at most $1/4$. By Hoeffding's inequality, 
\bas{
P\bb{\median\bb{\left\{a_{i} \right\}_{i \in [t]}} \notin I} \le P\bb{B > t/2} \le \exp\bb{-2(t/2 - \E\bbb{B})^2/t} \le \exp\bb{-t/8} \le \delta. 
}
\end{proof}

\begin{lemma}\label{lemma:MoM_modified_helper}
    Consider random variables $\left\{\bb{a_i, b_i}\right\}_{i \in [T]}$, $a_i, b_i \in \mathbb{R}$ such that $\left\{b_{i}\right\}_{i \in [T]}$ are mutually independent. For a fixed $c_{T} > 0$, define the event $\mathcal{E} := \left\{\forall i \in [n], a_{i} \leq c_{T}\right\}$. Further, suppose there exists a bound $d_{T}$ such that for any $i \in [T]$, $\Prob\bb{b_i \leq d_T} \geq 3/4$. Then for any $\delta \in \bb{0,1}$ and $T \geq 8\log\bb{1/\delta}$, we have 
    \bas{
        \Prob\bb{\median\bb{\left\{a_{i}+b_{i} \right\}_{i \in [T]}} \geq c_T + d_T} \leq \delta + \Prob\bb{\mathcal{E}^{\complement}}
    }
\end{lemma}
\begin{proof}
    Define indicator random variables $X_{i}, Y_{i}, i \in [T]$ as 
    \ba{
        X_{i} := \ind\bb{a_i + b_i \geq c_{T}+d_{T}}, \; Y_{i} := \ind\bb{b_i \geq d_{T}}
    }
    Then, we have
    \bas{
        \Prob\bb{\median\bb{\left\{a_{i}+b_{i} \right\}_{i \in [T]}} \geq c_T + d_T} &\leq \Prob\bb{\sum_{i\in[T]}X_i \geq \frac{T}{2}} \\
        &\leq \Prob\bb{\left\{\sum_{i\in[T]}X_i \geq \frac{T}{2}\right\} \bigcap \mathcal{E}} + \Prob\bb{\mathcal{E}^{\complement}} \\
        &\stackrel{(i)}{\leq} \Prob\bb{\left\{\sum_{i\in[T]}Y_i \geq \frac{T}{2}\right\} \bigcap \mathcal{E}} + \Prob\bb{\mathcal{E}^{\complement}} \\
        &\leq \Prob\bb{\sum_{i\in[T]}Y_i \geq \frac{T}{2}} + \Prob\bb{\mathcal{E}^{\complement}} \\
        &\stackrel{(ii)}{\leq} \exp\bb{-2\bb{T/2 -\sum_{i\in[T]}\E\bbb{Y_i}}^{2}/T} + \Prob\bb{\mathcal{E}^{\complement}} \\
        &\stackrel{(iii)}{\leq} \exp\bb{-T/8} + \Prob\bb{\mathcal{E}^{\complement}}\leq \delta + \Prob\bb{\mathcal{E}^{\complement}}
    }
    Here $(i)$ follows by the definition of the event $\mathcal{E}$, $(ii)$ follows from Hoeffding's inequality and $(iii)$ follows from $\E\bbb{Y_i} \leq 1/4$.
\end{proof}

\begin{lemma}\label{lemma:MoM_modified} Let $\left\{\nu_{i}\right\}_{i \in [T]}, \nu_{i} \in \R$ be random variables satisfying $\nu_{i} = \nu_{T} + a_{i} + b_{i}$ for a fixed scalar $\nu_{T}$ and random variables $\left\{\bb{x_i, y_i}\right\}_{i \in [T]}$, $x_i, y_i \in \mathbb{R}$ such that $\left\{y_{i}\right\}_{i \in [T]}$ are mutually independent. For a fixed $c_{T} > 0$, define the event $\mathcal{E} := \left\{\forall i \in [n], |x_{i}| \leq c_{T}\right\}$. Further, let there exist $d_{T}$ such that for any $i \in [T]$, $\Prob\bb{|y_i| \leq d_T} \geq 3/4$. Then for any $\delta \in \bb{0,1}$ and $T \geq 8\log\bb{1/\delta}$, 
\bas{
    \Prob\bb{\left|\median\bb{\left\{\nu_i \right\}_{i \in [T]}} - \nu_{T}\right| \geq c_T + d_T} \leq \delta + \Prob\bb{\mathcal{E}^{\complement}}
}   
\end{lemma}
\begin{proof}
    We first note that 
    \ba{
        \median\bb{\left\{\nu_i \right\}_{i \in [T]}} = \median\bb{\left\{\nu_{T} + x_{i} + y_{i} \right\}_{i \in [T]}} = \nu_{T} + \median\bb{\left\{ x_{i} + y_{i} \right\}_{i \in [T]}} \label{eq:median_eq_1}
    }
    Next we note that for any $\lambda > 0$, 
    \bas{
        \median\bb{\left\{ |x_{i}| + |y_{i}| \right\}_{i \in [T]}} \leq \lambda \implies \left|\left\{i \in [T] : |x_i| + |y_i| \leq \lambda\right\}\right| \geq \frac{T}{2}
    }
    Therefore, for all $i\in \mathcal{S} := \left\{i \in [T] : |x_i| + |y_i| \leq \lambda\right\}$ we have
    \bas{
       |x_i + y_i| \leq  |x_i| + |y_i| \leq \lambda 
    }
    This in turn implies that for $\mathcal{\tilde{S}} := \left\{i \in [T] : -\lambda \leq x_i + y_i \leq \lambda \right\}$, $|\tilde{S}| \geq \frac{T}{2}$. Therefore, 
    \bas{
        \median\bb{\left\{ x_{i} + y_{i} \right\}_{i \in [T]}} \in \bbb{-\lambda, \lambda} \iff |\median\bb{\left\{ x_{i} + y_{i} \right\}_{i \in [T]}}| \leq \lambda
    }
    Therefore, we have shown that for any $\lambda > 0$, 
    \bas{
        \median\bb{\left\{ |x_{i}| + |y_{i}| \right\}_{i \in [T]}} \leq \lambda \implies |\median\bb{\left\{ x_{i} + y_{i} \right\}_{i \in [T]}}| \leq \lambda 
    }
    or equivalently, 
    \bas{
        \median\bb{\left\{ |x_{i}| + |y_{i}| \right\}_{i \in [T]}} \geq |\median\bb{\left\{ x_{i} + y_{i} \right\}_{i \in [T]}}|
    }
    Substituting in \eqref{eq:median_eq_1}, we have
    \bas{
        |\median\bb{\left\{\nu_i \right\}_{i \in [T]}} - \nu_{T}| = |\median\bb{\left\{ x_{i} + y_{i} \right\}_{i \in [T]}}| \leq \median\bb{\left\{ |x_{i}| + |y_{i}| \right\}_{i \in [T]}}
    }
    The result then follows from Lemma~\ref{lemma:MoM_modified_helper} with $a_{i} := |x_i|$ and $b_i := |y_i|$.
\end{proof}
\bk

We next define a variant of Oja's algorithm, $\Ojamain$, which provides a high-probability convergence guarantee. This can be satisfied by geometric boosting of the standard $\Oja$ algorithm (see for e.g. Lemma 3.9 in \citet{kumarsarkar2024sparse}), and we use it to obtain our high-accuracy estimate, $\vmain$, used for recentering and variance computation subsequently in Algorithm~\ref{alg:variance_estimation}.

\begin{definition}\label{def:oja_high_prob} Let $N > 0$, $u_{0} \sim \mathcal{N}\bb{0, I}$ and $\left\{X_{i}\right\}_{i \in [N]}$ be $\iid$ datapoints satisfying Assumption~\ref{assumption:bounded_moments}. We define $\vmain \gets \Ojamain(\left\{X_{i}\right\}_{i \in [N]}, u_0)$ to be any variant of Oja's algorithm satisfying, with probability at least $1-\delta$,
\bas{
    \sin^2 \bb{\vmain, v_1} \le \frac{C\log\bb{\frac{1}{\delta}}\log\bb{N/\log\bb{\frac{1}{\delta}}}\Mtwo^{2}}{N(\eigengap)^{2}}
}
where $C > 0$ is a universal constant and $\delta \in \bb{0,1}$.
\end{definition}
\section{Estimator Concentration}\label{appendix:estimator_bias_variance}

\ojaerrordecomposition*
\begin{proof}
    We have, 
    \bas{
        \voja &= (v_{1}^{\top}\voja)v_{1} + \vp\vp^{\top}\voja \\
        &= (v_{1}^{\top}\voja)v_{1} + \frac{\vp\vp^{\top}B_{n}u_{0}}{\norm{B_nu_0}_{2}} \\
        &= (v_{1}^{\top}\voja)v_{1} + \frac{\vp\vp^{\top}B_{n}u_{0}}{c_{n}} + \Ethree{n} \\
        &= (v_{1}^{\top}\voja)v_{1} + \frac{\vp\vp^{\top}B_{n}v_{1}\sign(v_{1}^{\top}u_{0})}{(1+\eta_{n}\lambda_{1})^{n}} + \Ethree{n} + \Efour{n} \\
        &= (v_{1}^{\top}\voja)v_{1} + \frac{\vp\vp^{\top}(B_{n} - \E\bbb{B_n})v_{1}\sign(v_{1}^{\top}u_{0})}{(1+\eta_{n}\lambda_{1})^{n}} + \Ethree{n} + \Efour{n} \\
        &= (v_{1}^{\top}\voja)v_{1} + \frac{\vp\vp^{\top}(\sum_{k\geq1}T_{n,k})v_{1}\sign(v_{1}^{\top}u_{0})}{(1+\eta_{n}\lambda_{1})^{n}} + \Ethree{n} + \Efour{n}, \text{ using Theorem A.1 \cite{lunde2021bootstrapping}} \\
        &= (\vmain^{\top}\voja)\vmain + \Ezero{n} + \Eone{n} + \Etwo{n} + \Ethree{n} + \Efour{n}.
    }
\end{proof}

\begin{lemma}\label{lemma:hajek_decomposition} Let $\Eone{n}$ be as defined in Lemma~\ref{lemma:oja_error_decomposition}. Then, 
\bas{
    \Eone{n} := \eta_{n}Y_{n}, \text{ for } Y_{n} := \sum_{j=1}^{n}X^{n}_{j} \text{ and } X^{n}_{j} := \frac{\sign\bb{v_1^{\top}u_0}}{1+\eta_n\lambda_1}\vp\lambp^{n-j}\vp^{\top}\bb{A_{j}-\Sigma}v_{1}
}
where $\lambp \in \R^{(d-1) \times (d-1)}$ is a diagonal matrix with entries $\lambp(i,i) = \frac{1+\eta_n\lambda_{i+1}}{1+\eta_n\lambda_{1}}$.
\end{lemma}

Let $\left\{A_{i}\right\}_{i \in [n]}$ be symmetric independent matrices satisfying $\E\bbb{A_i} = \Sigma$, $\norm{\E\bbb{\bb{A_i-\Sigma}^{2}}}_{2} \leq \mathcal{V}$ and $\norm{A_i-\Sigma}_{2} \leq \mathcal{M}$. Define, 
\bas{
    \forall j \in [n], \;\; X_{j}^{n} := \vp\lambp^{n-j}\vp^{\top}\bb{A_{j}-\Sigma}v_{1}, \text{ and } Y_{n} := \sum_{j \in [n]}X_{j}^{n}
}

\subsection{Estimator Bias}\label{appendix:estimator_bias}


\begin{proof}[Proof of Lemma~\ref{lemma:second_moment_matrix}]
    Using the definitions of $Y_n$ and $X_{j}^n$ from Lemma~\ref{lemma:hajek_decomposition}, we have
    \bas{
        \frac{1}{\eta_n^2} \E\bbb{\Psi_{n,1}\Psi_{n,1}^{\top}} = \E\bbb{Y_n Y_n^\top} &= \sum_{j, k \in [i]}\E\bbb{X_j^n X_k^{n \top}} \\
        &= \sum_{j \in [n]}\E\bbb{X_{j}^{n}X_{j}^{n \top}}, \;\; \text{ since } A_{j}, A_{k} \text{ are independent for } j \neq k \\
        &= \frac{1}{\bb{1+\eta_n\lambda_1}^{2}}\sum_{j \in [n]}\vp\lambp^{n-j}\vp^{\top}\E\bbb{\bb{A_{j}-\Sigma}v_{1}v_{1}^\top\bb{A_{j}-\Sigma}}\vp\lambp^{n-j}\vp^{\top} \\
        & = \frac{1}{\bb{1+\eta_n\lambda_1}^{2}}\vp\bb{\sum_{j \in [n]}\lambp^{n-j}\underbrace{\vp^{\top}\E\bbb{\bb{A_{j}-\Sigma}v_{1}v_{1}^\top\bb{A_{j}-\Sigma}}\vp}_{:= \widetilde{M}}\lambp^{n-j}}\vp^{\top}.
    }
    Recall $R^{(n)} := \frac{1}{\bb{1+\eta_n\lambda_1}^{2}}\sum_{j \in [n]}\lambp^{n-j}\widetilde{M}\lambp^{n-j}$ and consider $\bb{k,l}^{\text{th}}$ entry of $R^{(n)}$. 
    \bas{
        R_{kl}^{(n)} &= \frac{1}{\bb{1+\eta_n\lambda_1}^{2}}e_{k}^{\top}\sum_{j \in [n]}\lambp^{n-j}\widetilde{M}\lambp^{n-j}e_{l} = \frac{1}{\bb{1+\eta_n\lambda_1}^{2}}\widetilde{M}_{kl}\sum_{j=1}^{n}\bb{d_{k}d_{l}}^{n-j} = \frac{1}{\bb{1+\eta_n\lambda_1}^{2}}\widetilde{M}_{kl} \bb{\frac{1-\bb{d_kd_l}^{n}}{1-d_kd_l}}.
    }
    Let $R_0(k,l)=\widetilde{M}_{k\ell}/(2\lambda_1-\lambda_{k+1}-\lambda_{\ell+1})$. 
    Note that 
    \bas{
        1-d_kd_l &= \frac{\eta_n\bb{2\lambda_{1}-\lambda_{k+1}-\lambda_{l+1}}}{1+\eta\lambda_{1}} - \frac{\eta_n^{2}\bb{\lambda_{1}-\lambda_{k+1}}\bb{\lambda_{1}-\lambda_{l+1}}}{\bb{1+\eta\lambda_{1}}^{2}} \\
        &= \frac{\eta_n\bb{2\lambda_{1}-\lambda_{k+1}-\lambda_{l+1}}}{1+\eta\lambda_{1}}\bbb{1 - \frac{\eta_{n}\bb{\lambda_{1}-\lambda_{k+1}}\bb{\lambda_{1}-\lambda_{l+1}}}{\bb{1+\eta\lambda_{1}}\bb{ \lambda_{1}-\lambda_{k+1} +  \lambda_{1}-\lambda_{l+1}}} } \\
        &\geq \frac{\eta_n\bb{2\lambda_{1}-\lambda_{k+1}-\lambda_{l+1}}}{1+\eta\lambda_{1}}\bbb{1 - \frac{\eta_{n}\bb{\lambda_{1}-\lambda_{k+1}}\bb{\lambda_{1}-\lambda_{l+1}}}{\bb{ \lambda_{1}-\lambda_{k+1} +  \lambda_{1}-\lambda_{l+1}}} } \\
        &\geq \frac{\eta_n\bb{2\lambda_{1}-\lambda_{k+1}-\lambda_{l+1}}}{1+\eta\lambda_{1}}\bbb{1 - \eta_{n}\min\left\{ \lambda_{1}-\lambda_{k+1},  \lambda_{1}-\lambda_{l+1}\right\} } \\
        &\geq \frac{\eta_n\bb{2\lambda_{1}-\lambda_{k+1}-\lambda_{l+1}}}{1+\eta\lambda_{1}}\bbb{1 - \eta_{n}\lambda_{1} } \\
        &\geq \eta_n\bb{2\lambda_{1}-\lambda_{k+1}-\lambda_{l+1}}\bb{1-O\bb{\eta_n\lambda_{1}}}
    }
    Then,
    \bas{
    R^{(n)}_{kl}-R_0(k,l)/\eta_n&=\frac{\widetilde{M}_{k\ell}}{\eta_n(2\lambda_1-\lambda_{k+1}-\lambda_{\ell+1})}\frac{(1+O(\eta_n\lambda_1))}{\bb{1+\eta_n\lambda_1}^{2}}-\frac{\widetilde{M}_{k\ell}}{\eta_n(2\lambda_1-\lambda_{k+1}-\lambda_{\ell+1})}\\
    &=\frac{\widetilde{M}_{k\ell}}{\eta_n(2\lambda_1-\lambda_{k+1}-\lambda_{\ell+1})}(1+O(\eta_n\lambda_1))-\frac{\widetilde{M}_{k\ell}}{\eta_n(2\lambda_1-\lambda_{k+1}-\lambda_{\ell+1})}\\
    &=\frac{\widetilde{M}_{k\ell}}{\eta_n(2\lambda_1-\lambda_{k+1}-\lambda_{\ell+1})}O(\eta_n\lambda_1)
    }
    So we have:
    \bas{
    \frac{\eta_n R^{(n)}_{kl}-R_0(k,l)}{R_0(k,l)}&=O(\eta_n\lambda_1)
    }
    Finally, we have:
    \bas{
    \|\eta_n R^{(n)}-R_0\|_F\leq \frac{\eta_{n} \lambda_1}{\lambda_1-\lambda_2}\|\widetilde{M}\|_F/2
    }
    Note that
    \bas{
    \|\widetilde{M}\|_F^{2} \leq \E\bbb{\norm{(A_i-\Sigma)v_1v_1^{\top}(A_i-\Sigma)}} \leq \E\bbb{\|A_i-\Sigma\|^2} \le \mathcal{M}_{2}^{2}.
    }
\end{proof}

\subsection{Estimator Concentration}\label{appendix:estimator_concentration}

In this section, we estimate the bias of the variance estimate output by Algorithm~\ref{alg:variance_estimation}. In the entirety of this section, we assume that the vector $\vmain$ is ``good'', i.e $\sin^{2}\bb{\vmain, v_1} \lesssim \frac{\log\bb{1/\delta}}{\delta^{3}}\frac{\eta_{N}\Mtwo^{2}}{\bb{\eigengap}}$, which happens with probability at least $1-\delta$. Recall that $\vmain \gets \Oja(\mathcal{D}_{N}, \eta_{N}, u_0)$ is the high accuracy estimate of $v_1$. We present all results using a general $n$ number of $\iid$ samples per split, which will later be replaced by $n/(m_1m_2)$ as required by Algorithm~\ref{alg:variance_estimation}. We denote $s_{n} := \frac{\log\bb{1/\delta}}{\delta^{3}}\frac{\eta_{n}\Mtwo^{2}}{\bb{\eigengap}}$ to be the upper bound on the $\sin^2$ error of the Oja vector due to~\cite{jain2016streaming}. While our results henceforth are written using $s_n$ and $s_n$ is not guaranteed to be smaller than $1$, it is straightforward to replace it by $\min\left\{s_n, 1\right\}$ since the $\sin^2$ error between any two vectors is always at most $1$.

\subsubsection{$\Ezero{n}$ Tail Bound}\label{sec:en0_concentration}

\begin{lemma}\label{lemma:en0_tail_bound} Let $\Ezero{n}$ be defined as in Lemma~\ref{lemma:oja_error_decomposition} for $\voja$ defined in \eqref{eq:voja_def}. Let $\left\{\Ezero{n}^{(i)}\right\}_{i \in [m]}$ and $\left\{\voja^{(i)}\right\}_{i \in [m]}$ be $m$ $\iid$ instances of $\Ezero{n}$ and $\voja$ respectively. 

Suppose the vector $\vmain$ satisfies the bound of Definition~\ref{def:oja_high_prob}. Then, for any $k \in [d]$,
\bas{
\sum_{i\in [m]}\frac{\bb{e_k^{\top} \Ezero{n}^{(i)}}^2}{m} \le \frac{C\log\bb{\frac{1}{\delta}}\log\bb{N/\log\bb{\frac{1}{\delta}}}\Mtwo^{2}}{N(\eigengap)^{2}}.
}
\end{lemma}
\begin{proof}
    For any $i \in [m]$,
    \bas{
    \Abs{e_k^{\top} \Ezero{n}^{(i)}} &= \Abs{e_k^{\top}\bb{v_1v_1^{\top} - \vmain \vmain^{\top}} \voja^{(i)}} \le \norm{v_1v_1^{\top} - \vmain \vmain^{\top}} = \sqrt{2}\Abs{\sin\bb{\vmain, v_1}}.
    }
    The result now follows from the assumed bound 
    \bas{
    \sin^2 \bb{\vmain, v_1} \le \frac{C\log\bb{\frac{1}{\delta}}\log\bb{N/\log\bb{\frac{1}{\delta}}}\Mtwo^{2}}{N(\eigengap)^{2}}.
    }
    for some universal constant $C > 0$. 
\end{proof} 
\subsubsection{$\Eone{n}$ (Hajek Projection) Concentration}
\label{appendix:Eone_bound}

\begin{lemma}\label{lemma:en1_concentration_bound} Let $\Eone{n}$ be defined as in Lemma~\ref{lemma:oja_error_decomposition} for $u_{0} = g/\norm{g}_{2}$ with $g \sim \mathcal{N}(0, \id_d)$. Let $\left\{\Eone{n}^{(i)}\right\}_{i \in [m]}$ and $\left\{g^{(i)}\right\}_{i \in [m]}$ be $m$  $\iid$ instances of $\Eone{n}$ and $g$ respectively. Then, for any $\delta \in \bb{0,1}$ and $k \in [d]$, with probability at least $1-\delta$,
\bas{
    \Abs{\frac{\sum_{i \in [m]}\bb{e_{k}^{\top}\Eone{n}^{(i)}}^{2}}{m} - \E\bbb{\bb{e_{k}^{\top}\Eone{n}}^{2}}} \le \frac{\sqrt{2} \E\bbb{\bb{e_k^\top \Eone{n}}^2} + \eta_n^2 b_k^2 \mathcal{M}_4^2 \sqrt{n}}{\sqrt{m\delta}}.
}
where $b_k := \norm{\vp^{\top}e_k}_{2}$.
\end{lemma}

\begin{proof}
Recall the notations $X_j^n = \vp\lambp^{n-j}\vp^{\top}\bb{A_{j}-\Sigma}v_{1}$ and $Y_n = \sum_{j=1}^n X_j^n$ from Lemma~\ref{lemma:second_moment_matrix}. Since $\vp \vp^\top X_j^n = X_j^n$ and $T_{n,1} = \eta_n \sum_{i=1}^n X_j^n$, $e_k^\top \Eone{n}$ can be written as
\ba{
e_k^\top \Eone{n} = \frac{e_k^\top \vp \vp^\top T_{n,1} v_1 \sign(v_1^\top u_0)}{(1+\eta_n \lambda_1)^n} = \frac{\eta_n \sign(v_1^\top u_0)}{(1+\eta_n \lambda_1)} \sum_{j=1}^n e_k^\top \vp \vp^\top X_j^n = \frac{\eta_n \sign(v_1^\top u_0)}{1+\eta_n \lambda_1} e_k^\top Y_n. \label{eq:en1-Y-relation}
}
Next, we bound the variance of $(e_k^{\top}Y_n)^2$.
\bas{
(e_k^{\top} Y_n)^2 = \sum_{j=1}^n \bb{e_k^{\top} X_j^n}^2 + 2\sum_{j < j'} \bb{e_k^{\top} X_j^n}\bb{e_k^{\top} X_{j'}^n}.
}
Most pairs of summands are uncorrelated.
\begin{itemize}
\item $\cov\bb{(e_k^\top X_j^n)^2, (e_k^\top X_{j'}^n)^2} = 0$ for any distinct $j, j' \in [n]$.

\item $\cov\bb{(e_k^\top X_{\ell}^n)^2, (e_k^\top X_j^n)(e_k^\top X_{j'}^n)} = 0$ for any $\ell \in [n]$ and $1 \le j < j' \le n$.

\item $\cov\bb{(e_k^\top X_{j}^n)(e_k^\top X_{j'}^n), (e_k^\top X_{\ell}^n)(e_k^\top X_{\ell'}^n)} = 0$ for any $1 \le j < j' \le n$ and $1 \le \ell < \ell' \le n$ such that $(j,j') \neq (\ell, \ell')$.
\end{itemize}
It follows that
\ba{
\Var\bb{(e_k^{\top} Y_n)^2} = \sum_{j=1}^n \Var\bb{(e_k^\top X_{j}^n)^2} + 4\sum_{j < j'} \Var\bb{(e_k^\top X_{j}^n)(e_k^\top X_{j'}^n)}. \label{eq:en1_error_decomposition}
}
We bound both terms separately. By Lemma~\ref{lemma:second_moment_matrix}, the second term can be bounded as
\ba{
4\sum_{j < j'} \Var\bb{(e_k^\top X_{j}^n)(e_k^\top X_{j'}^n)} &= 4 \sum_{i < j} \E\bbb{(e_k^\top X_{j}^n)^2}\E\bbb{(e_k^\top X_{j'}^n)^2} \nonumber \\
&\le 2 \sum_{j=1}^n \sum_{j'=1}^n \E\bbb{(e_k^\top X_{j}^n)^2}\E\bbb{(e_k^\top X_{j'}^n)^2} = 2 \E\bbb{(e_k^\top Y_n)^2}^2. \label{eq:en1_error_decomposition_1}
}
Next, we bound the first term of Equation~\eqref{eq:en1_error_decomposition}. For any $j \in [n]$,
\bas{
\Abs{e_k^\top X_j^n} &= \Abs{e_k^\top \vp \lambp^{n-j} \vp^\top (A_i-\Sigma) v_1} \le \norm{e_k^\top \vp} \norm{\lambp^{n-j}} \norm{\vp^\top (A_j-\Sigma) v_1} \le b_k \norm{A_j-\Sigma},
}
which implies
\ba{
\sum_{j=1}^n \Var\bb{(e_k^\top X_{j}^n)^2} \le \sum_{j=1}^n \E\bbb{(e_k^\top X_{j}^n)^4} \le \sum_{j=1}^n \E\bbb{b_k^4 \norm{A_j-\Sigma}^4} \le b_k^4 \mathcal{M}_4^4 n. \label{eq:en1_error_decomposition_2}
}
Combining equations~\eqref{eq:en1_error_decomposition},~\eqref{eq:en1_error_decomposition_1}, and~\eqref{eq:en1_error_decomposition_2} and using equality~\eqref{eq:en1-Y-relation}, 
\bas{
\Var\bb{(e_k^\top \Eone{n})^2} \leq 2\E\bbb{\bb{(e_k^\top \Eone{n})^2}}^2 + \frac{\eta_n^4}{(1+\eta_n \lambda_1)^4}b_k^4 \mathcal{M}_4^4 n.
}

By Chebyshev's inequality, for any $t > 0$,
\bas{
P\bb{\Abs{\frac{1}{m} \sum_{i=1}^m \bb{e_k^\top \Eone{n}^{(i)}}^2 - \E\bbb{\bb{e_k^\top \Eone{n}}^2}} \ge t} &\le \frac{\var\bb{\bb{e_k^\top \Eone{n}}^2}}{mt^2} \\
&\le \frac{ 2\E\bbb{\bb{(e_k^\top \Eone{n})^2}}^2 + \frac{\eta_n^4}{(1+\eta_n \lambda_1)^4}b_k^4 \mathcal{M}_4^4 n}{mt^2}.\label{eq:en2_inequality}
}
The result follows by setting $t = \frac{\sqrt{2} \E\bbb{\bb{(e_k^\top \Eone{n})^2}} + \eta_n^2 b_k^2 \mathcal{M}_4^2 \sqrt{n}}{\sqrt{m\delta}}$.

\end{proof}

\begin{remark} 
Note that in Lemma~\ref{lemma:en1_concentration_bound}, one can always provide a uniform bound on all elements using a Bernstein-type tail inequality rather than a Chebyshev bound. This is possible because we can use our concentration inequality in Lemma~\ref{lemma:oja_error_hajek_tail_bound}. However, there are two pitfalls of this approach; first, for failure probability $\delta$, the errors of the lower order terms ($\Etwo{n}, \Ethree{n}, \Efour{n}$) still depend polynomially on the $1/\delta$ (see Lemma~\ref{lemma:en2_tail_bound}, \ref{lemma:en3_tail_bound}, \ref{lemma:en4_tail_bound}), which limits the sample complexity of our estimator to have a  $\poly(1/\delta)$ factor, and secondly, Lemma~\ref{lemma:oja_error_hajek_tail_bound} requires a stronger $a.s.$ upper bound on $A_i-\Sigma$ for $i\in [n]$. However, we can get both a uniform bound over all coordinates $k \in [d]$, and a $\log(1/\delta)$ dependence on the sample complexity, using our median of means based algorithm (Algorithm~\ref{alg:variance_estimation}).
\end{remark}
\subsubsection{$\Etwo{n}$ tail bound}\label{appendix:Etwo_bound}



We start by providing a tail bound on higher order terms in the Hoeffding decomposition of $B_{n}-\E\bbb{B_n}$, which may be of independent interest. Let $\mathcal{S}_{n,k}:=\{\{i_1,\dots, i_k\}:1\leq i_1<\dots<i_k\leq n\}$.
    Consider a general product of $n$ matrices, where all but $k$ of the matrices are constant, and $k$ indexed by the subset $S$ are mean zero independent random matrices. With slight abuse of notation, let $M_{S,i}$ denote a constant matrix $M_i$ with $\|M_i\|=:m_i$ when $i\not\in S$ and $W_i$ when $i\in S$, $EW_i=0$, $W_i,i=1,\dots,n$ are mutually independent.
    \begin{align}
    T_{n,k}:=\sum_{S\in \mathcal{S}_{n,k}}\prod_{i=1}^n M_{S,n+1-i}
    \end{align} 
    Let $T_{n,k}$ be a scaled version of the $k^{th}$ term in the Hoeffding projection of the matrix product 
    $B_n:=\prod_{i=1}^n (I+\eta_n A_i)$. Let $W_i=A_i-\Sigma$. We want a tail bound for $\sum_{k \ge 2} T_{n,k}$.

    \begin{lemma}
    \label{lemma:higher-order-helper}
    For $S\in\mathcal{S}_{n,k}$, denote a function $M_{S,i}:=\eta_{n}(A_i-\Sigma)$ when $i\in S$ and $I+\eta_{n}\Sigma$ when $i\not\in S$. Suppose $q \ge 2$ and $\mathcal{M}_q$ are such that $\E\bbb{\norm{A_i-\Sigma}^q}^{1/q} \le \mathcal{M}_q$. Then, for any $1 \le j \le n$ and any $p \ge q$,
    \bas{
    \vertiii{\sum_{k \ge j} T_{n,j}}_{p,q} \leq 2 d^{1/p} (1+\eta_{n} \lambda_1)^n \bb{\frac{\eta_{n} \mathcal{M}_q \sqrt{np}}{1+\eta_{n} \lambda_1}}^j,
    }
    as long as $\frac{2\eta_{n} \mathcal{M}_q \sqrt{np}}{1+\eta_{n} \lambda_1} < 1$.
    \end{lemma}

    \begin{proof}
    We start by deriving a recurrence relation for $T_{n,k}$ as follows:
    \bas{
    T_{n,k} &= \sum_{S \in \mathcal{S}_{n,k}} \prod_{i=1}^n M_{S,n+1-i} \\
    &= \sum_{S \in \mathcal{S}_{n,k}, n \notin S} \prod_{i=1}^n M_{S,n+1-i} +  \sum_{S \in \mathcal{S}_{n,k}, n \in S} \prod_{i=1}^n M_{S,n+1-i} \\
    &= \sum_{S \in \mathcal{S}_{n-1,k}}  (I + \eta_{n} \Sigma) \prod_{i=2}^{n} M_{S,n+1-i} + \sum_{S \in \mathcal{S}_{n-1,k-1}} \eta_{n} (A_n - \Sigma) \prod_{i=2}^{n} M_{S,n+1-i} M_{S,n+1-i} \\
    &=  (I + \eta_{n} \Sigma) \bb{\sum_{S \in \mathcal{S}_{n-1,k}} \prod_{i=1}^{n-1} M_{S,n-i}} + \eta_{n} (A_n - \Sigma) \bb{\sum_{S \in \mathcal{S}_{n-1,k-1}} \prod_{i=1}^{n-1} M_{S,n-i}} \\
    &= (I+\eta_{n} \Sigma)T_{n-1,k} + \eta_{n}(A_n - \Sigma)T_{n-1,k-1}.
    }
    Next, we apply Proposition 4.3. of~\cite{huang2022matrix} to bound $\vertiii{T_{n,k}}_{p,q}$. To apply the proposition, we require $\E\bbb{\eta_{n}(A_n - \Sigma)T_{n-1,k-1} | (I+\eta_{n} \Sigma)T_{n-1,k}} = 0$. Indeed, by independence of $A_1, A_2, \dots, A_n$,
    \bas{
    \E\bbb{\eta_{n}(A_n - \Sigma)T_{n-1,k-1} | (I+\eta_{n} \Sigma)T_{n-1,k}} = \E\bbb{\eta_{n}(A_n - \Sigma)} \E\bbb{T_{n-1,k-1} | (I+\eta_{n} \Sigma)T_{n-1,k}} = 0.
    }
    Therefore, the proposition implies that
    \bas{
    \vertiii{T_{n,k}}_{p,q}^2 \le \vertiii{(I+\eta_{n} \Sigma)T_{n-1,k}}_{p,q}^2 + (p-1) \vertiii{\eta_{n}(A_n - \Sigma)T_{n-1,k-1}}_{p,q}^2.
    }
    From Equation~4.1. and Equation~5.3. of~~\cite{huang2022matrix},
    \bas{
    \vertiii{(I+\eta_{n} \Sigma)T_{n-1,k}}_{p,q} &\le \norm{I+\eta_{n} \Sigma}_{\text{op}} \vertiii{T_{n-1,k}}_{p,q}, \text{ and} \\
    \vertiii{\eta_{n}(A_n - \Sigma)T_{n-1,k-1}}_{p,q} &\le \eta_{n} \E\bbb{\norm{A_n-\Sigma}^q}^{1/q} \vertiii{T_{n-1,k-1}}_{p,q}.
    }
    Plugging these bounds into the recurrence yields
    \bas{
    \vertiii{T_{n,k}}_{p,q}^2 \le (1+\eta_{n} \lambda_1)^2 \vertiii{T_{n,k-1}}_{p,q}^2 + \eta_{n}^2 \mathcal{M}_q^2 (p-1) \E\bbb{\norm{A_n-\Sigma}^q}^{2/q} \vertiii{T_{n-1,k-1}}_{p,q}^2.
    }
    Letting $f_{n,k} \defeq \vertiii{T_{n,k}}_{p,q}^2$, we have the following recurrence for all $n \ge k \ge 1$:
    \bas{
    f_{n,k} \le (1+\eta_{n} \lambda_1)^2 f_{n-1,k} + \eta_{n}^2 \mathcal{M}_q^2 (p-1) f_{n-1,k-1}.
    }
    Defining $a_{n,k} \defeq \frac{f_{n,k}}{(1+\eta_{n} \lambda_1)^{2(n-k)} (\eta_{n}^2 \mathcal{M}_q^2 (p-1))^k}$, we recover an inequality resembling Pascal's identity:
    \bas{
    a_{n,k} \le a_{n-1,k} + a_{n-1,k-1}.
    }
    Moreover, $a_{n,k} = 0$ for all $n < k$ and $a_{n,0} = (1+\eta_{n} \lambda_1)^{-2n} \vertiii{(I+\eta_{n} \Sigma)^{n}}_{p,q}^2 \le d^{2/p}$. Inducting on $n$ and $k$ shows
    \bas{
    a_{n,k} \le d^{2/p} \binom{n}{k}.
    }
    Translating this back to the bound on the norm of $T_{n,k}$, we conclude
    \bas{
    \vertiii{T_{n,k}}_{p,q} \le \sqrt{(1+\eta_{n} \lambda_1)^{2(n-k)} \bb{\eta_{n}^2 \mathcal{M}_q^2 (p-1)}^k d^{2/p} \binom{n}{k}} \le d^{1/p} (1+\eta_{n} \lambda_1)^{n-k} \bb{\eta_{n} \mathcal{M}_q \sqrt{np}}^k 
    }
    Since norms are sub-additive and $\frac{\eta_{n} \mathcal{M}_q \sqrt{np}}{1+\eta_{n} \lambda_1} < \frac{1}{2}$,
    \bas{
    \vertiii{\sum_{k \ge j} T_{n,k}}_{p,q} &\le \sum_{k=j}^n d^{1/p} (1+\eta_{n} \lambda_1)^{n-k} \bb{\eta_{n} \mathcal{M}_q \sqrt{np}}^k \\
    &= d^{1/p} (1+\eta_{n} \lambda_1)^n \sum_{k=j}^n \bb{\frac{\eta_{n} \mathcal{M}_q \sqrt{np}}{1+\eta_{n} \lambda_1}}^k \\
    &\le 2 d^{1/p} (1+\eta_{n} \lambda_1)^n \bb{\frac{\eta_{n} \mathcal{M}_q \sqrt{np}}{1+\eta_{n} \lambda_1}}^j.
    }
    \end{proof}

    
    \begin{lemma}
    \label{lemma:higher-order-norm}
        For $S\in\mathcal{S}_{n,k}$, denote a function $M_{S,i}:=\eta_n(A_i-\Sigma)$ when $i\in S$ and $I+\eta_n\Sigma$ when $i\not\in S$. Then, for any $1 \le j \le n$, and $2 \le q \le 4 \log d$,
    \bas{
    P\bb{\norm{\sum_{k \ge j} T_{n,k}} \ge 
    \frac{3 (1+\eta_n \lambda_1)^n \bb{\eta_n \mathcal{M}_q \sqrt{4n \log d}}^j}{\delta^{\frac{1}{4\log d}}}} \le \delta,
    }
    as long as $4 \eta_n \mathcal{M}_q \sqrt{n \log d} < 1$.
    \end{lemma}
    \begin{proof}

    Let $p = 4 \log d$; note that the assumption $\frac{2 \eta_n \mathcal{M}_q \sqrt{np}}{1+\eta_n \lambda_1} < 1$ holds. By Markov's inequality, Equation~4.2. of~\cite{huang2022matrix}, and Lemma~\ref{lemma:higher-order-helper},
    \bas{
    P\bb{\norm{\sum_{k \ge j} T_{n,k}} \ge (1+\eta_n \lambda_1)^n t} &\le \inf_{p' \ge 2} \bb{(1+\eta_n \lambda_1)^n t}^{-p'} \E\bbb{\norm{\sum_{k \ge j} T_{n,k}}^{p'}} \\
    &\le \inf_{p' \ge 2} \bb{(1+\eta_n \lambda_1)^n t}^{-p'} \E\bbb{\vertiii{\sum_{k \ge j}  T_{n,k}}_{p',q}^{p'}} \\
    &\le \bb{\frac{2 d^{1/p} \bb{\frac{\eta_n \mathcal{M}_q \sqrt{np}}{1+\eta_n \lambda_1}}^j}{t}}^p \le \bb{\frac{3 \bb{\eta_n \mathcal{M}_q \sqrt{4n \log d}}^j}{t}}^{4\log d}.
    }
    for all $t > 0$. The lemma follows by setting $t = 3 \bb{\eta_n \mathcal{M}_q \sqrt{4n \log d}}^j{\delta^{\frac{-1}{4\log d}}}$.

    

    \end{proof}


\begin{lemma}\label{lemma:en2_norm}
Let $\Etwo{n}$ be as defined in Lemma~\ref{lemma:oja_error_decomposition} with $u_0 = g/\norm{g}_{2}$. Then, for any $\delta \in (0,1)$,
\bas{
\Prob\bb{\norm{\Etwo{n}} \le \frac{12\eta_{n}^2 \mathcal{M}_2^2 n \log d} {\sqrt{\delta}}} \geq 1-\delta.
}
\end{lemma}

\begin{proof}
By Lemma~\ref{lemma:higher-order-norm}, with probability at least $1-\delta$,
\bas{
\norm{\sum_{k\geq 2}T_{n,k}} \le \frac{3 (1+\eta_{n} \lambda_1)^n \bb{\eta_{n} \mathcal{M}_2 \sqrt{4n \log d}}^2}{\delta^{\frac{1}{4\log d}}} < \frac{12(1+\eta_{n} \lambda_1)^n \eta_{n}^2 \mathcal{M}_2^2 n \log d} {\sqrt{\delta}}.
}
Conditioned on this event,
\bas{
\norm{\Etwo{n}} &= \frac{\norm{\vp\vp^{\top}(\sum_{k\geq 2}T_{n,k})v_{1}\sign(v_{1}^{\top}u_{0})}}{(1+\eta_{n}\lambda_{1})^{n}} \le \frac{\norm{\vp \vp^{\top}} \norm{\sum_{k\geq 2}T_{n,k}}\norm{v_{1}}}{(1+\eta_{n}\lambda_{1})^{n}} \\
&\le \frac{\norm{\sum_{k\geq 2}T_{n,k}}}{(1+\eta_{n}\lambda_{1})^{n}} \le \frac{12\eta_{n}^2 \mathcal{M}_2^2 n \log d} {\sqrt{\delta}}.
}
\end{proof}

\begin{lemma}\label{lemma:en2_tail_bound} Let $\Etwo{n}$ be defined as in Lemma~\ref{lemma:oja_error_decomposition} for $u_{0} = g/\norm{g}_{2}$ with $g \sim \mathcal{N}(0, \id_d)$. Let $\left\{\Etwo{n}^{(i)}\right\}_{i \in [m]}$ and $\left\{g^{(i)}\right\}_{i \in [m]}$ be $m$ $\iid$ instances of $\Etwo{n}$ and $g$ respectively, and let $\delta \in (0,1)$. Then, with probability at least $1-\delta$,
\bas{
    \frac{\sum_{i \in [m]}\bb{e_{k}^{\top}\Etwo{n}^{(i)}}^{2}}{m} \le \frac{144b_k^2 \eta_{n}^4 \mathcal{M}_2^4 n^2 \log^2 d}{\delta},
}
for all $k \in [d]$, where $b_k := \norm{\vp^{\top}e_k}_{2}$.
\end{lemma}

\begin{proof}
We have
\bas{
\Abs{e_k^\top \Etwo{n}} &= \frac{\Abs{e_k^\top \vp\vp^{\top}(\sum_{k\geq 2}T_{n,k})v_{1}\sign(v_{1}^{\top}u_{0})}}{(1+\eta_{n}\lambda_{1})^{n}} = \frac{\Abs{e_k^\top \vp\vp^{\top} \vp \vp^{\top} (\sum_{k\geq 2}T_{n,k})v_{1}\sign(v_{1}^{\top}u_{0})}}{(1+\eta_{n}\lambda_{1})^{n}} \\
&= \Abs{e_k^{\top} \vp^{\top} \vp \Etwo{n}} \le \norm{e_k^\top \vp} \norm{\Etwo{n}} \le \frac{b_k \norm{\sum_{k\geq 2}T_{n,k}}}{(1+\eta_{n}\lambda_{1})^{n}}.
}
By Lemma~\ref{lemma:en2_norm}, for each $i \in [m]$, with probability at least $1-\frac{\delta}{m}$,
\bas{
\Abs{e_k^\top \Etwo{n}^{(i)}} \le \frac{12b_k \eta_{n}^2 \mathcal{M}_2^2 n \log d} {\sqrt{\delta/m}}.
}
By a union bound, the above holds for all $i \in [m]$ with probability at least $1-\delta$. Under this event,
\bas{
\frac{\sum_{i \in [m]}\bb{e_{k}^{\top}\Etwo{n}^{(i)}}^{2}}{m} \le \frac{\sum_{i \in [m]} \bb{\frac{12b_k \eta_{n}^2 \mathcal{M}_2^2 n \log d} {\sqrt{\delta/m}}}^2}{m} = \frac{144b_k^2 \eta_{n}^4 \mathcal{M}_2^4 n^2 \log^2 d}{\delta}.
}
\end{proof}

\subsubsection{$\Ethree{n}$ tail bound}
\label{appendix:Ethree_bound}

\begin{lemma}\label{lemma:en3_norm}
Let $\Ethree{n}$ be as defined in Lemma~\ref{lemma:oja_error_decomposition} with $u_0 = g/\norm{g}_{2}$. Let $\eta_n$ be set according to Lemma~\ref{lemma:learning_rate_choice}. Fix $\delta \in \bb{0,1}$. Then for any $\epsilon > 0$
we have with probability at least $1-\delta$,
\bas{
    \norm{\Ethree{n}}_{2} \lesssim \sqrt{s_n} \bb{ \frac{ d\exp\bb{-2\eta_{n}n\bb{\lambda_{1}-\lambda_{2}} + \eta_{n}^{2}n\bb{\lambda_{1}^{2}+\Mtwo^{2}}} +  \frac{\eta_{n}\Mtwo^2}{\bb{\lambda_{1}-\lambda_{2}}  } }{\delta^3 (1-\delta) \log^{-1}(1/\delta)}}^{\frac{1}{2}} + \sqrt{s_n} \frac{\eta_n\sqrt{n}\Mtwo\log\bb{d}}{\delta^{\frac{1}{2}}}.
}
where $s_n := \frac{C\log\bb{\frac{1}{\delta}}}{\delta^{3}}\frac{\eta_n \Mtwo^{2}}{\bb{\eigengap}}$ for a universal constant $C > 0$.
\end{lemma}
\begin{proof}
Let $c_n=(1+\eta_{n}\lambda_1)^n |u_0^Tv_1|$. We first note that 
\ba{
    \norm{\Ethree{n}}_2 &= \norm{\vp\vp^{\top}B_{n}u_{0}\bb{\frac{1}{\norm{B_{n}u_0}_{2}} - \frac{1}{c_{n}}}}_2 = \norm{\frac{\vp\vp^{\top}B_{n}u_{0}}{\norm{B_{n}u_0}_{2}}\bb{1 - \frac{\norm{B_{n}u_0}_{2}}{c_{n}}}}_2 \notag \\
    &\leq \norm{\frac{\vp\vp^{\top}B_{n}u_{0}}{\norm{B_{n}u_0}_{2}}}_{2}\Abs{\frac{\norm{B_{n}u_0}_{2}}{c_{n}} - 1}. \label{eq:en3_error_decomposition_1}
}
We bound each of the two multiplicands separately. The first term corresponds to the $\sin$ error between $\voja$ and $v_1$:
\bas{
    \norm{\frac{\vp\vp^{\top}B_{n}u_{0}}{\norm{B_{n}u_0}_{2}}}_{2}^{2} = 1 - \frac{\bb{v_{1}^{\top}B_{n}u_0}^{2}}{\norm{B_{n}u_0}_{2}^{2}} = \sin^{2}\bb{\voja, v_{1}}.
} 
By Corollary 1 of \cite{lunde2021bootstrapping}, 
\ba{
\Prob\bb{\norm{\frac{\vp\vp^{\top}B_{n}u_{0}}{\norm{B_{n}u_0}_{2}}}_{2}^{2} > s_n} = \Prob\bb{\sin^{2}\bb{\voja, v_{1}} > s_n} \leq \delta. \label{eq:sinsquared_scaling}
}

It follows that for any $\eps > 0$,
\ba{
 \Prob\bb{\norm{\Ethree{n}}_{2} > \epsilon\sqrt{s_n}} & \leq \Prob\bb{\norm{\frac{\vp\vp^{\top}B_{n}u_{0}}{\norm{B_{n}u_0}_{2}}}_{2}^{2} > s_n} +\Prob\bb{\Abs{\frac{\norm{B_{n}u_0}_{2}}{c_{n}} - 1} > \epsilon} \\
 &\leq \delta + \Prob\bb{\Abs{\frac{\norm{B_{n}u_0}_{2}}{c_{n}} - 1} > \epsilon}. 
 \label{eq:en3_union_bound}
}
To bound the second term, we adapt the proof of Lemma B.2 in~\cite{lunde2021bootstrapping}. Letting $a_1=\Abs{v_1^{\top}u_0}$,
\begin{align}
 \left|\frac{\|B_nu_0\|}{c_n}-1\right|  & \leq  \left|\frac{\|B_n v_1 a_1\| - \|a_1 (I+\eta_{n}\Sigma)^n v_1\|}{c_n} \right| +   \frac{\|B_n V_\perp V_\perp^T u_0 \|}{c_n} \notag \\
 &= \left|\frac{\|B_n v_1\| - \| (I+\eta_{n}\Sigma)^n v_1\|}{(1+\eta_{n} \lambda_1)^{n}} \right| + \frac{\|B_n V_\perp V_\perp^T u_0 \|}{c_n} \notag \\
 &\le \frac{\normop{B_n - \E [B_n]}}{(1+\eta_n \lambda_1)^n} + \frac{\|B_n V_\perp V_\perp^T u_0 \|}{c_n}. 
 \label{eq:en3_error_decomposition_2} 
\end{align}
For the first summand, using Eq 5.6 of~\cite{huang2022matrix} with $q=2$ and by Markov's inequality, 
\begin{align}
\Prob\left( \frac{\normop{B_n - \E\bbb {B_n}}}{ (1+\eta_n\lambda_1/n)^{n}}  > \frac{\epsilon}{2} \right) \leq   \ \frac{\E\bbb{\normop{B_n - \E\bbb{B_n}}^2}}{(1+\eta_{n}\lambda_1)^{n}\epsilon^2} \leq \frac{C\eta_n^2 n\Mtwo^{2}(1+\log d)^2}{\epsilon^2} \label{eq:en3_bound_first_summand}
\end{align}

For the second summand of equation~\eqref{eq:en3_error_decomposition_2}, 
define the event
\bas{
\mathcal{G}=\left\{\frac{\|B_n \vp\vp^T u_0\|^2}{|v_1^T u_0|^2}\leq \frac{C\log (1/\delta)}{\delta^2}\tr{\vp^T B_n^T B_n \vp}\right\}.
}
By Proposition B.6 of~\cite{lunde2021bootstrapping}, $P(\mathcal{G}) \ge 1-\delta$ where $C > 0$ is some universal constant. Since $P(A|B)P(B) = P(A\cap B) \leq P(A)$,
Markov's inequality 
together with Lemma 5.2 of \cite{jain2016streaming} with $\mathcal{V} \leq \Mtwo^{2}$ yields


\begin{align}
 &\Prob\bb{\frac{\|B_n \vp\vp^T u_0\|}{c_n}\geq \frac{\eps}{2}|\mathcal{G}}\\
 &\leq 
 \frac{1}{1-\delta}\Prob\bb{\mathrm{trace}(V_\perp B_n^T B_n V_\perp^T)  \geq \frac{\epsilon^2}{4} \cdot \frac{\delta^{2}}{C\log (1/\delta)} } \notag \\
 & \leq \frac{1}{1-\delta}C\frac{ d\exp\bb{-2\eta_{n}n\bb{\lambda_{1}-\lambda_{2}} + \eta_{n}^{2}n\bb{\lambda_{1}^{2}+\Mtwo^{2}}} +  \frac{\eta_{n}\Mtwo^2\exp(n\eta_n^2(2\lambda_1^2+\Mtwo^2))}{2\bb{\lambda_{1}-\lambda_{2}}  }  }{ \epsilon^2 \delta^{2}\log^{-1}\bb{1/\delta}  } \\
 & \leq \frac{1}{1-\delta}C\frac{ d\exp\bb{-2\eta_{n}n\bb{\lambda_{1}-\lambda_{2}} + \eta_{n}^{2}n\bb{\lambda_{1}^{2}+\Mtwo^{2}}} +  \frac{e \eta_{n}\Mtwo^2 }{2\bb{\lambda_{1}-\lambda_{2}}}  }{ \epsilon^2 \delta^{2}\log^{-1}\bb{1/\delta}},
 \label{eq:en3_bound_second_summand}
\end{align}
where the last bound follows from Lemma~\ref{lemma:learning_rate_choice}.




Finally, define the error $\eps$ as
\ba{
    \epsilon := \bb{C \frac{ d\exp\bb{-2\eta_{n}n\bb{\lambda_{1}-\lambda_{2}} + \eta_{n}^{2}n\bb{\lambda_{1}^{2}+\Mtwo^{2}}} +  \frac{\eta_{n}\Mtwo^2}{\bb{\lambda_{1}-\lambda_{2}}  }}{\delta^3 (1-\delta) \log^{-1}(1/\delta)}}^{\frac{1}{2}} + \frac{\eta_n\sqrt{n}\Mtwo\log\bb{d}}{\delta^{\frac{1}{2}}}. \label{eq:choice_of_epsilon}
}
Substituting $\epsilon$ in equations~\eqref{eq:en3_bound_second_summand} and~\eqref{eq:en3_bound_first_summand}, and combining with equation~\eqref{eq:en3_error_decomposition_2},
\ba{
    \Prob\bb{\left|\frac{\|B_nu_0\|}{c_n}-1\right| > \epsilon} &\leq \Prob\bb{\frac{\|B_n \vp\vp^T u_0\|}{c_n} >  \frac{\eps}{2}} +  \Prob\bb{ \frac{\normop{B_n - \E\bbb {B_n}}}{ (1+\eta_n\lambda_1/n)^{n}}  > \frac{\epsilon}{2}} \\
    &\leq \Prob\bb{\frac{\|B_n \vp\vp^T u_0\|}{c_n} >  \frac{\eps}{2} | \mathcal{G}} + \Prob(\mathcal{G}^{\complement}) +  \Prob\bb{ \frac{\normop{B_n - \E\bbb {B_n}}}{ (1+\eta_n\lambda_1/n)^{n}}  > \frac{\epsilon}{2}} \le 3\delta.
    \label{eq:en3_decomposition_term_1_bound}
}
From equations~\eqref{eq:en3_union_bound} and~\eqref{eq:en3_decomposition_term_1_bound}, we conclude
\bas{
\Prob\bb{\norm{\Ethree{n}}_{2} > \epsilon\sqrt{s_n}} \le 4\delta.
}
\end{proof}

\begin{lemma}\label{lemma:en3_tail_bound} Let $\Ethree{n}$ be defined as in Lemma~\ref{lemma:oja_error_decomposition} for $u_{0} = g/\norm{g}_{2}$ with $g \sim \mathcal{N}(0, \id_d)$. Let $\eta_n$ be set according to Lemma~\ref{lemma:learning_rate_choice}. Let $\left\{\Ethree{n}^{(i)}\right\}_{i \in [m]}$ and $\left\{g^{(i)}\right\}_{i \in [m]}$ be $m$ $\iid$ instances of $\Ethree{n}$ and $g$ respectively. Then for any $\delta \in \bb{0,1}$, with probability at least $1-\delta$,
\bas{
    & \frac{\sum_{i \in [m]}\bb{e_{k}^{\top}\Ethree{n}^{(i)}}^{2}}{m} \\
    & \lesssim s_nb_{k}^{2}\bb{m^{3}\bb{ \frac{ d\exp\bb{-2\eta_{n}n\bb{\lambda_{1}-\lambda_{2}} + \eta_{n}^{2}n\bb{\lambda_{1}^{2}+\Mtwo^{2}}} +  \frac{\eta_{n}\Mtwo^2}{\bb{\lambda_{1}-\lambda_{2}}  } }{\delta^3 (1-\delta/m) \log^{-1}(m/\delta)}} + m\frac{\eta_n^{2}n\Mtwo^{2}\log^{2}\bb{d}}{\delta}}.
}
for all $k \in [d]$, where $b_k := \norm{\vp^{\top}e_{k}}_{2}$ and $s_n := \frac{C\log\bb{\frac{1}{\delta}}}{\delta^{3}}\frac{\eta_n \Mtwo^{2}}{\bb{\eigengap}}$ for a universal constant $C > 0$.
\end{lemma}
\begin{proof}
Using Lemma~\ref{lemma:en3_norm}, for any fixed $i \in [m]$, with probability at least $1-\delta$,
\ba{
    \norm{\Ethree{n}^{(i)}}_{2} \lesssim \sqrt{s_n}\bb{\bb{ \frac{ d\exp\bb{-2\eta_{n}n\bb{\lambda_{1}-\lambda_{2}} + \eta_{n}^{2}n\bb{\lambda_{1}^{2}+\Mtwo^{2}}} +  \frac{\eta_{n}\Mtwo^2}{\bb{\lambda_{1}-\lambda_{2}}  } }{\delta^3 (1-\delta) \log^{-1}(1/\delta)}}^{\frac{1}{2}} + \frac{\eta_n\sqrt{n}\Mtwo\log\bb{d}}{\delta^{\frac{1}{2}}}}. \label{eq:en3_norm_i_bound}
}
Furthermore, note that 
\ba{
\Abs{e_{k}^{\top}\Ethree{n}^{(i)}}_{2} &= \Abs{e_{k}^{\top}\vp\vp^{\top}B_{n}u_{0}\bb{\frac{1}{\norm{B_{n}u_0}_{2}} - \frac{1}{c_{n}}}}_{2} \notag \\
&= \Abs{e_{k}^{\top}\vp\vp^{\top}\vp\vp^{\top}B_{n}u_{0}\bb{\frac{1}{\norm{B_{n}u_0}_{2}} - \frac{1}{c_{n}}}}_{2} \notag \\
&\leq \norm{e_{k}^{\top}\vp\vp^{\top}}_{2}\norm{\vp\vp^{\top}B_{n}u_{0}\bb{\frac{1}{\norm{B_{n}u_0}_{2}} - \frac{1}{c_{n}}}}_{2} \notag \\
&= b_k\norm{\vp\vp^{\top}B_{n}u_{0}\bb{\frac{1}{\norm{B_{n}u_0}_{2}} - \frac{1}{c_{n}}}}_{2} = b_{k}\norm{\Ethree{n}^{(i)}}_{2} \label{eq:en3_ek_to_norm_conversion}
}
The result then follows by a union bound over all $i \in [m]$ for the event in \eqref{eq:en3_norm_i_bound} and using \eqref{eq:en3_ek_to_norm_conversion}.

\end{proof}

\subsubsection{$\Efour{n}$ tail bound}
\label{appendix:Efour_bound}

\begin{lemma}\label{lemma:en4_norm} Let $\Efour{n}$ be defined as in Lemma~\ref{lemma:oja_error_decomposition} for $u_{0} = g/\norm{g}_{2}$ with $g \sim \mathcal{N}(0, \id_d)$. Let $\eta_n$ be set according to Lemma~\ref{lemma:learning_rate_choice}. For any $\delta \in \bb{0,1}$, with probability at least $1-\delta$,
\bas{
    \norm{\Efour{n}} \leq \frac{1}{\delta^{3/2}}\bb{d\exp\bb{-2\eta_{n}n\bb{\lambda_{1}-\lambda_{2}} + \eta_{n}^{2}n\bb{\lambda_{1}^{2}+\Mtwo^{2}}} + \frac{e\eta_{n}^{3}n\Mtwo^{4}\bb{1+2\log\bb{d}}}{2\bb{\lambda_{1}-\lambda_{2}} + \eta_{n} \bb{\lambda_{1}^{2}-\lambda_{2}^{2}-\Mtwo^{2}}}}^{1/2}.
}
\end{lemma}

\begin{proof}
Recall that
\bas{
\norm{\Efour{n}} &= \frac{\norm{\vp \vp^{\top} B_{n} \vp \vp^{\top} u_0}}{|v_1^{\top} u_0| (1+\eta_n \lambda_1)^n} =  \frac{\norm{\vp \vp^{\top} B_{n} \vp \vp^{\top} g}}{|v_1^{\top} g| (1+\eta_n \lambda_1)^n}.
}
To bound this quantity, we will bound its square instead. Using Markov's inequality, with probability at least $1-\delta/2$, 
\bas{
\norm{\vp \vp^{\top} B_{n} \vp \vp^{\top} g}^2 
&\le \frac{2}{\delta} E\bbb{\norm{\vp \vp^{\top} B_{n} \vp \vp^{\top} g}^2} \\
&= \frac{2}{\delta} \Tr\bb{E\bbb{\bb{\vp \vp^{\top} B_{n} \vp \vp^{\top} g}\bb{\vp \vp^{\top} B_{n} \vp \vp^{\top} g}^{\top}}} \\
&= \frac{2}{\delta} \E\bbb{\Tr\bb{\vp^{\top}B_{n}\vp \vp^{\top}B_{n}^{\top}\vp}}.
}
By Lemma B.3 of \cite{lunde2021bootstrapping},
\bas{\frac{\E\bbb{\Tr\bb{\vp^{\top}B_{n}\vp\vp^{\top}B_{n}^{\top}\vp}}}{(1+\eta_{n}\lambda_{1})^{2n}} &\leq d\exp\bb{-2\eta_{n}n\bb{\lambda_{1}-\lambda_{2}} + \eta_{n}^{2}n\bb{\lambda_{1}^{2}+\Mtwo^{2}}} + \frac{e\eta_{n}^{3}n\Mtwo^{4}\bb{1+2\log\bb{d}}}{2\bb{\lambda_{1}-\lambda_{2}} + \eta_{n} \bb{\lambda_{1}^{2}-\lambda_{2}^{2}-\Mtwo^{2}}}. 
}
Also, with probability at least $1-\delta/2$, $|v_1^{\top} g| \ge \delta/2$ (see Proposition 7 from \cite{lunde2021bootstrapping} for anticoncentration of gaussians). Combining the two bounds yields the result.
\end{proof}

\begin{lemma}\label{lemma:en4_tail_bound} Let $\Efour{n}$ be defined as in Lemma~\ref{lemma:oja_error_decomposition} for $u_{0} = g/\norm{g}_{2}$ with $g \sim \mathcal{N}(0, \id_d)$. Let $\eta_n$ be set according to Lemma~\ref{lemma:learning_rate_choice}. Let $\left\{\Efour{n}^{(i)}\right\}_{i \in [m]}$ and $\left\{g^{(i)}\right\}_{i \in [m]}$ be $m$ $\iid$ instances of $\Efour{n}$ and $g$ respectively. Fix $\delta \in \bb{0,1}$. Then, conditioned on $\mathcal{E}$, with probability at least $1-\delta$, 
\bas{
    &\frac{\sum_{i \in [m]}\bb{e_{k}^{\top}\Efour{n}^{(i)}}^{2}}{m} \\
    &\quad\quad\quad \leq \frac{b_{k}^{2}m^{2}}{\delta^{3}(1-\delta)}\bb{d\exp\bb{-2\eta_{n}n\bb{\lambda_{1}-\lambda_{2}} + \eta_{n}^{2}n\bb{\lambda_{1}^{2}+\Mtwo^{2}}} + \frac{e\eta_{n}^{3}n\Mtwo^{4}\bb{1+2\log\bb{d}}}{2\bb{\lambda_{1}-\lambda_{2}} + \eta_{n} \bb{\lambda_{1}^{2}-\lambda_{2}^{2}-\Mtwo^{2}}}}
}
for all $k \in [d]$, where $b_k := \norm{\vp^{\top}e_k}_{2}$.
\end{lemma}
\begin{proof}
Note that 
\bas{
\bb{e_{k}^{\top}\Efour{n}}^{2} \le \norm{\vp^{\top}e_k}_{2}^{2} \underbrace{\bb{\frac{\norm{\vp^{\top}B_{n}\vp\vp^{\top}u_{0}}_{2}}{\Abs{v_{1}^{\top}u_{0}}(1+\eta_{n}\lambda_{1})^{n}}}^{2}}_{\Phi_{n}} 
}
Let $\Phi_{n}^{(i)}$ correspond to the $i^{\text{th}}$ instance of the random variable $\Phi_{n}$.
Then, for any $k \in [d]$, 
    \ba{
        \frac{1}{m}\sum_{i \in [m]}\bb{e_{k}^{\top}\Efour{n}^{(i)}}^{2} &\leq \frac{\norm{\vp^{\top}e_k}_{2}^{2}}{m}\sum_{i \in [m]} \Phi_{n}^{(i)}. \label{eq:en4_cs}
    }
    Define the event $\mathcal{E} := \left\{|v_{1}^{\top}g| \geq \frac{\delta}{m}\right\}$ and let $\mathcal{E}^{(i)}$, $i\in [m]$ be the $i^{th}$ instance of this event. First, observe that:
    \ba{
    \E[\Phi_{n} | \mathcal{E}] = \E\bbb{\bb{\frac{\norm{\vp^{\top}B_{n}\vp\vp^{\top}u_{0}}_{2}}{\Abs{v_{1}^{\top}u_{0}}(1+\eta_{n}\lambda_{1})^{n}}}^{2}\bigg|\mathcal{E}}
    \notag &= \E\bbb{\frac{\vp^{\top}B_{n}\vp\vp^{\top}gg^{\top}\vp\vp^{\top}B_{n}^{\top}\vp}{\bb{v_{1}^{\top}g}^{2}(1+\eta_{n}\lambda_{1})^{2n}}\bigg|\mathcal{E}} \notag \\
       &\leq \frac{m^{2}}{\delta^2(1+\eta_{n}\lambda_{1})^{2n}}\E\bbb{\vp^{\top}B_{n}\vp\vp^{\top}gg^{\top}\vp\vp^{\top}B_{n}^{\top}\vp\bigg|\mathcal{E}} \notag \\
       &\leq \frac{m^{2}}{\delta^{2}\mathbb{P}(\mathcal{E})}\frac{\E\bbb{\Tr\bb{\vp^{\top}B_{n}\vp\vp^{\top}B_{n}^{\top}\vp}}}{(1+\eta_{n}\lambda_{1})^{2n}} \label{eq:en4_exp}
    }
    Now, using Markov's inequality conditioned on $\bigcap_{i \in [m]}\mathcal{E}^{(i)}$, we have with probability at least $1-\Prob(\bigcap_{i \in [m]}\mathcal{E}^{(i)})$, 
     \ba{\frac{1}{m}\sum_{i \in [m]}\Phi_{n}^{(i)}&\leq \frac{1}{\delta}\E\bbb{\Phi_n^{(i)}\bigg|\bigcap_{j \in [m]}\mathcal{E}^{(j)}} \notag \\
     \mbox{(By $\iid$ nature of the instances)}&=\frac{1}{\delta}\E\bbb{\Phi_n^{(i)}\bigg|\mathcal{E}^{(i)}} 
     = \frac{1}{\delta}\E[\Phi_n | \mathcal{E}] \notag \\
       &\leq \frac{m^{2}}{\delta^{3}\mathbb{P}(\mathcal{E})}\frac{\E\bbb{\Tr\bb{\vp^{\top}B_{n}\vp\vp^{\top}B_{n}^{\top}\vp}}}{(1+\eta_{n}\lambda_{1})^{2n}} \label{eq:en4_markov}
    }
    The last step uses Eq~\ref{eq:en4_exp}.
    Using Lemma B.3 from \cite{lunde2021bootstrapping}, we have
    \ba{\frac{\E\bbb{\Tr\bb{\vp^{\top}B_{n}\vp\vp^{\top}B_{n}^{\top}\vp}}}{(1+\eta_{n}\lambda_{1})^{2n}} &\leq d\exp\bb{-2\eta_{n}n\bb{\lambda_{1}-\lambda_{2}} + \eta_{n}^{2}n\bb{\lambda_{1}^{2}+\Mtwo^{2}}} + \frac{e\eta_{n}^{3}n\Mtwo^{4}\bb{1+2\log\bb{d}}}{2\bb{\lambda_{1}-\lambda_{2}} + \eta_{n} \bb{\lambda_{1}^{2}-\lambda_{2}^{2}-\Mtwo^{2}}} \label{eq:en4_expectation_bound}
    }
    Finally, we note that using Proposition 7 from \cite{lunde2021bootstrapping}, we have 
    \ba{
        \forall i \in [m], \Prob\bb{\mathcal{E}^{(i)}} \geq 1 - \frac{\delta}{m} \implies  \mathbb{P}\bb{\bb{\bigcap_{i \in [m]}\mathcal{E}^{(i)}}^{\complement}} \leq \sum_{i \in [m]}\Prob\bb{\mathcal{E}_i^{\complement}} \leq \sum_{i \in [m]}\frac{\delta}{m} = \delta \label{eq:good_event_union_bound}
    }
    The result follows by substituting \eqref{eq:en4_expectation_bound} in \eqref{eq:en4_markov} and then using \eqref{eq:en4_cs}, along with the union-bound provided in \eqref{eq:good_event_union_bound}.
\end{proof}

\subsubsection{Total Variance Bound}\label{sub_appendix:uncertainty}
We now put together the results from Lemmas~\ref{lemma:en0_tail_bound},~\ref{lemma:en1_concentration_bound},~\ref{lemma:en2_tail_bound},~\ref{lemma:en3_tail_bound}, and~\ref{lemma:en4_tail_bound} to provide a high probability bound on the error of the variance estimator Algorithm~\ref{alg:variance_estimation}. 

\begin{figure}
    \centering   \includegraphics[width=0.5\linewidth]{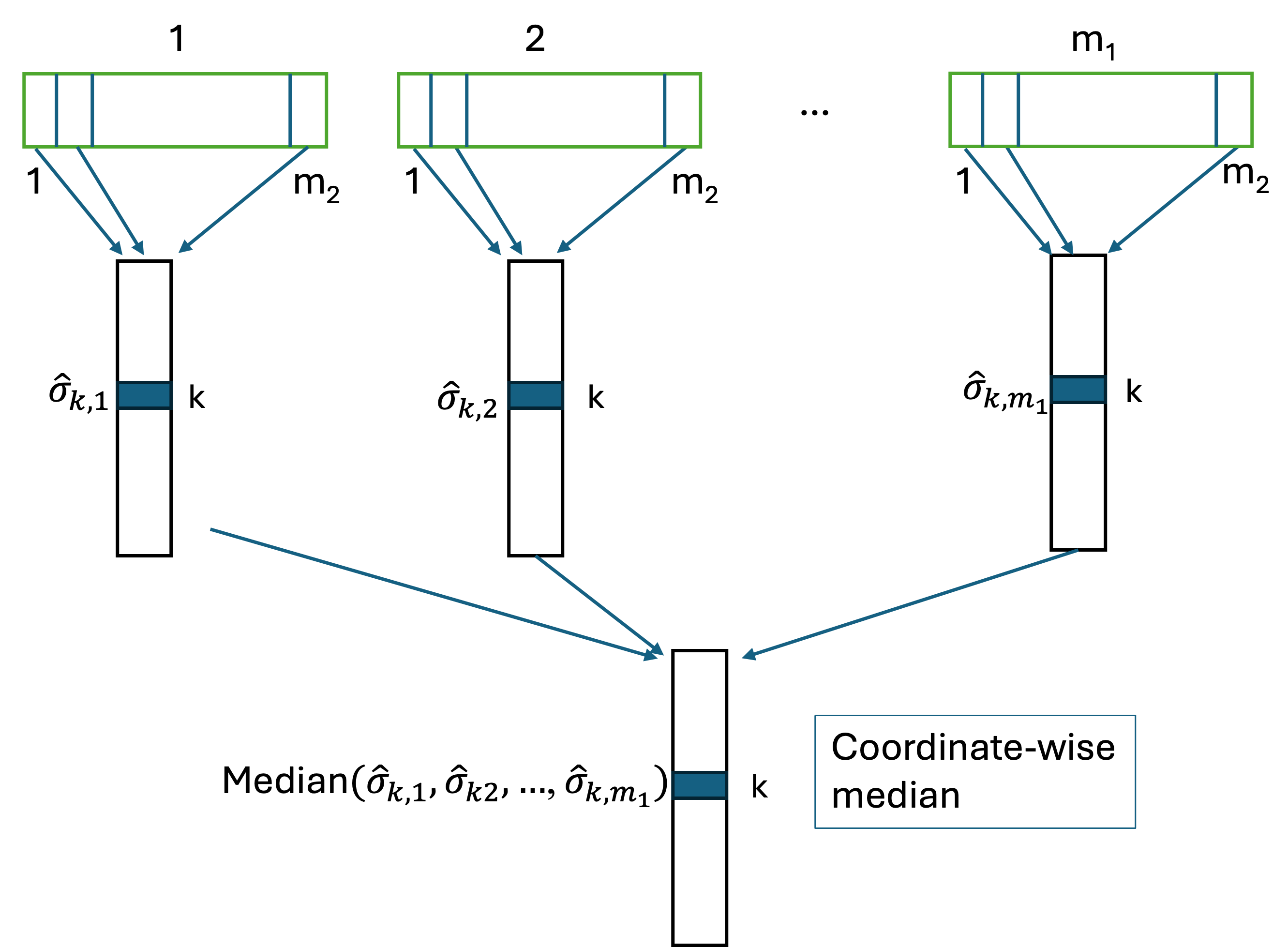}
    \caption{Schematic picture of Algorithm~\ref{alg:variance_estimation}}
    \label{fig:subsampling}
\end{figure}

Figure~\ref{fig:subsampling} summarizes how the variance estimation algorithm works. The algorithm first computes an Oja vector $\vmain$ using $N$ samples. Then, $n$ samples are divided into $m_1$ batches, with each batch containing $n/m_1$ samples. These $n$ samples need not be disjoint from the $N$ samples used to compute the high-accuracy estimate $\vmain$. Then, the ${\ell}^{\text{th}}$ batch of $n/m_1$ samples is split into $m = m_2$ batches of size $B \defeq n/m_1m_2$ each. Oja vectors $\left\{\hat{v}_{j}\right\}_{j \in [m_2]}$ are computed on each of these $m_2$ batches, and
\ba{
    \hat{\sigma}^{2}_{k, \ell} := \sum_{j \in [m_2]} \dfrac{\bb{e_k^{\top} \bb{\hat{v}_j - (\vmain^\top \hat{v}_j)\vmain}}^2}{m_2}. \label{eq:def_sigma_hat_ell}
}
for all $k \in [d]$. The overall estimate for the variance of the $k^{\text{th}}$ coordinate is $\mathsf{Median}\bb{\left\{\hat{\sigma}_{k, \ell}\right\}_{\ell \in [m_1]}}$. Since this variance scales with the inverse of the learning parameter $\eta_B$, we define the scale-free $\hat{\gamma}_k \defeq \mathsf{Median}\bb{\left\{\hat{\sigma}_{k, \ell}\right\}_{\ell \in [m_1]}}/(\eta_B \bb{\eigengap})$.
For each $k \in [d]$, define the quantities
\bas{
    b_k := \norm{e_k^{\top}\vp}, \;\;\;\; c_k := \sqrt{\frac{\E\bbb{\bb{e_k^{\top}\Eone{B}}^{2}}}{\eta_B}\frac{\eigengap}{\Mtwo^{2}}}.
}
Under this setting, we show that each $\hat{\sigma}_{k, \ell}^2$ approximates the true variance with at least $3/4$ probability. We assume that the learning rate $\eta_B$ satisfies
\ba{
\eta_B \le \frac{1}{2\lambda_1} + \frac{\eigengap}{2 \Mtwo^2}. \label{eta_B_upper_bound}
}
It can be verified that this assumption is satisfied by the bounds on $B$ provided in \eqref{eq:n_lower_bound}.

Let $\vmain$ be the high probability estimate as defined in~\ref{def:oja_high_prob}. For the remainder of the proof, we condition on the event 
\ba{
    \mathcal{E} := \left\{\sin^2 \bb{\vmain, v_1} \le \frac{C\log\bb{\frac{1}{\delta}}\log\bb{N/\log\bb{\frac{1}{\delta}}}\Mtwo^{2}}{N(\eigengap)^{2}}\right\} \label{def:good_event_vmain}
}
for some universal constant $C > 0$, and assume this event $\mathcal{E}$ holds with probability at least $1-\delta$. 

\begin{lemma}\label{lemma:constant_prob_error_bound_all}
For $\delta \in \bb{0,1}$ and any $\ell \in [m_1]$, under assumption~\ref{eta_B_upper_bound}, with probability at least $3/4$,
\ba{
    \Abs{\hat{\sigma}_{k,\ell}^2 - \eta_B \bb{\eigengap} e_k^\top \V e_k} &\le  8\bb{\frac{1}{\sqrt{m}} + \frac{2}{m}}\eta_B \bb{\eigengap} e_k^\top \V e_k \notag \\ & + O\bb{\frac{b_k^2 \log^2 B}{B^{3/2} m^{1/2}} \bb{\frac{\Mfour}{\eigengap}}^2 + \frac{\log\bb{\frac{1}{\delta}}\log\bb{N/\log\bb{\frac{1}{\delta}}}}{N} \bb{\frac{\Mtwo}{\eigengap}}^2} \notag \\
    &+ O\bb{\frac{b_k^2 m^2 \log^2 d \log^4 B}{B^2}  \bb{\frac{\Mtwo}{\eigengap}}^4 + \frac{\lambda_1 \Mtwo^2 \log^2 B}{B^2 \bb{\eigengap}^3}}. \label{eq:total_error_bound_all_coord}
}
\end{lemma}
\begin{proof}
Drop the index $\ell$ for convenience of notation. Let $\delta_0 \defeq 1/20$. By triangle inequality,
\ba{
\Abs{\hat{\sigma}_{k}^2 - \eta_B \bb{\eigengap} e_k^\top \V e_k} &\le \Abs{\hat{\sigma}_{k,\ell}^2 - \E\bbb{\bb{e_{k}^{\top}\Eone{B}}^{2}}} +  \Abs{\E\bbb{\bb{e_{k}^{\top}\Eone{B}}^{2}}- \eta_B \bb{\eigengap} e_k^\top \V e_k} \label{eq:triangle_ineq}
}
and by Lemma~\ref{lemma:second_moment_matrix},
\ba{
\Abs{\E\bbb{\bb{e_{k}^{\top}\Eone{B}}^{2}}- \eta_B \bb{\eigengap} e_k^\top \V e_k} \le \frac{\eta_B^2 \Mtwo^2 \lambda_1}{\eigengap} \lesssim \frac{\lambda_1 \Mtwo^2 \log^2 B}{B^2 \bb{\eigengap}^3}. \label{eq:true_hajek_close}
}
By equation~\eqref{eq:ojadecomp} and Lemma~\ref{lemma:square_expansion_cs}, for any $\eps \in (0,1)$,
    \ba{
    \Abs{\hat{\sigma}_{k}^{2} - \E\bbb{\bb{e_k^{\top} \Eone{B}}^2}} &\le (1+\eps) \Abs{\frac{\sum_{j \in [m]} \bb{e_k^{\top} \Eone{B}^{(j)}}^2}{m} - \E\bbb{\bb{e_k^{\top} \Eone{B}}^2}} + \eps \E\bbb{\bb{e_k^{\top} \Eone{B}}^2} \nonumber \\
    &+ \underbrace{\frac{8}{\eps} \sum_{j \in [m]} \frac{\bb{e_k^{\top} \Ezero{B}^{(j)}}^2 + \bb{e_k^{\top} \Etwo{B}^{(j)}}^2 + \bb{e_k^{\top} \Ethree{B}^{(j)}}^2 + \bb{e_k^{\top} \Efour{B}^{(j)}}^2}{m}}_{:= \errorsmall}. \label{eq:total_concentration_bound}
    }
Set $\eps = 2/\sqrt{m}$. By Lemmas~\ref{lemma:en0_tail_bound},~\ref{lemma:en2_tail_bound},~\ref{lemma:en3_tail_bound}, and~\ref{lemma:en4_tail_bound}, along with Lemma~\ref{lemma:learning_rate_choice} to bound $nd\exp\bb{-\eta_n n \bb{\eigengap}} = o\bb{1}$, we have with probability at least $1-4\delta_0$ that 
\ba{
    & \frac{\errorsmall}{8/\epsilon} \\ &\lesssim 
     \frac{\log\bb{\frac{1}{\delta}}\log\bb{N/\log\bb{\frac{1}{\delta}}} \Mtwo^2}{\eigengap} + b_k^2 \eta_B^4 \Mtwo^4 B^2 \log^2 d + s_{B}b_k^{2} m \eta_{B}^{2}B \Mtwo^{2}\log^{2} d + \frac{\eta_{B}^{3} B \Mtwo^{4} \log (d) b_k^{2} m^{2}}{2\bb{\eigengap} + \eta_B\bb{\lambda_1^2-\lambda_2^2-\Mtwo^2}} \nonumber\\
    &\lesssim \frac{\log\bb{\frac{1}{\delta}}\log\bb{N/\log\bb{\frac{1}{\delta}}}}{N} \bb{\frac{\Mtwo}{\eigengap}}^2 + \frac{b_k^2 m \log^2 d \log^4 B}{B^2} \bb{\frac{\Mtwo}{\eigengap}}^4 + \frac{b_k^2 m^2 \log d \log^3 B}{B^2}  \bb{\frac{\Mtwo}{\eigengap}}^4. \label{eq:const_prob_small}
}

where we used Assumption~\ref{eta_B_upper_bound} to bound the last term. By Lemma~\ref{lemma:en1_concentration_bound}, with probability $1-\delta_0$, 
    \ba{
    \Abs{\frac{\sum_{j \in [m]} \bb{e_k^{\top} \Eone{B}^{(j)}}^2}{m} - \E\bbb{\bb{e_k^{\top} \Eone{B}}^2}} &\le \frac{\sqrt{2} \E\bbb{(e_k^\top \Eone{B})^2} + \eta_B^2 b_k^2 \mathcal{M}_4^2 \sqrt{B}}{\sqrt{m\delta_0}} \nonumber \\ 
    &\le 4\epsilon \E\bbb{(e_k^\top \Eone{B})^2} + \frac{b_k^2 \log^2 B}{B^{3/2} m^{1/2}} \bb{\frac{\Mfour}{\eigengap}}^2 
    \label{eq:const_prob_hajek}
    }
We now combine equations~\eqref{eq:true_hajek_close}, ~\eqref{eq:total_concentration_bound},~\eqref{eq:const_prob_small}, and~\eqref{eq:const_prob_hajek} in~\eqref{eq:triangle_ineq} to conclude that with probability at least $1-5\delta_0 = 3/4$,
\bas{
& \Abs{\hat{\sigma}_{k,\ell}^2 - \eta_B \bb{\eigengap} e_k^\top \V e_k} \\
& \leq (1+\epsilon) \bb{4\epsilon \eta_B \bb{\eigengap} e_k^\top \V e_k + \frac{b_k^2 \log^2 B}{B^{3/2} m^{1/2}} \bb{\frac{\Mfour}{\eigengap}}^2} + (1+\epsilon)(1+4\epsilon) \frac{\eta_B^2 \Mtwo^2 \lambda_1}{\eigengap} + \errorsmall,
}
which simplifies to the lemma statement.

\end{proof}
Next, assume that the following relations hold:
\begin{gather}
 N \gtrsim \frac{m B\log\bb{\frac{1}{\delta}}}{c_k^2 \log B} \log \bb{\frac{m B}{ c_k^2\log B}}. \label{eq:N_lower_bound}\\
B \gtrsim  m^{3}\bb{\frac{b_k}{c_k}}^{2}\bb{\frac{\Mtwo}{\lambda_1 - \lambda_2}}^{2}\log^{3}\bb{B}\log^{2}\bb{d}. \label{eq:n_lower_bound}\\
B \gtrsim \max\bb{m \bb{\frac{b_k}{c_k}}^4 \bb{\frac{\Mfour}{\Mtwo}}^4 \log^2 B, \frac{m \lambda_1 \log B}{c_k^2 \bb{\eigengap}}}. \label{eq:B_lower_bound}
\end{gather}

These assumptions on $N$ and $B$ \textit{subsume} the assumption on the learning rate $\eta_B$ in equation~\ref{eta_B_upper_bound}.

Using equation~\ref{eq:true_hajek_close} and the relation
\ba{\frac{\E\bbb{\bb{e_k^{\top}\Eone{B}}^{2}}}{m} &= \frac{\eta_{B}c_{k}^{2}}{m}\frac{\Mtwo^{2}}{\eigengap}. \label{eq:Eone_expectation_c_k}
} and comparing it with each term in the smaller order error of Lemma~\ref{lemma:constant_prob_error_bound_all} yields the following Lemma.
\begin{lemma}\label{lemma:constant_prob_helper}
Under assumptions~\eqref{eq:N_lower_bound},~\eqref{eq:n_lower_bound}, and~\eqref{eq:B_lower_bound}, we have the following upper bound on the right hand side of equation~\eqref{eq:total_error_bound_all_coord} in Lemma~\ref{lemma:constant_prob_error_bound_all}.
\ba{
 & \frac{\log\bb{\frac{1}{\delta}}\log\bb{N/\log\bb{\frac{1}{\delta}}}}{N} \bb{\frac{\Mtwo}{\eigengap}}^2 + \frac{b_k^2 \log^2 B}{B^{3/2} m^{1/2}} \bb{\frac{\Mfour}{\eigengap}}^2 + \frac{b_k^2 m^2 \log^2 d \log^4 B}{B^2} \bb{\frac{\Mtwo}{\eigengap}}^4 \notag \\
 &  + \frac{\lambda_1 \Mtwo^2 \log^2 B}{B^2 \bb{\eigengap}^3}  \leq \frac{\eta_B \bb{\eigengap} e_k^\top \V e_k}{m}. \label{eq:total_error_bound_some_coord}
}
\end{lemma}
It follows that a stronger multiplicative guarantee holds for any coordinate $k$ that satisfies the above assumptions:
\begin{lemma}\label{lemma:constant_prob_error_bound_some}
For any coordinate $k$ that satisfies Lemma~\ref{lemma:constant_prob_error_bound_all} and assumptions~\ref{eq:N_lower_bound},~\ref{eq:n_lower_bound}, and ~\ref{eq:B_lower_bound}, 
\bas{
    \Abs{\hat{\sigma}_{k}^2 - \eta_B \bb{\eigengap} e_k^\top \V e_k} \le O\bb{\frac{\eta_B \bb{\eigengap} e_k^\top \V e_k}{\sqrt{m}}}.
}
\end{lemma}
Given a per-coordinate guarantee with success probability $3/4$, we can boost the success probability and give a high-probability guarantee over all coordinates using the median-of-means approach (see Lemma~\ref{lemma:MoM_modified}). Since the vector $\vmain$ is shared between all Oja vectors in the uncertainty algorithm, the estimates $\hat{\sigma}_{k,i}^2$ are not independent for any coordinate $k$. However, the errors due to the terms
$\Eone{n}^{(i)}, \Etwo{n}^{(i)}, \Ethree{n}^{(i)}, \Efour{n}^{(i)}$ are mutually independent because they do not depend on $\vmain$. Moreover, the uncertainty in the estimate due to $\Ezero{n}^{(i)}$ can be accounted for by the sin-squared error of $\vmain$ only, which is bounded with high probability by equation~\ref{eq:vtilde}.

\begin{lemma}\label{lemma:high_prob_error_bound}
Let $\left \{ \hat{\gamma}_k \right \}_{k \in [d]}$ be the output of Algorithm~\ref{alg:variance_estimation}. Under assumption~\eqref{eta_B_upper_bound}, with probability $1-2\delta$, for all $k \in [d]$,
\bas{
    \Abs{\hat{\gamma}_k - \V_{kk}} &\le  8\bb{\frac{1}{\sqrt{m}} + \frac{2}{m}} \V_{kk} + O\bb{\frac{b_k^2 \log B}{\sqrt{mB}} \bb{\frac{\Mfour}{\eigengap}}^2 + \frac{B \log\bb{\frac{1}{\delta}}\log\bb{N/\log\bb{\frac{1}{\delta}}}}{N \log B} \bb{\frac{\Mtwo}{\eigengap}}^2}  \nonumber\\
    &+ O\bb{\frac{b_k^2 m^2 \log^2 d \log^3 B}{B}  \bb{\frac{\Mtwo}{\eigengap}}^4 + \frac{\lambda_1 \Mtwo^2 \log B}{B \bb{\eigengap}^3}}.
}
Moreover, let $K$ be the set of indices in [d] that satisfy assumptions~\eqref{eq:N_lower_bound},~\eqref{eq:n_lower_bound}, and~\eqref{eq:B_lower_bound}. Then, for all $k \in K$,
\bas{
 \Abs{\hat{\gamma}_k - e_k^\top \V e_k} = O\bb{\frac{  \V_{kk} }{\sqrt{m}}}.
}
\end{lemma}

\begin{proof}
Recall the error decomposition equation~\eqref{eq:triangle_ineq} of Lemma~\ref{lemma:constant_prob_error_bound_all}. The first term is bounded in equation~\eqref{eq:true_hajek_close},
while the second term is bounded above as
\bas{
    \Abs{\hat{\sigma}_{k, \ell}^{2} - \E\bbb{\bb{e_k^{\top} \Eone{B}}^2}} &\le (1+\eps) \Abs{\frac{\sum_{j \in [m]} \bb{e_k^{\top} \Eone{B}^{(j)}}^2}{m} - \E\bbb{\bb{e_k^{\top} \Eone{B}}^2}} + \eps \E\bbb{\bb{e_k^{\top} \Eone{B}}^2} \\
    &+ \underbrace{\frac{8}{\eps} \sum_{j \in [m]} \frac{\bb{e_k^{\top} \Ezero{B}^{(j)}}^2 + \bb{e_k^{\top} \Etwo{B}^{(j)}}^2 + \bb{e_k^{\top} \Ethree{B}^{(j)}}^2 + \bb{e_k^{\top} \Efour{B}^{(j)}}^2}{m}}_{:= \errorsmall} = \alpha_{k,\ell} + \beta_{k,\ell},
    }
where $\eps = 2/\sqrt{m}$, $\alpha_{k, \ell} = \dfrac{8}{\eps m} \sum_{j \in [m]} (e_k^{\top} \Ezero{B}^{(j)})^2$ and $\beta_{k,\ell}$ is the sum of the remaining summands. We drop the index $k$ from the subscript presently. 

By assumption~\eqref{def:good_event_vmain} on $\vmain$, the event
\bas{
\mathcal{E} := \left\{\sin^2 \bb{\vmain, v_1} \le \frac{C\log\bb{\frac{1}{\delta}}\log\bb{N/\log\bb{\frac{1}{\delta}}}\Mtwo^{2}}{N(\eigengap)^{2}}\right\}
}
occurs with probability at least $1-\delta$. Under this assumption, Lemma~\ref{lemma:en0_tail_bound} implies 
\bas{
\alpha_{\ell} = 4\sqrt{m} \sum_{i\in [m]}\frac{\bb{e_k^{\top} \Ezero{n}^{(i)}}^2}{m} \lesssim \frac{\sqrt{m} \log\bb{\frac{1}{\delta}}\log\bb{N/\log\bb{\frac{1}{\delta}}}\Mtwo^{2}}{N(\eigengap)^{2}}. 
}
for all $\ell \in [m_1]$. Also, following the proof of Lemma~\ref{lemma:constant_prob_error_bound_all}, with probability at least $3/4$,
\bas{
\beta_{\ell} &\lesssim 8\bb{\frac{1}{\sqrt{m}} + \frac{2}{m}}\eta_B \bb{\eigengap} \V_{kk} + O\bb{\frac{b_k^2 \log^2 B}{B^{3/2} m^{1/2}} \bb{\frac{\Mfour}{\eigengap}}^2} \\
&+ O\bb{\frac{b_k^2 m^2 \log^2 d \log^4 B}{B^2}  \bb{\frac{\Mtwo}{\eigengap}}^4 + \frac{\lambda_1 \Mtwo^2 \log^2 B}{B^2 \bb{\eigengap}^3}}.
}
Also, the quantities $\beta_1, \beta_2, \dots, \beta_{\ell}$ are mutually independent as they do not involve the error term $\Ezero{n}$ or $\vmain$. By Lemma~\ref{lemma:MoM_modified}, the choice $m_1 = 8 \log (d/\delta)$ and conditioned on the event $\mathcal{E}$, the median
\ba{\label{eq:median_sig_bound}
\median\bb{\{\sigma_{k,\ell}^2\}_{\ell \in [m_1]}} = \eta_B \bb{\eigengap} \gamma_k
}
satisfies the bound of Lemma~\ref{lemma:constant_prob_error_bound_all}, with probability at least $1-\delta/d$. 

Taking a union bound over all $k \in [d]$ and rescaling equation~\ref{eq:median_sig_bound} by $\frac{1}{\eta_B \bb{\eigengap}}$, it follows that conditioned on $\mathcal{E}$, with probability at least $1-\delta$, 
\bas{
\Abs{\hat{\gamma}_k - \V_{kk}} &\le  8\bb{\frac{1}{\sqrt{m}} + \frac{2}{m}} \V_{kk} + O\bb{\frac{b_k^2 \log B}{\sqrt{mB}} \bb{\frac{\Mfour}{\eigengap}}^2 + \frac{B \log\bb{\frac{1}{\delta}}\log\bb{N/\log\bb{\frac{1}{\delta}}}}{N \log B} \bb{\frac{\Mtwo}{\eigengap}}^2}  \nonumber\\
&+ O\bb{\frac{b_k^2 m^2 \log^2 d \log^3 B}{B}  \bb{\frac{\Mtwo}{\eigengap}}^4 + \frac{\lambda_1 \Mtwo^2 \log B}{B \bb{\eigengap}^3}}.
}
Since the event $\mathcal{E}$ holds with probability $1-\delta$, the bound holds unconditionally with probability at least $1-2\delta$. 
The second result follows by Lemma~\ref{lemma:constant_prob_error_bound_some}.

\end{proof}

\begin{remark}\label{remark:prop2_higher_order}
The first term of the error of Lemma~\ref{lemma:high_prob_error_bound} is $O\bb{ \V_{kk}/\sqrt{m}}$, where $m = \log n$. We verify that for $N = n = m_1 m$ the other terms are smaller asymptotically in $n$. Since $m = \log n$ and $m_2 = 8 \log (20d/\delta)$ where $d = \text{poly}(n)$, we have
\bas{
B = \frac{n}{m m_1} = \Theta\bb{\frac{n}{\log n \log d}}.
}
Therefore, each summand with a $\sqrt{B}$ or $B$ in the denominator of the error of Lemma~\ref{lemma:high_prob_error_bound} is $\tilde{O}(1/\sqrt{n})$. It suffices to show that $\frac{1}{\sqrt{m}}$ asymptotically dominates $\frac{B \log N'}{N' \log B}$ where $N' = N/\log\bb{\frac{1}{\delta}}$. Note that $1 \le \log d \le 5 \log n$, $B = \tilde{\Theta}(n)$ and $\log B = \Theta(\log n)$. Therefore,
\bas{
\frac{B \log N'}{N' \log B} &= \frac{ \log (1/\delta) \log \bb{\frac{Bm_1 m}{\log (1/\delta)}}}{m_1 m \log B} = \frac{\log (1/\delta) \log (Bm_1) }{m_1 m \log B} + \frac{\log (1/\delta) \log \frac{m}{\log(1/\delta)}}{m_1 m \log B} \\
&= \Theta\bb{\frac{\log(1/\delta) \log n}{\log n \log (d/\delta) \log n} + \frac{\log \bb{\frac{m}{\log (1/\delta)}}}{\frac{m}{\log (1/\delta)} \log n}} = O\bb{\frac{1}{\log n}} = o\bb{\frac{1}{\sqrt{m}}}.
}
\end{remark}
\section{Entrywise Error Bounds}\label{appendix:entrywise_error_bounds}

\newcommand{\yij}{Y_{j,i}}

\begin{lemma}\label{lemma:oja_error_hajek_tail_bound}
Let the learning rate, $\eta_n$, be set according to Lemma~\ref{lemma:learning_rate_choice}. Further, for $X_i \sim \mathcal{P}, A_i = X_iX_i^{\top}$, let $\normop{A_i - \Sigma} \leq \Mone$ almost surely. Then, for $\delta \in \bb{0,1}$, with probability at least $1-\delta$, we have for all $k \in [d]$, 
\bas{
    \Abs{e_{k}^{\top} \Eone{n}} &\lesssim  \sqrt{\eta_{n}\bb{e_k^{\top}\vp R_0\vp^{\top}e_{k}}\log\bb{\frac{d}{\delta}}} +  \eta_{n}b_k\bb{\Mone\log\bb{\frac{d}{\delta}} + \Mtwo\sqrt{\frac{\lambda_{1}}{\eigengap}}\sqrt{\log\bb{\frac{d}{\delta}}}}
}
where $\Eone{n}$ is defined in Lemma~\ref{lemma:oja_error_decomposition}, $b_k := \norm{\vp^{\top}e_k}_{2}$,  $\widetilde{M} := \E\bbb{\vp^{\top}\bb{A_j-\Sigma}v_1v_1^{\top}\bb{A_j-\Sigma}^{\top}\vp}$ and $R_0 \in \R^{(d-1) \times (d-1)}$ with entires
\bas{
    R_0(k,l) := \frac{\widetilde{M}_{k\ell}}{2\lambda_1-\lambda_{k+1}-\lambda_{\ell+1}}, \;\; \forall k, l \in [d-1]
}
\end{lemma}
\begin{proof}
Using Lemma~\ref{lemma:hajek_decomposition}, we have 
\bas{
e_{k}^{\top}\Eone{n} = \eta_{n}e_{k}^{\top}Y_n = \sum_{j=1}^n \eta_{n}e_{k}^{\top}X_j^{n}, \text{ where  } X_j^{n} := \vp\Lambda_{\perp}^{n-j}\vp^{\top}(A_j-\Sigma)v_1
}
Let $\alpha_{j} := \eta_{n}e_{k}^{\top}X_j^{n}$. Then, note that $\E\bbb{\alpha_j} = 0$. Furthermore, 
\bas{
    \E\bbb{\alpha_j^{2}} &= \eta_{n}^{2}e_{k}^{\top}\vp\lambp^{n-j}\E\bbb{\vp^{\top}\bb{A_j-\Sigma}v_1v_1^{\top}\bb{A_j-\Sigma}^{\top}\vp}\lambp^{n-j}\vp^{\top}e_{k} = \eta_{n}^{2}e_{k}^{\top}\vp\lambp^{n-j}\widetilde{M}\lambp^{n-j}\vp^{\top} e_{k} =: \sigma_{jk}^{2}, \\
    \Abs{\alpha_j} &= \Abs{\eta_{n}e_k^{\top}\vp\Lambda_{\perp}^{n-j}\vp^{\top}(A_j-\Sigma)v_1} \leq \eta_{n}b_k\normop{\lambp^{n-j}}\Mone \leq \eta_{n}b_k\Mone
}
Therefore, using the fact that $\alpha_{j}$ are independent of each other, along with Bernstein's inequality, (see e.g. Proposition 2.14 and the subsequent discussion in \cite{wainwright2019high}), we have with probability at least $1-\delta$, 
\bas{
    \Abs{e_{k}^{\top}\Eone{n}} &\leq \sqrt{\bb{\sum_{j=1}^{n}\sigma_{jk}^{2}}\log\bb{\frac{1}{\delta}}} + \eta_{n}\Mone b_k\log\bb{\frac{1}{\delta}}
}
Furthermore, considering a union bound over $k \in [d]$, we have for all $k \in [d]$, 
\bas{
   \Abs{e_{k}^{\top}\Eone{n}} &\leq \sqrt{\bb{\sum_{j=1}^{n}\sigma_{jk}^{2}}\log\bb{\frac{d}{\delta}}} + \eta_{n}\Mone\log\bb{\frac{d}{\delta}} 
}
Finally, using Lemma~\ref{lemma:second_moment_matrix}, we have
\bas{
\sum_{j=1}^{n}\sigma_{jk}^{2} &= \eta_{n}^{2}e_{k}^{\top}\bb{\sum_{j=1}^{n}\vp\lambp^{n-j}\widetilde{M}\lambp^{n-j}}\vp^{\top} e_{k} \\
&= \eta_{n}^{2}e_{k}^{\top}\E\bbb{Y_{n}{Y_n}^{\top}} e_{k} \\
&= \eta_{n}^{2}e_{k}^{\top}\vp\bb{R^{(n)}} \vp^{\top}e_{k} \\
&= \eta_{n}^{2}e_{k}^{\top}\vp\bb{\frac{R_0}{\eta_n} + \bb{R^{(n)} - \frac{R_0}{\eta_n}}}\vp^{\top} e_{k} \\
&\leq \eta_{n}e_{k}^{\top}\vp R_0 \vp^{\top}e_{k} + \eta_{n}^{2}b_{k}^{2}\norm{R^{(n)} - \frac{R_0}{\eta_n}}_{F} \\
&\leq \eta_{n}e_{k}^{\top}\vp R_0 \vp^{\top}e_{k} + \frac{\eta_n^{2}b_{k}^{2}\lambda_1\Mtwo^{2}}{\bb{\eigengap}}
}
which completes our proof.



\end{proof}

\begin{lemma}\label{lemma:oja_error_decomposition_higher_order_tail_bounds}
    Let the learning rate, $\eta_n$, be set according to Lemma~\ref{lemma:learning_rate_choice}. Then, for $\delta \in \bb{0,1}$, with probability at least $1-\delta$, we have
    \bas{
         \norm{\Etwo{n} + \Ethree{n} + \Efour{n}}_{2} &\lesssim \frac{ \eta_{n}^2n \Mtwo^2 \log d} {\sqrt{\delta}} + \frac{\sqrt{s_n}\eta_n\sqrt{n}\Mtwo\log\bb{d}}{\sqrt{\delta}} \\
         &\quad\quad + \frac{\log\bb{\frac{1}{\delta}}}{\delta^{3}} \bb{\sqrt{d}\exp\bb{-\eta_{n}n\bb{\lambda_{1}-\lambda_{2}}} + \frac{\sqrt{\eta_{n}^{3}n}\Mtwo^{2}\log\bb{d}}{\sqrt{\lambda_{1}-\lambda_{2}}}} 
    }
    and for all $k \in [d]$,
    \bas{
         \Abs{e_{k}^{\top}\bb{ \Etwo{n} + \Ethree{n} + \Efour{n}}} &\leq b_{k}\norm{\Etwo{n} + \Ethree{n} + \Efour{n}}_{2} \\
         &\lesssim \frac{b_k \eta_{n}^2n \Mtwo^2 \log d} {\sqrt{\delta}} + \frac{b_k\sqrt{s_n}\eta_n\sqrt{n}\Mtwo\log\bb{d}}{\sqrt{\delta}} \\
         &\quad\quad + b_{k}\frac{\log\bb{\frac{1}{\delta}}}{\delta^3}\bb{\sqrt{d}\exp\bb{-\eta_{n}n\bb{\lambda_{1}-\lambda_{2}}} + \frac{\sqrt{\eta_{n}^3 n}\Mtwo^{2}\log\bb{d}}{\sqrt{\lambda_{1}-\lambda_{2}}}} 
    }
    where $\Etwo{n}, \Ethree{n}, \Efour{n}$ are as defined in Lemma~\ref{lemma:oja_error_decomposition}, $b_k := \norm{\vp^{\top}e_k}_{2}$ and $s_n := \frac{C\log\bb{\frac{1}{\delta}}}{\delta^{3}}\frac{\eta_n\Mtwo^{2}}{\bb{\eigengap}}$ for a universal constant $C > 0$.
\end{lemma}
\begin{proof}
    We have 
    \ba{
        \norm{\Etwo{n} + \Ethree{n} + \Efour{n}}_{2} \leq  \Abs{e_k^{\top}\Etwo{n}} + \Abs{e_k^{\top}\Ethree{n}} + \Abs{e_k^{\top}\Efour{n}} \label{eq:E2_E3_E4_triangle_ineq}
    }
    Using Lemma~\ref{lemma:en2_norm}, we have for all $k \in [d]$, with probability at least $1-\frac{\delta}{3}$, 
    \ba{
        \norm{\Etwo{n}} \leq \frac{12\eta_{n}^2 \mathcal{M}_2^2 n \log d} {\sqrt{\delta/3}} \leq  \frac{21\eta_{n}^2 \mathcal{M}_2^2 n \log d} {\sqrt{\delta}}. \label{eq:residual_en2_bound}
    }
    Using Lemma~\ref{lemma:en3_norm} , along with the definition of $\eta_n$ in Lemma~\ref{lemma:learning_rate_choice}, with probability at least $1-\frac{\delta}{3}$,
    \ba{
    \norm{\Ethree{n}}_{2} &\lesssim \frac{\sqrt{s_n}\sqrt{\log\bb{\frac{1}{\delta}}}}{\delta^{\frac{3}{2}}}  \bb{\sqrt{d}\exp\bb{-\eta_{n}n\bb{\lambda_{1}-\lambda_{2}}} + \frac{\sqrt{\eta_{n}}\Mtwo}{\sqrt{\lambda_{1}-\lambda_{2}}}} + \sqrt{s_n}\frac{\eta_n\sqrt{n}\Mtwo\log\bb{d}}{\sqrt{\delta}} \notag \\
    &\lesssim \frac{\sqrt{\log\bb{\frac{1}{\delta}}}}{\delta^{\frac{3}{2}}}  \bb{\sqrt{d}\exp\bb{-\eta_{n}n\bb{\lambda_{1}-\lambda_{2}}} + \sqrt{\frac{C \log (1/\delta)}{\delta^3} \frac{\eta_{n} \Mtwo^2}{\eigengap}} \cdot \frac{\sqrt{\eta_{n}}\Mtwo \log d}{\sqrt{\lambda_{1}-\lambda_{2}}}} \notag \\
    &\lesssim \frac{\log\bb{\frac{1}{\delta}}}{\delta^3} 
    \bb{\sqrt{d}\exp\bb{-\eta_{n}n\bb{\lambda_{1}-\lambda_{2}}} + \frac{\sqrt{\eta_{n}^3 n}\Mtwo^{2}\log\bb{d}}{\sqrt{\lambda_{1}-\lambda_{2}}}},
    \label{eq:residual_en3_bound}
    }
    where the second inequality used $s_n \le 1$. 
    Using Lemma~\ref{lemma:en4_norm}, along with the definition of $\eta_n$ in Lemma~\ref{lemma:learning_rate_choice}, with probability at least $1-\frac{\delta}{3}$,
    \ba{
    \norm{\Efour{n}}_{2} \lesssim \frac{1}{\delta^{\frac{3}{2}}}  \bb{\sqrt{d}\exp\bb{-\eta_{n}n\bb{\lambda_{1}-\lambda_{2}}} + \frac{\sqrt{\eta_{n}^{3}n}\Mtwo^{2}\log\bb{d}}{\sqrt{\lambda_{1}-\lambda_{2}}}} \label{eq:residual_en4_bound}
    }
    The first result follows by a union bound over \eqref{eq:residual_en2_bound}, \eqref{eq:residual_en3_bound}, \eqref{eq:residual_en4_bound} and substituting in \eqref{eq:E2_E3_E4_triangle_ineq}.
    Finally, note that using Lemma~\ref{lemma:oja_error_decomposition}, $\exists x_{n}, y_{n}, z_{n} \in \R^{d-1}$ such that  of $\Etwo{n} = \vp\vp^{\top}x_{n}$, $\Ethree{n} = \vp\vp^{\top}x_{n}$, $\Efour{n} = \vp\vp^{\top}x_{n}$. Therefore,  
    \bas{
        \Abs{e_{k}^{\top}\bb{\Etwo{n} + \Ethree{n} + \Efour{n}}} &= \Abs{e_{k}^{\top}\vp\vp^{\top}\bb{x_n + y_n + z_n}} \\
        &= \Abs{e_{k}^{\top}\vp\vp^{\top}\vp\vp^{\top}\bb{x_n + y_n + z_n}} \\
        &\leq \norm{e_{k}^{\top}\vp\vp^{\top}}_{2}\norm{\vp\vp^{\top}\bb{x_n + y_n + z_n}}_{2} \\
        &= b_k \norm{\Etwo{n} + \Ethree{n} + \Efour{n}}_{2}
    }
    which completes the proof of the second result.
\end{proof}
Now we are ready to prove a detailed version of Theorem~\ref{thm:main:entrywise_concentration_bound}.
\begin{lemma}\label{lemma:entrywise_concentration_bound} Let the learning rate, $\eta_n$, be set according to Lemma~\ref{lemma:learning_rate_choice}. Further, for $X_i \sim \mathcal{P}, A_i = X_iX_i^{\top}$, let $\normop{A_i - \Sigma} \leq \Mone$ almost surely. Define $\roja := \voja - \bb{v_1^{\top}\voja}v_1$.  Then, with probability at least $1-\delta$, for all $k \in [d]$, 
\bas{
    \Abs{e_k^{\top}\roja} &\lesssim \sqrt{\eta_{n}\bb{e_k^{\top}\vp R_0\vp^{\top}e_{k}}\log\bb{\frac{d}{\delta}}} +  \eta_{n}b_k\bb{\Mone\log\bb{\frac{d}{\delta}} + \Mtwo\sqrt{\frac{\lambda_{1}}{\eigengap}}\sqrt{\log\bb{\frac{d}{\delta}}}} \\
    &\quad\quad + b_{k}\frac{\log\bb{\frac{1}{\delta}}}{\delta^3}\bb{\sqrt{d}\exp\bb{-\eta_{n}n\bb{\lambda_{1}-\lambda_{2}}} + \frac{\sqrt{\eta_{n}^{3}n}\Mtwo^{2}\log\bb{d}}{\sqrt{\lambda_{1}-\lambda_{2}}}} \\
    &\quad\quad + \frac{b_k \eta_{n}^2n \Mtwo^2 \log d} {\sqrt{\delta}} + \frac{b_k\sqrt{s_n}\eta_n\sqrt{n}\Mtwo\log\bb{d}}{\sqrt{\delta}} 
}
where $b_k := \norm{\vp^{\top}e_k}_{2}$, $s_n := \frac{C\log\bb{\frac{1}{\delta}}}{\delta^{3}}\frac{\eta_n\Mtwo^{2}}{\bb{\eigengap}}$,  $\widetilde{M} := \E\bbb{\vp^{\top}\bb{A_j-\Sigma}v_1v_1^{\top}\bb{A_j-\Sigma}^{\top}\vp}$ and $R_0 \in \R^{(d-1) \times (d-1)}$ with entires
    \bas{
    R_0(k,l) = \frac{\widetilde{M}_{k\ell}}{2\lambda_1-\lambda_{k+1}-\lambda_{\ell+1}}, k, l \in [d-1]
    }.
\end{lemma}
\begin{proof}
    Using Lemma~\ref{lemma:oja_error_decomposition}, we have
    \bas{
        e_{k}^{\top}\roja := e_{k}^{\top}\Eone{n} + e_{k}^{\top}\Etwo{n} + e_{k}^{\top}\Ethree{n} + e_{k}^{\top}\Efour{n}
    }
    Therefore, 
    \bas{
        \Abs{e_{k}^{\top}\roja} &\leq \Abs{e_{k}^{\top}\Eone{n}} + \Abs{e_{k}^{\top}\Etwo{n} + e_{k}^{\top}\Ethree{n} + e_{k}^{\top}\Efour{n}}
    }
    The result then following by a union bound over the events defined in Lemma~\ref{lemma:oja_error_hajek_tail_bound} and Lemma~\ref{lemma:oja_error_decomposition_higher_order_tail_bounds}.
\end{proof}

\section{Central Limit Theorem for entries of the Oja vector}\label{appendix:entrywise_clt}
We consider the following setup from~\cite{ChernoCLT2015}.
Let $\mathcal{A}^{\text{re}}$ denote the class of all hyperrectangles in $\mathbb{R}^p$. That is, $\mathcal{A}^{\text{re}}$ consists of all sets $A$ of the form:
\begin{equation}
    A = \{w \in \mathbb{R}^p : a_j \leq w_j \leq b_j \text{ for all } j = 1, \dots, p\}
\end{equation}
for some real values $a_j$ and $b_j$ satisfying $-\infty \leq a_j \leq b_j \leq \infty$ for each $j = 1, \dots, p$. 

Consider 
\[
S_n^X  =\frac{1}{\sqrt{n}} \sum_{i=1}^{n} X_i.
\]
where $X_i,i\in [n]\in \mathbb{R}^p$ are independent random vectors with $\E[X_{ij}]=0$ and $\E[X_{ij}^2]<\infty$, for $i\in [n], j\in [p]$.
Consider the following  Gaussian approximation to $S_n^X$.
Define the normalized sum for the Gaussian random vectors:
\[
S_n^Y =\frac{1}{\sqrt{n}} \sum_{i=1}^{n} Y_i,
\]
where $Y_1, \dots, Y_n$ be independent mean zero Gaussian random vectors in $\mathbb{R}^p$ such that each $Y_i$ has the same covariance matrix as $X_i$. We are interested in bounding the quantity 
\bas{\rho_{n}\bb{\mathcal{\mathcal{A}^{\text{re}}}} := \sup_{A \in \mathcal{A}^{\text{re}}}\Abs{\Prob\bb{S_{n}^{X} \in A} - \Prob\bb{S_{n}^{Y} \in A}}}
Let $C_n \geq 1$ be a sequence of constants possibly growing to infinity as $n \rightarrow \infty$, and let $b, q > 0$ be some constants. Assume that $X_{i}$ satisfy,
   

\begin{itemize}
    \item[(M.1)] \( n^{-1} \sum_{i=1}^{n} \mathbb{E}[X_{ij}^2] \geq b \) \textit{for all} \( j = 1, \dots, p \),
    \item[(M.2)] \( n^{-1} \sum_{i=1}^{n} \mathbb{E}[|X_{ij}|^{2+k}] \leq C_n^k \) \textit{for all} \( j = 1, \dots, p \) \textit{and} \( k = 1,2 \).
\end{itemize}

Further, the authors consider examples where one of the following conditions also holds:

\begin{itemize}
    \item[(E.1)] \( \mathbb{E}[\exp(|X_{ij}| / C_n)] \leq 2 \) \textit{for all} \( i = 1, \dots, n \) \textit{and} \( j = 1, \dots, p \),
    \item[(E.2)] \( \mathbb{E}[(\max_{1 \leq j \leq p} |X_{ij}| / C_n)^q] \leq 2 \) \textit{for all} \( i = 1, \dots, n \).
\end{itemize}

Let
\[
D_n^{(1)} = \left( \frac{C_n^2 \log^7 (pn)}{n} \right)^{1/6}, \quad
D_{n,q}^{(2)} = \left( \frac{C_n^2 \log^3 (pn)}{n^{1 - 2/q}} \right)^{1/3}.
\]

Now we present Proposition 2.1 \citep{ChernoCLT2015}.

\begin{theorem}[\label{thm:prop2pt1}Proposition 2.1~\citep{ChernoCLT2015}]
    \textit{Suppose that conditions (M.1) and (M.2) are satisfied. Then under (E.1), we have}
\[
\rho_n(\mathcal{A}^{\text{re}}) \leq C D_n^{(1)},
\]
\textit{where the constant} \( C \) \textit{depends only on} \( b \); \textit{while under (E.2), we have}
\[
\rho_n(\mathcal{A}^{\text{re}}) \leq C \{ D_n^{(1)} + D_{n,q}^{(2)} \},
\]
\textit{where the constant} \( C \) \textit{depends only on} \( b \) \textit{and} \( q \).
\end{theorem}

Next, we will need the following result cited by~\citet{chernozhukov2017detailed}.
\begin{theorem}[\label{thm:Nazarov}Nazarov's inequality~\citep{nazarov2003maximal}, Theorem~1 in~\citep{ChernoCLT2015}]
    
Let \( Y = (Y_1, \dots, Y_p)^T \) be a centered Gaussian random vector in \( \mathbb{R}^p \) such that 
\[
\mathbb{E}[Y_j^2] \geq \sigma^2, \quad \text{for all } j = 1, \dots, p,
\]
for some constant \( \sigma > 0 \). Then, for every \( y \in \mathbb{R}^p \) and \( \delta > 0 \),
\[
\Prob(Y \leq y + \delta) - \Prob(Y \leq y) \leq \frac{\delta}{\sigma} (\sqrt{2 \log p} + 2).
\]
Here, for vector $y \in \R^{p}$, $y + \delta$ denotes the vector constructed by adding $\delta$ to each entry of $y$.
\end{theorem}

Now we are ready to state our main result in Proposition~\ref{prop:clt_appendix},

\begin{proposition}[\label{prop:clt_appendix}CLT for a suitable subset of entries]
Suppose the learning rate $\eta_n$, set according to Lemma~\ref{lemma:learning_rate_choice}, satisfies $\frac{\Mtwo^{2} \lambda_1 \eta_n}{\bb{\eigengap}^2}\leq \frac{C_0 b}{2}$ for some $b>0$ and a small universal constant $C_0$.  Let $\{X_i\}_{i=1}^n\in \mathbb{R}^d$ be i.i.d. mean-zero random vectors with covariance matrix \( \Sigma \) such that for all vectors \( v \in \mathbb{R}^d \), we have
\[
\mathbb{E} \left[ \exp \left( v^T X_1 \right) \right] \leq \exp \left( \frac{\sigma^2 v^T \Sigma v}{2} \right).
\]

Let $\roja := \voja - (v_{1}^{\top}\voja)v_{1}$. Consider the set $J := \{j:  \V_{jj}\geq  b\}$, and let $p :=|J|$.
Let $H_i := \frac{\sign(u_0^Tv_1)}{1+\eta_n\lambda_1} \vp\lambp^{n-i}\vp^{\top}\bb{X_{i}X_{i}^{\top}-\Sigma}v_{1}$. Let $Y_i\in \mathbb{R}^p$ be independent mean zero normal vectors such that $$\E[Y_iY_i^T]=\frac{n\eta_n}{\eigengap}\E[H_i[J]H_i[J]^T].$$ 
    Then, {\small
    \bas{
    &\sup_{A\in \mathcal{A}_{re}}\left|P\bb{\frac{\roja[J]}{\sqrt{\bb{\eigengap}\eta_n}}\in A}-P\bb{\frac{\sum_i Y_i}{\sqrt{n}}\in A}\right|= 
   \tilde{O}\bb{ \max\bb{\bb{\frac{\Mfour}{\eigengap }}^{1/3}n^{-1/6}, \bb{\frac{\Mtwo}{\eigengap}}^{1/2}n^{-1/8}}},
    }
    }
    where $\tilde{O}$ hides logarithmic factors in $n$, $p$, and constants depending on $b$.
\end{proposition}
\begin{proof}[Proof of Proposition~\ref{prop:clt_appendix}]
Consider the error decomposition of the Oja vector in Lemma~\ref{lemma:oja_error_decomposition}. 
   We have $\roja = \Psi_{n,1} + \Psi_{n,2} + \Psi_{n,3} + \Psi_{n,4}$, where $\Psi_{n,1}, \Psi_{n,2}, \Psi_{n,3}, \Psi_{n,4}$ are defined in Equation~\eqref{eq:ojadecomp}.
  Let $\rem:=\Psi_{n,2} + \Psi_{n,3} + \Psi_{n,4}$.
   
   For any $\delta \in \bb{0,1}$, $\exists \epsilon > 0$ such that from Lemma~\ref{lemma:oja_error_decomposition_higher_order_tail_bounds} we have,  
   \bas{
   \Prob((\eta_n\bb{\eigengap})^{-1/2}\|\rem\|_{2}\geq \epsilon)\leq \delta
   }
   we will specify $\epsilon$ as needed in the proof.
  


For all $i \in [n]$, let 
\ba{
U_i:=\underbrace{\sqrt{n\eta_n/\bb{\eigengap}}}_{c_n} H_i
}
 We show that $U_1, U_2, \dots, U_n$ satisfy conditions (M.1) and (M.2) with suitable constants. 

For (M.1), using equation \eqref{eq:ojadecomp}, 

\ba{ \sum_{i=1}^{n}H_{i} = \Psi_{n,1}.
\label{eq:psin1}}

By Lemma~\ref{lemma:second_moment_matrix} (equation \eqref{eq:variance_diff_bound}), there exists a universal constant $C_0$ such that
\bas{
\Abs{e_j^{\top} \bb{ \frac{\eta_n}{\eigengap} \sum_{i=1}^{n} \E\bbb{H_iH_i^{\top}} - \V} e_j} \leq \frac{\eta_n \lambda_1 \Mtwo^2}{C_0 \bb{\eigengap}^2} \le \frac{b}{2} \le \frac{\V_{jj}}{2}.
}
for all $j \in J$, where the last two inequalities follow by assumption and definition of $J$. This implies for all $j\in J$,
\bas{
\frac{\eta_n}{\eigengap} \sum_{i}^{n} \E\bbb{H_{ij}^2} \ge \V_{jj}/2 \ge b/2 \iff \frac{1}{n} \sum_i \E[U_{ij}^2] \geq \V_{jj}/2\geq b/2
}
 
To show (M.2), by Lyapunov's inequality and Assumption~\ref{assumption:bounded_moments}:
    \bas{
    \E\bbb{\|U_{ij}^{2+k}\|_{2}} = \E \bbb{c_n^{2+k}|H_{ij}|^{2+k}}&\leq 2 (c_n\Mfour)^{2+k}
    }
    for $k \in \{1,2\}$, where $C_n:=2c_n\Mfour$.
    
    We now check condition E.1. Now note that for any unit vector $u \in \R^{d}$, $u^T H_i$ is subexponential with parameter  $\sigma^2\lambda_1$  (Proposition 2.7.1. of~\citep{vershynin2018high}). Hence, there exists a constant $C > 0$ such that 
    \bas{
    \E [\exp(|H_{ij}|/C\lambda_1 \sigma^2)] \leq 2
    }
    Therefore, 
    \bas{
    \E [\exp(|U_{ij}|/C\lambda_1 c_n\sigma^2)]\leq 2.
    }
    
    Now we set $C_n :=\max(2c_n\Mfour,C\lambda_1 c_n\sigma^2)$.
    
    Using Eq~\ref{eq:psin1},
    \bas{\frac{1}{\sqrt{(\eigengap)\eta_n}}\Psi_{n,1}[J]=\sqrt{\eta_n/\bb{\eigengap}}\sum_i H_i[J]=\frac{1}{\sqrt{n}}\sum_iU_i[J],}
     the random variables $U_{i}[J], i \in [n]$ satisfy conditions (M.1), (M.2) and (E.1). By Theorem~\ref{thm:prop2pt1},
    \bas{
    \rho(\mathcal{A}^{\text{re}})\leq C \left( \frac{C_n^2 \log^7 (pn)}{n} \right)^{1/6}
    }

Recall from the statement of the proposition that $Y_i,i\in [n]$ are mean zero independent Gaussian vectors in $ \mathbb{R}^{p}$ with the same covariance structure as $U_i[J]$, i.e, $\E\bbb{Y_{i}Y_{i}^{\top}} = \E\bbb{U_i[J]U_i[J]^{\top}}$.

Let $S_W$ be the random variable $\sum_i W_i$ for any collection $W$ of $n$ random variables $W_1, W_2, \dots, W_n$. Consider the vector $S_{W}[J]$ to be the projection of $W$ on the set $J$, defined as $e_{i}^{\top}S_{W}[J] = e_{i}^{\top}S_{W}$ for $i \in J$.

Recall that 
\bas{
e_i^T\roja:= e_{i}^{\top}\bb{\sum_{j=1}^n\eta_{n}H_j + \rem}.
}
Let $A :=\{u\in \mathbb{R}^{p}|u_i\in [a_i,b_i], i\in J\}$. Let $A_{\epsilon}^+ :=\{X|X_i\in [a_i-\epsilon,b_i+\epsilon], i\in [p]\}$ and $A_{\epsilon}^- := \{X|X_i\in [a_i+\epsilon,b_i-\epsilon], i\in J\}$. 

Let $S_R[J] := \sum_{i \in J}e_i^T\roja$. Then, we have $S_R[J]=\eta_n  S_H[J]+\rem[J]$. 

We will use the following identity for vectors $G_1,G_2\in \mathbb{R}^p$.  
\bas{
\Prob(G_1\in A_{\epsilon}^-, \|G_2\|\leq \epsilon)\leq \Prob(G_1+G_2\in A,\|G_2\|\leq \epsilon)\leq P(G_1\in A_{\epsilon}^+, \|G_2\|\leq \epsilon)
}

So, 
\bas{
\Prob(G_1+G_2\in A)&\leq \Prob(G_1\in A_{\epsilon}^+, \|V\|\leq \epsilon)+P(\|V\|\geq \epsilon)\\
\Prob(G_1+G_2\in A)&\geq P(G_1\in A_{\epsilon}^-, \|G_2\|\leq \epsilon)
}

Using $G_1=S_U[J]/\sqrt{n}$ and $G_2=(\eta_n \bb{\eigengap})^{-1/2}\rem$, we have:
\bas{
&\Prob((\bb{\eigengap}\eta_n)^{-1/2}\roja[J]\in A)-\Prob(S_Y/\sqrt{n}\in A)\\
   &\leq \Prob((\bb{\eigengap}\eta_n)^{-1/2}\roja[J] \in A, (\eta_n\bb{\eigengap})^{-1/2}\|\rem\|\leq \epsilon)+\Prob((\eta_n \bb{\eigengap})^{-1/2}\|\rem\|_{2}\geq \epsilon) \\
   & \;\;\;\; -\Prob(S_Y/\sqrt{n}\in A)\\
   &\leq \Prob(S_U[J]/\sqrt{n}\in A_\epsilon^+)+\Prob((\eta_n \bb{\eigengap})^{-1/2}\|\rem\|_{2}\geq \epsilon)-\Prob(S_Y/\sqrt{n}\in A) =: \gamma_A.
}
Note that $\gamma_A$ can be bounded as
\bas{
   \gamma_A &\leq |\Prob(S_U[J]/\sqrt{n}\in A_\epsilon^+)-\Prob(S_Y/\sqrt{n}\in A_\epsilon^+)| \\
   &+|\Prob(S_Y/\sqrt{n}\in A_\epsilon^+)-\Prob(S_Y/\sqrt{n}\in A)|+\Prob((\eta_n \bb{\eigengap})^{-1/2}\|\rem\|\geq \epsilon).
    }
Similarly,
\bas{
   &\Prob((\bb{\eigengap} \eta_n)^{-1/2}\roja[J]\in A)-\Prob(S_Y/\sqrt{n}\in A) \ge \omega_{A},
   }
   where
   \bas{
   \omega_{A} &:= \Prob(S_U[J]/\sqrt{n}\in A_\epsilon^-, (\eta_n \bb{\eigengap})^{-1/2}\|\rem\|\geq \epsilon)-\Prob(S_Y/\sqrt{n}\in A)\\
   &\geq \Prob(S_U[J]/\sqrt{n}\in A_\epsilon^-)-\Prob((\eta_n \bb{\eigengap})^{-1/2}\|\rem\|\geq \epsilon)\\
   &-\Prob(S_Y/\sqrt{n}\in A_\epsilon^-)+\Prob(S_Y/\sqrt{n}\in A_\epsilon^-)-\Prob(S_Y/\sqrt{n}\in A).
}
Therefore, we have by Theorem~\ref{thm:prop2pt1} that for some constant $C'$ that depends only on $b$,
    \ba{
    \sup_{A\in \mathcal{A}_{re}}|\gamma_{A}|\leq  C' \left( \frac{C_n^2 \log^7 (pn)}{n} \right)^{1/6}+\Abs{\Prob(S_Y/\sqrt{n}\in A_\epsilon^+)-\Prob(S_Y/\sqrt{n}\in A)}+\delta.\label{eq:gamma}
    }
    
    Similarly, 
    \ba{\label{eq:omega}
    \sup_{A\in \mathcal{A}_{re}}|\omega_{A}|\leq  C' \left( \frac{C_n^2 \log^7 (pn)}{n} \right)^{1/6}+\Abs{\Prob(S_Y/\sqrt{n}\in A_\epsilon^-)-\Prob(S_Y/\sqrt{n}\in A)}+\delta.
    }
    For $\Prob(S_Y/\sqrt{n}\in A_\epsilon^+)-\Prob(S_Y/\sqrt{n}\in A)$, we will use Nazarov's inequality (Lemma~\ref{thm:Nazarov}):
    \ba{\label{eq:aplusnazarov}
   \Abs{\Prob(S_Y/\sqrt{n}\in A_\epsilon^+)-\Prob(S_Y/\sqrt{n}\in A)}\leq \frac{\sqrt{2}\epsilon}{b^{1/2}}(\sqrt{2\log p}+2).
    }

For bounding the terms concerning $A_{\epsilon}^-$, we need to be a little careful because if $b_i-a_i \leq  2\epsilon$, then $A_{\epsilon}^-$ has measure zero under the Gaussian distribution. If $A_{\epsilon}^-$ is nonempty, then we have the same bound as Eq~\ref{eq:aplusnazarov}.
However, in case that is not true, note that there must be some $i\in [p]$ such that $b_i-a_i \leq 2\epsilon$.
Hence
\ba{\label{eq:aminusnazarovempty}
   \Abs{\Prob(S_Y/\sqrt{n}\in A_\epsilon^-)-\Prob(S_Y/\sqrt{n}\in A)}&=\Prob(S_Y/\sqrt{n}\in A)\notag\\
   \rd &=\Prob(S_Y[i]/\sqrt{n}\in [a_i,b_i])\notag \\
   &\leq \frac{2\epsilon}{\sqrt{\pi}b^{1/2}}.\bk
    }
    So overall,
   \ba{\label{eq:aminusnazarov}
   \Abs{\Prob(S_Y/\sqrt{n}\in A_\epsilon^-)-\Prob(S_Y/\sqrt{n}\in A)}&=\Prob(S_Y/\sqrt{n}\in A)\notag\\
   \rd &=\Prob(S_Y[i]/\sqrt{n}\in [a_i,b_i])\notag \\
   &\leq \max\bb{\frac{2\epsilon}{\sqrt{\pi}b^{1/2}},\frac{\sqrt{2}\epsilon}{b^{1/2}}(\sqrt{2\log p}+2).\bk
    }
    }
    Putting Eqs~\ref{eq:gamma},~\ref{eq:omega}, ~\ref{eq:aplusnazarov} and~\ref{eq:aminusnazarov} together, we have, for some absolute constant $C_1$:
    \ba{
    &\sup_{A\in \mathcal{A}_{re}}|\Prob((\bb{\eigengap}\eta_n)^{-1/2}\roja[J]\in A)-\Prob(n^{-1/2}S_Y\in A)|\leq \max(\sup_{A\in \mathcal{A}_{re}}|\gamma_{A}|,\sup_{A\in \mathcal{A}_{re}}|\omega_{A}|) \notag \\
    &\lesssim \left(\frac{C_n^2 \log^7 (pn)}{n} \right)^{1/6}+\frac{C_1\epsilon}{b^{1/2}}\sqrt{\log p}+\delta.  \label{eq:clt_bound_1}
    }
We invoke Lemma A.2.3 in~\cite{kumarsarkar2024markovoja} to see that:  $\Mfour \leq \lambda_1+\sigma^2\tr{\Sigma}$. Therefore, for some constant $C'' > 0$,
\bas{
C_n=\max(2c_n\Mfour,C\lambda_{1}c_n\sigma^2) \leq C''\sqrt{\frac{n\eta_n}{\bb{\eigengap}}}\Mfour.
}
From Lemma~\ref{lemma:oja_error_decomposition_higher_order_tail_bounds} and the assumption on the learning rate (Lemma~\ref{lemma:learning_rate_choice}),
\ba{
\sqrt{\eta_{n}\bb{\eigengap}}\epsilon \lesssim \frac{ \eta_{n}^2n \Mtwo^2 \log d} {\sqrt{\delta}} + \frac{\sqrt{s_n}\eta_n\sqrt{n}\Mtwo\log\bb{d}}{\sqrt{\delta}}
         + \frac{\log\bb{\frac{1}{\delta}}}{\delta^{3}}\bb{ \frac{\sqrt{\eta_{n}^{3}n}\Mtwo^{2}\log\bb{d}}{\sqrt{\lambda_{1}-\lambda_{2}}}} \label{eq:Sv_tail_bound}.  \bk
}


    Substituting the bound on $\eps$ from equation~\eqref{eq:Sv_tail_bound} into equation~\eqref{eq:clt_bound_1} and optimizing over $\delta$ yields
     \ba{
\delta
=\tilde{O}\bb{\bb{\frac{\log p}{b}}^{1/8} \sqrt{\frac{\Mtwo}{\eigengap}}n^{-1/8}}. \label{eq:delta_bound}
    }
    Substituting the choice of $\delta$ from equation~\eqref{eq:delta_bound} in~\eqref{eq:clt_bound_1}, we conclude
    \bas{
    & \sup_{A\in \mathcal{A}^{\text{re}}}|\Prob((\bb{\eigengap}\eta_n)^{-1/2}\roja[J]\in A)-\Prob(n^{-1/2}S_Y\in A)| \\
    & \quad\quad = \tilde{O}\bb{ \max\bb{\bb{\frac{\Mfour}{\eigengap }}^{1/3}n^{-1/6}, \bb{\frac{\Mtwo}{\eigengap}}^{1/2}n^{-1/8}}}.
    }
\end{proof}

\end{appendix}
\end{document}